%% file: ma_cpomdp_ieee.tex
\documentclass[journal,final]{IEEEtran}
\usepackage[noadjust]{cite}
\usepackage{amsmath,amssymb,amsfonts}
\usepackage{graphicx}
\usepackage{textcomp}
\usepackage{xcolor}
\usepackage{hyperref}
\usepackage{amsmath}
\usepackage{accents}
\makeatletter
\newcommand{\sbullet}{%
\hbox{\fontfamily{lmr}\fontsize{.4\dimexpr(\f@size pt)}{0}\selectfont\textbullet}}

\usepackage[linesnumbered,ruled,vlined,algo2e]{algorithm2e}

\SetCommentSty{mycommfont}
\let\oldnl\nl
\newcommand{\nonl}{\renewcommand{\nl}{\let\nl\oldnl}}
\SetKwInput{KwInput}{Input}                
\SetKwInput{KwOutput}{Output}              
\SetKwInOut{Parameter}{Parameter}
\SetKwRepeat{Do}{do}{while}%

\usepackage{array}
\usepackage{dblfloatfix}
\usepackage{url}
\usepackage{tikz}
\usetikzlibrary{positioning}
\usepackage{cuted} 

\def\BibTeX{{\rm B\kern-.05em{\sc i\kern-.025em b}\kern-.08em
		T\kern-.1667em\lower.7ex\hbox{E}\kern-.125emX}}

\usepackage{enumerate}
\usepackage{array}
\usepackage[utf8]{inputenc}
\usepackage{amsmath}
\usepackage{amsthm, amssymb}
\usepackage{nccmath}
\usepackage{mathtools}
\usepackage{bbm}
\usepackage[makeroom]{cancel}
\usepackage{amssymb}
\graphicspath{{./figures/}}
\usepackage{fixltx2e}
\usepackage{soul}
\usepackage{xcolor}
\usepackage{ upgreek }
\usepackage{dblfloatfix}
\usepackage[colorinlistoftodos]{todonotes}
\usepackage{listings}
\usepackage{color}
\usepackage{cuted}
\usepackage{url}
\usepackage{hyperref}

\input{Template/MATH_template.tex}

\input{Template/notation.tex}
\usepackage{float,layouts}
\usepackage[justification=centering]{caption}
\usepackage[T1]{fontenc}

\usepackage{cite}
\usepackage{hyperref}
\usepackage{amsmath}
\usepackage{algorithm,algorithmic}

\usepackage{array}

\ifCLASSOPTIONcompsoc
  \usepackage[caption=false,font=normalsize,labelfont=sf,textfont=sf]{subfig}
\else
  \usepackage[caption=false,font=footnotesize]{subfig}
\fi
\usepackage{dblfloatfix}
\usepackage{url}

\usepackage{color-edits}
\addauthor{vj}{red}

\allowdisplaybreaks

\begin{document}
%
\title{Cooperative Multi-Agent Constrained POMDPs: Strong Duality and Primal-Dual Reinforcement Learning with Approximate Information States
}
%
%
%
%
\author{Nouman~Khan,~\IEEEmembership{Member,~IEEE,}
       and~Vijay~Subramanian,~\IEEEmembership{Senior Member,~IEEE}
}

\maketitle
\begin{abstract}
We study the problem of decentralized constrained POMDPs in a team-setting where the multiple non-strategic agents have asymmetric information. Strong duality is established for the setting of infinite-horizon expected total discounted costs when the observations lie in a countable space, the actions are chosen from a finite space, and the immediate cost functions are bounded. Following this, connections with the common-information and approximate information-state approaches are established. The approximate information-states are characterized independent of the Lagrange-multipliers vector so that adaptations of the multiplier (during learning) will not necessitate new representations. Finally, a primal-dual multi-agent reinforcement learning (MARL) framework based on centralized training distributed execution (CTDE) and three time-scale stochastic approximation is developed with the aid of recurrent and feed-forward neural-networks as function-approximators.
\end{abstract}


%
\IEEEpeerreviewmaketitle

\input{Sections/introduction.tex}
\input{Sections/problem.tex} 
\input{Sections/strong_duality.tex}
\input{Sections/history_embedding.tex}
\input{Sections/marl.tex}
\input{Sections/conclusion.tex}
\input{Sections/appendices.tex}
\section*{Acknowledgment}
This work was funded in part, by NSF via grants ECCS2038416, EPCN1608361, EARS1516075, CNS1955777, and CCF2008130 for V. Subramanian, and grants EARS1516075, CNS1955777, and CCF2008130 for N. Khan. The authors would also like to thank Hsu Kao for helpful discussions. 

\ifCLASSOPTIONcaptionsoff
  \newpage
\fi



\bibliographystyle{ieeetr}
\bibliography{ma_cpomdp_ieee}
%



%

\begin{IEEEbiography}
[{\includegraphics[width=1in, height=1.25in, clip, keepaspectratio]{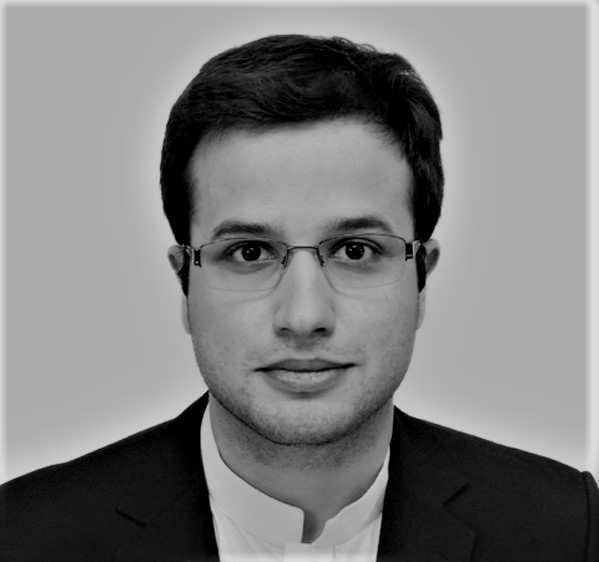}}]{Nouman Khan} (Member, IEEE) is a Ph.D candidate in the department of Electrical Engineering and Computer Science (EECS) at the University of Michigan, Ann Arbor, MI, USA. He received the B.S. degree in Electronic Engineering from the GIK Institute of Engineering Sciences and Technology, Topi, KPK, Pakistan, in 2014 and the M.S. degree in Electrical and Computer Engineering from the University of Michigan, Ann Arbor, MI, USA in 2019. His research interests include stochastic systems and their analysis and control.
\end{IEEEbiography}

\begin{IEEEbiography}[{\includegraphics[width=1in, height=1.25in, clip, keepaspectratio]{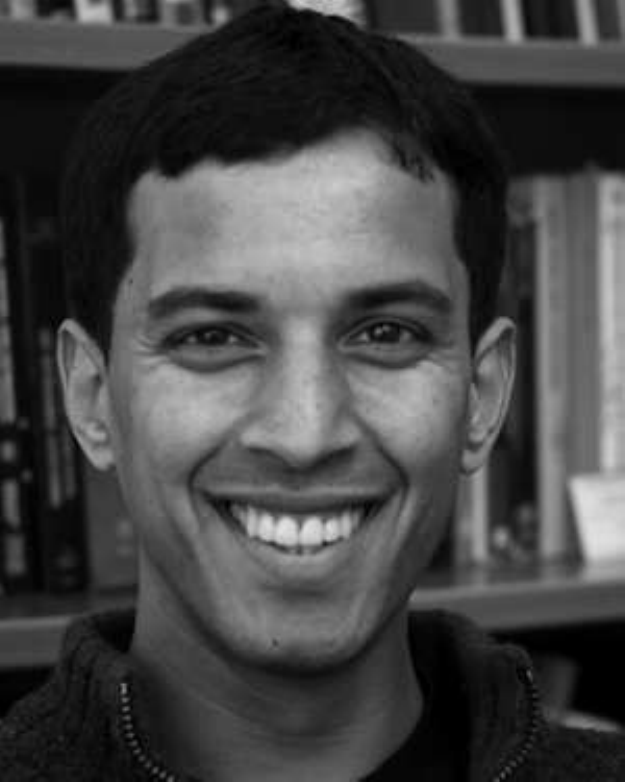}}]{Vijay Subramanian} (Senior Member, IEEE) received the Ph.D. degree in electrical engineering from the University of Illinois at Urbana-Champaign, Champaign, IL, USA, in 1999. He was a Researcher with Motorola Inc., and also with Hamilton Institute, Maynooth, Ireland, for a few years following which he was a Research Faculty with the Electrical Engineering and Computer Science (EECS) Department, Northwestern University, Evanston, IL, USA. In 2014, he joined the University of Michigan, Ann Arbor, MI, USA, where he is currently an Associate Professor with the EECS Department. His research interests are in stochastic analysis, random graphs, game theory, and mechanism design with applications to social, as well as economic and technological networks. 
\end{IEEEbiography}
\end{document}

%% file: Template/MATH_template.tex
\usepackage{mathrsfs}
\usepackage{amsfonts} 
\usepackage{epstopdf}
\usepackage{epsf,epsfig,verbatim,amssymb,amsmath,array,cite,amsthm,multicol,multirow}  
\usepackage{psfrag,bm,xspace}
\usepackage{hhline}
\usepackage{physics}

\newtheorem{thm}{Theorem}
\newtheorem{prop}{Proposition}
\newtheorem{dfn}{Definition}
\newtheorem{rem}{Remark}
\newtheorem{lem}[thm]{Lemma}
\newtheorem{cor}{Corollary}[thm]

\newtheorem{assumption}{Assumption}

\theoremstyle{definition}

\newtheorem*{prob*}{Problem}

\theoremstyle{remark}

\newcounter{relctr} 
\everydisplay\expandafter{\the\everydisplay\setcounter{relctr}{0}} 

\newcommand\labelrel[2]{%
  \begingroup
    \refstepcounter{relctr}%
    \stackrel{\textnormal{(\alph{relctr})}}{\mathstrut{#1}}%
    \originallabel{#2}%
  \endgroup
}
\AtBeginDocument{\let\originallabel\label} 

\allowdisplaybreaks 
\global\long\def\11{\mathbbm{1}}

\newcommand\numberthis{\addtocounter{equation}{1}\tag{\theequation}}
\newcommand{\ra}{\rightarrow}
\newcommand{\la}{\leftarrow}
\newcommand{\mcl}{\mathcal}
\newcommand{\mbb}{\mathbb}
\newcommand{\mbf}{\mathbf}
\newcommand{\eps}{\epsilon}
\newcommand{\ov}{\overline}
\newcommand{\udl}{\underline}
\newcommand{\uda}{\underaccent}
\newcommand{\ulbar}[1]{ \uda{\bar}{\bar{#1}} }

\def \defeq{\overset{\Delta}{=}}

\def \wh{\widehat}
\def \wt{\widetilde}
\def \pr{\mbb{P}}

\def \l{\left}
\def \r{\right}

\newcommand{\argmin}{\operatorname{argmin}}
\newcommand\mymathop[1]{\mathop{\operatorname{#1}}}

%% file: Template/notation.tex
\newcommand{\dotp}[2]{\langle #1 \; , \; #2 \rangle}
\newcommand{\Stt}[1]{S_{#1}}
\newcommand{\stt}[1]{s_{#1}}
\newcommand{\sspace}{\mcl{S}}
\newcommand{\Otn}[2]{O_{#1}^{(#2)}}
\newcommand{\Ot}[1]{O_{#1}}

\newcommand{\ot}[1]{o_{#1}}
\newcommand{\on}[1]{o^{(#1)}}
\newcommand{\otn}[2]{o_{#1}^{(#2)}}
\newcommand{\ospace}{\mcl{O}}
\newcommand{\onspace}[1]{ \ospace^{(#1)} }
\newcommand{\Atn}[2]{A_{#1}^{(#2)}}
\newcommand{\At}[1]{A_{#1}}

\newcommand{\atn}[2]{a_{#1}^{(#2)}}
\newcommand{\at}[1]{a_{#1}}
\newcommand{\an}[1]{a^{(#1)}}
\newcommand{\aspace}{\mcl{A}}
\newcommand{\anspace}[1]{\aspace^{(#1)}}
\newcommand{\borel}[1]{\mcl{B}\l(#1\r)}

\newcommand{\m}[1]{M_1\l( #1 \r)}
\newcommand{\metric}[1]{d_{#1}}

\newcommand{\Hst}[1]{H_{#1}}
\newcommand{\Hstn}[2]{H_{#1}^{(#2)}}
\newcommand{\hst}[2]{#1_{#2}}

\newcommand{\hstn}[3]{#1_{#2}^{(#3)}}
\newcommand{\hsspace}{\mcl{H}}
\newcommand{\hstspace}[1]{\hsspace_{#1}}
\newcommand{\hsnspace}[1]{\hsspace^{(#1)}}
\newcommand{\hstnspace}[2]{\hsspace_{#1}^{(#2)}}

\newcommand{\uspace}{\mcl{U}}
\newcommand{\uspacen}[1]{\uspace^{(#1)}}
\newcommand{\mun}[1]{\mu^{(#1)}}
\newcommand{\uspacemix}{\uspace_{\mathrm{mixed}}}
\newcommand{\vvspace}{\mcl{V}}
\newcommand{\ut}[2]{#1_{#2}}
\newcommand{\un}[2]{#1^{(#2)}}
\newcommand{\utn}[3]{#1_{#2}^{(#3)}}
\newcommand{\prup}[2]{ \pr^{\l(#1\r)}_{#2} }
\newcommand{\E}[2]{ \mbb{E}^{\l(#1\r)}_{#2}}
\newcommand{\cCost}{c\l( \Stt{t}, \At{t} \r)}

\newcommand{\dCost}{ d\l( \Stt{t}, \At{t} \r) }
\newcommand{\lCost}{l^{(\lambda)}\l( \Stt{t}, \At{t} \r) }

\newcommand{\constraintv}{\breve{D}}	
\newcommand{\lag}[2]{L^{(P_1, \alpha)} \l(#1, #2 \r)}	
\newcommand{\lags}[2]{L \l(#1, #2 \r)}	
\newcommand{\lagmix}[2]{ \wh{L}^{(P_1, \alpha)}\l(#1, #2 \r) }	
\newcommand{\lagsmix}[2]{\wh{L}\l(#1, #2 \r) }

\newcommand{\fullccost}[1]{C^{(P_1, \alpha)}\l(#1\r)}
\newcommand{\fullccosts}[1]{C\l(#1\r)}

\newcommand{\fulldkcosts}[2]{D_{#1}\l(#2 \r)}

\newcommand{\fulldcost}[1]{D^{(P_1, \alpha)}\l( #1 \r)}
\newcommand{\fulldcosts}[1]{D\l( #1 \r)}

\newcommand{\optcost}{\udl{C}^{(P_1, \alpha)}}
\newcommand{\optcosts}{\udl{C}}

\newcommand{\infsup}[2]{\inf\limits_{#1}\sup\limits_{#2}}
\newcommand{\supinf}[2]{\sup\limits_{#1}\inf\limits_{#2}}

\newcommand{\xuspace}{\mcl{X}_{\mcl{U}}}
\newcommand{\xuspacen}[1]{\mcl{X}_{\mcl{U}^{(#1)}}}
\newcommand{\xtspace}[1]{\mcl{X}_{#1}}
\newcommand{\useq}[2]{\hspace{0pt}^{#1}#2}
\newcommand{\pruphst}[3]{p_{P_1} \l(#1, #2, #3\r)}
\newcommand{\pruphsts}[3]{p\l(#1, #2, #3\r)}
\newcommand{\zuphst}[3]{W_{P_1}\l(#1, #2, #3\r)}
\newcommand{\zuphsts}[3]{W\l(#1, #2, #3\r)}

\newcommand{\proj}[2]{\texttt{proj}^{(#1)}\l( #2\r)}

\newcommand{\Gt}[1]{\Gamma_{#1}}
\newcommand{\Gtn}[2]{\Gamma_{#1}^{(#2)}}
\newcommand{\gt}[1]{\gamma_{#1}}
\newcommand{\gtn}[2]{\gamma_{#1}^{(#2)}}
\newcommand{\gtspace}[1]{\mcl{G}_{#1}}
\newcommand{\gtnspace}[2]{\mcl{G}_{#1}^{(#2)}}
\newcommand{\wtHstn}[2]{\wt{H}_{#1}^{(#2)} }
\newcommand{\wthstnspace}[2]{\wt{\mcl{H}}_{#1}^{(#2)} }

\newcommand{\Qtl}[2]{Q_{#1}^{\l(#2\r)}}
\newcommand{\whQtl}[2]{\wh{Q}_{#1}^{\l(#2\r)}}
\newcommand{\whQl}[1]{\wh{Q}^{\l(#1\r)}}

\newcommand{\Vtl}[2]{V_{#1}^{\l(#2\r)}}
\newcommand{\whVtl}[2]{\wh{V}_{#1}^{\l(#2\r)}}
\newcommand{\whVl}[1]{\wh{V}^{\l(#1\r)}}

\newcommand{\Zhtn}[2]{\wh{Z}_{#1}^{\l(#2\r)}}

\newcommand{\zhn}[1]{\hat{z}^{\l(#1\r)}}
\newcommand{\zhtn}[2]{\wh{z}_{#1}^{\l(#2\r)}}

\newcommand{\zhnspace}[1]{\hat{\mcl{Z}}^{(#1)}}
\newcommand{\zhnspaces}[1]{\hat{\udl{\mcl{Z}}}^{(#1)}}
\newcommand{\zhtnspace}[2]{\hat{\mcl{Z}}_{#1}^{\l(#2\r)}}
\newcommand{\zhtnspaces}[2]{\hat{\udl{\mcl{Z}}}_{#1}^{\l(#2\r)}}

\newcommand{\lamdah}{\wh{\xi}}
\newcommand{\Lamdaht}[1]{\wh{\Xi}_{#1}}

\newcommand{\Lamdahtn}[2]{\wh{\Xi}_{#1}^{\l(#2\r)}}
\newcommand{\lamdaht}[1]{\wh{\xi}_{#1}}
\newcommand{\lamdahtn}[2]{\wh{\xi}_{#1}^{\l(#2\r)}}
\newcommand{\lamdahn}[1]{\wh{\xi}^{\l(#1\r)}}
\newcommand{\lamdahspace}{\mcl{E}}
\newcommand{\lamdahtspace}[1]{\wh{\mcl{E}}_{#1}}
\newcommand{\lamdahnspace}[1]{\wh{\mcl{E}}^{\l(#1\r)}}
\newcommand{\lamdahtnspace}[2]{\wh{\mcl{E}}_{#1}^{\l(#2\r)}}

\newcommand{\varthetahtn}[2]{\wh{\vartheta}_{#1}^{\l(#2\r)}}
\newcommand{\varthetahn}[1]{\wh{\vartheta}^{\l(#1\r)}}
\newcommand{\phihtn}[2]{\wh{\phi}_{#1}^{\l(#2\r)}}
\newcommand{\varphin}[1]{\varphi^{(#1)}}
\newcommand{\rhon}[1]{\rho^{(#1)}}
\newcommand{\psin}[1]{\psi^{(#1)}}

%% file: Sections/introduction.tex
\section{Introduction}\label{sec:introduction}
Single-Agent Markov Decision Processes (SA-MDPs) \cite{bellman57} and Single-Agent Partially Observable Markov Decision Processes (SA-POMDPs) \cite{astrom65} have long served as the basic building-blocks in the study of sequential decision-making. 
An SA-MDP is an abstraction in which an agent (sequentially) interacts with a fully-observable Markovian environment to solve a multi-period optimization problem; in contrast, in SA-POMDP, the agent only gets to observe a noisy or incomplete version of the environment. 
In 1957, Bellman proposed dynamic-programming as an approach to solve SA-MDPs \cite{bellman57,howard:dp}. This combined with the characterization of SA-POMDP into an equivalent SA-MDP \cite{smallwood1973optimal,sondik1978optimal,kaelbing199899} (in which the agent maintains a belief about the environment's true state) made it possible to extend dynamic-programming results to SA-POMDPs. \emph{Reinforcement learning}~\cite{sutton98} based algorithmic frameworks (\cite{watkins1992,rummery1994,schulman2015,schulman2017,subramanian19,subramanian22} to name a few) use data-driven dynamic-programming approaches to solve such single-agent sequential decision problems when the environment is unknown.

In many engineering systems, there are multiple decision-makers that collectively solve a sequential decision-making problem but with safety constraints: e.g., a team of robots performing a joint task, a fleet of automated cars navigating a city, multiple traffic-light controllers in a city, etc. Bandwidth constrained communications and communication delays in such systems lead to a decentralized team problem with information asymmetry. In this work, we study a fairly general abstraction of such systems and develop an algorithmic framework that (approximately) solves the underlying constrained decision problem. The abstraction that we consider is that of a cooperative multi-agent constrained POMDP, 
hereon referred to as MA-C-POMDP. The special cases of MA-C-POMDP when there are no constraints, when there is only one agent, or when the environment is fully observable to each agent are referred to as MA-POMDP\footnote{For a good introduction to MA-POMDPs, see \cite{oliehoek16}.}, SA-C-POMDP, and MA-C-MDP respectively. The relationships among these models are shown in Figure \ref{fig:model_relationships}.

\subsection{Related Work}
Prior work on planning and learning under constraints has primarily focused on single-agent constrained MDP (SA-C-MDP) where unlike in SA-MDPs, the agent solves a constrained optimization problem. For this setup, a number of fundamental results from the planning perspective have been derived -- for instance \cite{altman94,altman96,feinberg94,feinberg95,feinberg96,feinberg2000,feinberg2020}; see \cite{altman-constrainedMDP} for a details of the convex-analytic approach. 
These results have led to the development of many algorithms in the learning setting: see \cite{borkar2005AnAA,bhatnagar2010AnAA,bhatnagar2012AnOA,Wei2022APM,wei22a-pmlr-v151,bura2021,vaswani2022}. Unlike SA-C-MDPs, rigorous results for SA-C-POMDPs are limited; few  reference works include \cite{dongho2011,jongmin18,undurti2010,jamgochian2022}.

There has been little work done on MA-POMDPs and MA-C-POMDPs due to challenges arising from the combination of \emph{partial observability} of the environment and \emph{information-asymmetry}\footnote{Here, information-asymmetry refers to the mismatch in the information each agent has when choosing their action; information asymmetry in decision problems typically results in non-classical information structures~\cite{mahajan2012information}.}; solving a finite-horizon MA-POMDP with more than two agents is known to be NEXP-complete \cite{bernstein00}. Nevertheless, conceptual approaches exist to establish solution methodologies and structural properties in (finite-horizon) MA-POMDPs namely: i) the person-by-person approach~\cite{witsenhausen1979structure}; ii) the designer's approach~\cite{witsenhausen1973standard}; and iii) the \emph{common-information (CI) approach}~\cite{nayyar13,nayyar14}, which we will use in the second part of this work. In the CI approach, the common information of all agents is used to instantiate a fictitious/virtual entity, known as the \emph{coordinator}, that takes an action based only on this common information, while its action is an enforcing prescription to each agent on how to act given a specific realization of their private history. Akin to a SA-POMDP, this transformation leads to the formulation of a dynamic program that in principle can be used 
to solve the (finite-horizon) MA-POMDP. The CI approach has also led to the development of a multi-agent reinforcement learning (MARL) framework \cite{hsu22} where agents learn good compressions of common and private information that can suffice for approximate optimality. On the empirical front, a few worth-mentioning works include \cite{gupta17,nolan2020,rashid2020,rashid2020-2}.

The technical challenges increase even more so for MA-C-POMDPs where restriction of the policy-profile space to 
deterministic policy-profiles is no longer an option\footnote{Restricting to deterministic policies can be sub-optimal in SA-C-MDPs and SA-C-POMDPs: see \cite{altman-constrainedMDP} and \cite{dongho2011} for details.}. Thus, 
the coordinator in the equivalent SA-C-POMDP has an uncountable prescription space, which leads to an uncountable state-space in its equivalent SA-C-MDP. This is an issue because most fundamental results in the theory of SA-C-MDPs (largely based on occupation-measures) rely heavily on the state-space being countable; see \cite{altman-constrainedMDP}. Due to these reasons, the study of MA-C-POMDPs calls for a new methodology, one which avoids this transformation and directly studies the problem. Our work takes the first steps in this direction.

\subsection{Contributions}
We present the first rigorous approach for MA-C-POMDPs that is based on a principled theory of strong duality and using measure theoretic results. In particular,
\begin{itemize}
\item
In Theorem \ref{thm:strongduality}, we establish strong duality and existence of a saddle-point for a fairly general formulation of a MA-C-POMDP.
\item
In Section \ref{sec:history_embedding}, we propose a universal compression methodology similar to \cite{hsu22} that holds for any valid vector of Lagrange-multipliers.
\item
In Section \ref{sec:marl}, we combine strong duality with the proposed compression methodology to lay out the design of a primal-dual MARL framework based on centralized training and decentralized execution (CTDE) -- for developing close-to-optimal algorithms (using neural-networks to approximate any  functions).
\end{itemize}

\subsection{Organization}
The rest of the paper is organized as follows. A mathematical formulation of cooperative MA-C-POMDPs is introduced in Section \ref{sec:problem}. Strong duality results are derived in Section \ref{sec:strongduality} with details deferred to Appendices \ref{sec:appendix:intermediary_results} and \ref{sec:appendix:helpful_facts}. Computationally feasible algorithms based on compression of agents' histories to approximate information states are developed in Sections \ref{sec:history_embedding} and \ref{sec:marl}. Finally, concluding remarks are given in Section \ref{sec:conclusion}.

\subsection{Notation}
Before we present the model, we highlight the key notation in this paper.
\begin{itemize}
\item The sets of integers and positive integers are respectively denoted by $\mbb{Z}$ and $\mbb{N}$. For integers $a$ and $b$, $[a,b]_{\mbb{Z}}$ represents the set $\{a, a+1, \dots, b\}$ if $a\le b$ and $\emptyset$ otherwise. The notations [a] and $[a,\infty]_{\mbb{Z}}$ are used as shorthands for $[1, a]_{\mbb{Z}}$ and $\{a, a+1, \dots \}$, respectively.
\item For integers $a \le b$ and $c \le d$, and a quantity of interest $q$, $\un{q}{a:b}$ is a shorthand for the vector $\l( \un{q}{a}, \un{q}{a+1}, \dots, \un{q}{b} \r)$ while $\ut{q}{c:d}$ is a shorthand for the vector $\l( \ut{q}{c}, \ut{q}{c+1}, \dots, \ut{q}{d} \r)$. The combined notation $\utn{q}{a:b}{c:d}$ is a shorthand for the vector $(\utn{q}{i}{j}: i \in [a,b]_{\mbb{Z}}, j \in [c, d]_{\mbb{Z}})$. The infinite tuples $\l( \un{q}{a}, \un{q}{a+1}, \dots, \r)$ and $\l( \ut{q}{c}, \ut{q}{c+1}, \dots, \r)$ are respectively denoted by $ \un{q}{a:\infty}$ and $\ut{q}{c:\infty}$.
\item For two real-valued vectors $v_1$ and $v_2$, the inequalities $v_1 \le v_2$ and $v_1 < v_2$ are meant to be element-wise inequalities.
\item Probability and expectation operators are denoted by $\pr$ and $\mbb{E}$, respectively. Random variables are denoted by upper-case letters and their realizations by the corresponding lower-case letters.  At times, we also use the shorthand  $\mbb{E}\l[ \cdot | x \r] \defeq \mbb{E}\l[ \cdot | X = x \r]$ and $\pr\l( y | x \r) \defeq \pr\l( Y = y | X = x \r)$ for conditional quantities.
\item Whenever a conditional probability or conditional expectation is written, it is implicitly assumed that the conditioning event has non-zero probability.
\item Topological spaces are denoted by upper-case calligraphic letters. For a topological-space $\mcl{W}$, $\borel{\mcl{W}}$ denotes the Borel $\sigma$-algebra, measurability is determined with respect to $\borel{\mcl{W}}$, and $\m{\mcl{W}}$ denotes the set of all probability measures on $\borel{\mcl{W}}$ endowed with the topology of weak convergence. Also, unless stated otherwise, ``measure'' means a non-negative measure.
\item Unless otherwise stated, if a set $\mcl{W}$ is countable, as a topological space it will be assumed to have the discrete topology. Therefore, the corresponding Borel $\sigma$-algebra $\borel{\mcl{W}}$ will be the power-set $2^{\mcl{W}}$.
\item Unless stated otherwise, the product of a collection of topological spaces will be assumed to have the product topology.
\item The notation in Appendices \ref{sec:appendix:helpful_facts} and \ref{sec:appendix:minimax} is exclusive and should be read independent of the rest of the manuscript.
\item To aid the reading experience, we strongly encourage the reader to refer to important notation and definitions listed in Appendix \ref{sec:appendix:notation}.
\end{itemize}
\begin{figure}
    \centering
    \includegraphics[width=0.9\linewidth]{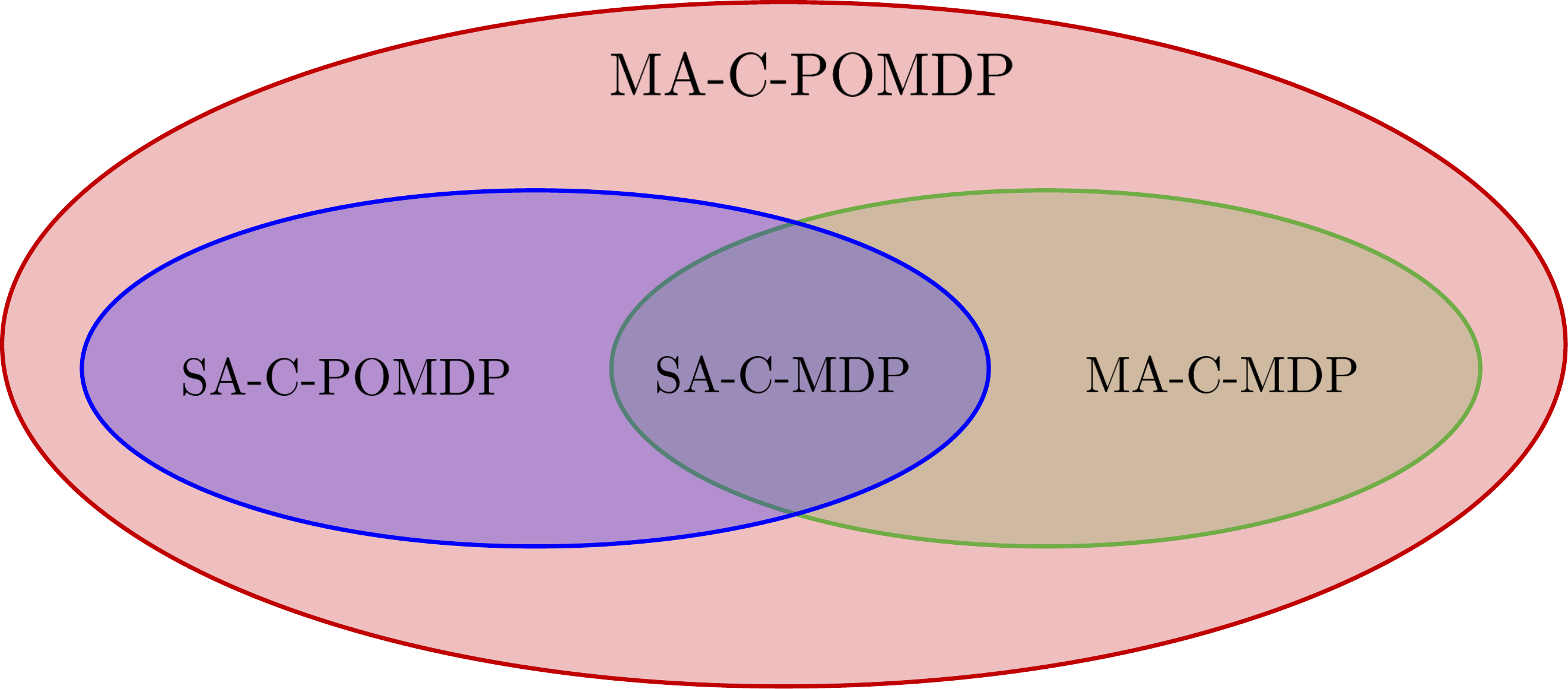}
    \caption{Relationships between Models of Team-based Sequential Decision-Making under Constraints.}
    \label{fig:model_relationships}
\end{figure}

%% file: Sections/problem.tex
\section{Model}\label{sec:problem}
Let $ \l( N, \sspace, \ospace, \aspace, \mcl{P}_{tr}, \l( c, d\r), P_{1}, \mcl{U}, \alpha \r)$
denote a (cooperative) MA-C-POMDP with $N$ 
agents, state space $\sspace$, joint-observation space $\ospace$, a joint-action space $\aspace$, transition-law $\mcl{P}_{tr}$, immediate-cost functions $c$ and $d$, (fixed) initial distribution $P_1$, space of decentralized policy-profiles $\mcl{U}$, and discount factor $\alpha\in\l(0, 1\r)$. The decision problem has the following attributes and notations.
\begin{enumerate}
\item \textbf{State Process}: 
The state-space $\sspace$ is some topological space with a Borel $\sigma$-algebra $\borel{\sspace}$. 
The state-process is denoted by $\l\{\Stt{t} \r\}_{t=1}^{\infty}$.
\item \textbf{Joint-Observation Process}: The joint-observation space $\ospace$ is a countable discrete space of the form $ \ospace = \prod_{n=0}^{N} \onspace{n}$, where $\onspace{0}$ denotes the common observation space of all agents and $\onspace{n}$ denotes the private observation space of agent $n \in [N]$. The joint-observation process is denoted by $\l\{ \Ot{t} \r\}_{t=1}^{\infty}$ where $\Ot{t} = \Otn{t}{0:N}$ and is such that at time $t$, agent $n\in [N]$ observes $\Otn{t}{0}$ and $\Otn{t}{n}$ only.

\item \textbf{Joint-Action Process}: The joint-action space $\aspace $ is a finite discrete space of the form $ \aspace  = \prod_{n=1}^{N} \anspace{n} $, where $\anspace{n}$ denotes the action space of agent $n\in [N]$. The joint-action process is denoted by $\l\{ \At{t} \r\}_{t=1}^{\infty}$ where $\At{t} = \Atn{t}{1:N}$ and $\Atn{t}{n}$ denotes the action of agent $n$ at time $t$.\footnote{The results in this work also hold if for every $( \hstn{h}{t}{0}, \hstn{h}{t}{n}) \in \hstnspace{t}{0}\times\hstnspace{t}{n}$, agent $n$ is allowed to take action from a separate finite discrete space $\anspace{n}(\hstn{h}{t}{0}, \hstn{h}{t}{n})$.} Since all $\anspace{n}$'s and $\aspace$ are finite, they are all compact metric spaces.\footnote{Hence, also complete and separable.}

\item \textbf{Transition-law}: 
At time $t \in \mbb{N}$, given the current state $\Stt{t}$ and current joint-action $\At{t}$, the next state $\Stt{t+1}$ and the next joint-observation $\Ot{t+1}$ are determined in a time-homogeneous manner independent of all previous states, all previous and current joint-observations, and all previous joint-actions. The transition-law is given by 
\begin{align*}
\mcl{P}_{tr} \defeq \l\{ P_{saBo}: s \in \sspace, a\in\aspace, B \in  \borel{\sspace}, o \in \ospace \r\},\numberthis\label{eq:transitionlaw}
\end{align*}
where for all $t \in \mbb{N}$, 
\begin{align}
\begin{split}\label{eq:psabo}
&\pr\l( \Stt{t+1} \in B , \Ot{t+1} = o | \Stt{1:t-1}  = \stt{1:t-1}, \r.\\
&\hspace{20pt} \l. \Ot{1:t} = \ot{1:t}, \At{1:t-1} = \at{1:t-1}, 
\Stt{t} = s, \At{t} = a \r)\\
&\hspace{10pt}
= \pr\l( \Stt{t+1} \in B, \Ot{t+1} = o | \Stt{t} = s, \At{t} = a \r)\\
&\hspace{10pt} 
\defeq P_{saBo}.
\end{split}
\end{align}

\item \textbf{Immediate-costs}: 
The immediate cost $c : \sspace \times \aspace \mapsto \mbb{R}$ is a real-valued function whose expected discounted aggregate (to be defined later) we would like to minimize. On the other hand, the immediate cost $d : \sspace \times \aspace \mapsto \mbb{R}^K$ is $\mbb{R}^K$-valued function whose expected discounted aggregate we would like to keep within a specified threshold. For these reasons, we call $c$ and $d$ as the immediate objective and constraint costs respectively. We make use of the following assumption on immediate costs.

\begin{assumption}[Bounded Immediate-costs]\label{assmp:boundedcosts}
\begin{enumerate}
\item[]
\item[(a)] Immediate objective cost is bounded from below, i.e., there exists $\udl{c} \in \mbb{R}$ such that
\begin{align*}
\udl{c} \le c(\cdot, \cdot). \numberthis\label{eq:cboundedbelow}
\end{align*}
\item[(b)] Immediate objective cost is bounded from above, i.e., there exists $\ov{c} \in \mbb{R}$ such that
\begin{align*}
c(\cdot, \cdot)  \le \ov{c}. \numberthis\label{eq:cboundedabove}
\end{align*}
\item[(c)] Immediate constraints costs are bounded both from above and below, i.e., there exist $\udl{d}, \ov{d} \in \mbb{R}^K$ such that
\begin{align*}
\udl{d} \le d(\cdot, \cdot) \le \ov{d}. \numberthis\label{eq:dbounded}
\end{align*}
\end{enumerate}
Let $ \ulbar{c} = |\udl{c}| \vee |\ov{c}|$ and $ \ulbar{d} = \| \udl{d} \|_{\infty} \vee  \| \ov{d} \|_{\infty}$ so that under (a)-(c), we have 
\begin{align*}
|c(\cdot, \cdot)| \le \ulbar{c} < \infty \text{ and } \| d(\cdot, \cdot) \|_{\infty} \le \ulbar{d} < \infty.
\end{align*}
\end{assumption}

\item \textbf{Initial Distribution}: $P_1$ is a (fixed) probability measure for the initial state and initial joint-observation, i.e., $P_1 \in \m{\sspace\times \ospace}$ and
\begin{align}
\begin{split}\label{eq:initialdistribution}
P_1\l( B, o \r) &\defeq \pr\l( \Stt{1} \in B, \Ot{1} = o \r).
\end{split}
\end{align}    

\item \textbf{Space of Policy-Profiles}: At time $t\in \mbb{N}$, the \emph{common history} of all agents is defined as all the common observations received thus far, i.e.,  $\Hstn{t}{0} \defeq \l( \Otn{1:t}{0}  \r)$. Similarly, the \emph{private history} of agent $n\in [N]$ at time $t$ is defined as all observations received and all the actions taken by the agent thus far (except for those that are part of the common information), i.e., 
\begin{align}
\begin{split}\label{eq:Hstn}
\Hstn{1}{n} &\defeq \Otn{1}{n} \setminus \Otn{1}{0}, \text{ and}\\
\Hstn{t}{n} &\defeq \l( \Hstn{t-1}{n}, (\Atn{t-1}{n}, \Otn{t}{n})\setminus \Otn{t}{0} \r) \ \forall t \in [2,\infty]_{\mbb{Z}}.
\end{split}
\end{align}
Finally, the \emph{joint-history} at time $t$ is defined as the tuple of the common history and all the private histories at time $t$, i.e., $\Hst{t}  \defeq \Hstn{t}{0:n}$. 

With the above setup, we define a (decentralized) behavioral policy-profile $u $ as a tuple $\un{u}{1:N} \in \uspace \defeq \prod_{n=1}^{N} \uspace^{(n)} $ where $\un{u}{n}$ denotes some behavioral policy used by agent $n$, i.e., $\un{u}{n}$ itself is a tuple of the form $ \utn{u}{1:\infty}{n}$ where $\utn{u}{t}{n}$ maps $\hstnspace{t}{0} \times \hstnspace{t}{n}$ to $\m{\anspace{n}}$, and where agent $n$ uses the distribution $\utn{u}{t}{n} ( \Hstn{t}{0}, \Hstn{t}{n} )$ to choose its action $\Atn{t}{n}$. We pause to emphasize that at any time $t$, each agent randomizes over its action-set independently of all other agents, that is, without any \textit{common randomness}. Thus, given a joint-history $\hst{h}{t} \in \hstspace{t}$ at time $t$, the probability that joint-action $\at{t} \in \aspace$ is taken is given by
\begin{align*}
\ut{u}{t}\l(\at{t}|\hst{h}{t}\r) = \ut{u}{t}\l(\hst{h}{t}\r)\l(\at{t} \r) 
&\defeq \prod_{n=1}^{N} \utn{u}{t}{n}\l( \hstn{h}{t}{0}, \hstn{h}{t}{n}  \r) \l( \atn{t}{n} \r)\\
&=\prod_{n=1}^{N} \utn{u}{t}{n}\l( \atn{t}{n} \big| \hstn{h}{t}{0}, \hstn{h}{t}{n}  \r).\numberthis\label{eq:uah}
\end{align*}
\begin{rem}	
With Assumption \ref{assmp:boundedcosts}(a) and \ref{assmp:boundedcosts}(c), the conditional expectations $\mbb{E}_{P_1} \l[ \cCost \mid \Hst{t} = \hst{h}{t}, \At{t} = \at{t} \r] $ and $\mbb{E}_{P_1} \l[ \dCost \mid \Hst{t} = \hst{h}{t}, \At{t} = \at{t} \r] $ exist, are unique, and are bounded from below. Furthermore, the latter are element-wise finite. 	
\end{rem}

\item \textbf{Optimization Problem}: Let $\prup{u}{P_1}$ be the probability measure corresponding to policy-profile $u\in\uspace$ and initial-distribution $P_1$ and let $\E{u}{P_1}$ denote the corresponding expectation operator.\footnote{The existence and uniqueness of $\prup{u}{P_1}$ can be ensured by an adaptation of the Ionesca-Tulcea theorem \cite{tulcea49}.} We define \textit{infinite-horizon expected total discounted costs} $C:\uspace\ra \mbb{R} \cup \{ \infty \}$ and $D:\uspace \ra \mbb{R}^K$ as 
\begin{align*}
\fullccosts{u} = \fullccost{u} &\defeq  \E{u}{P_1} \l[ \sum_{t=1}^{\infty} \alpha^{t-1} \cCost \r],
\numberthis\label{eq:C}\\
\text{and}\ 
\fulldcosts{u} = \fulldcost{u} &\defeq  \E{u}{P_1}\l[ \sum_{t=1}^{\infty} \alpha^{t-1} \dCost \r].\numberthis\label{eq:D}
\end{align*}
\begin{rem}\label{rem:real_valued_aggregatecosts}
Assumption \ref{assmp:boundedcosts}(a) (objective cost bounded from below) and \ref{assmp:boundedcosts}(c) (constraint costs bounded) ensures that $\fullccosts{u}\in \mbb{R} \cup \{ \infty \}$, and $\fulldcosts{u} \in \mbb{R}^K$ with (absolute) element-wise bound $\ulbar{d}/(1-\alpha)$.
\end{rem}
The decision process proceeds as follows: \textit{i}) At time $t\in\mbb{N}$, the current state $\Stt{t}$ and observations $\Ot{t}$ are generated (according to $\mcl{P}_{tr}$ and/or $P_1$); \textit{ii}) Each agent $n\in[N]$ chooses an action $\an{n} \in \anspace{n}$ based on $\Hstn{t}{0}, \Hstn{t}{n}$; \textit{iii}) the immediate-costs $c\l( \Stt{t}, \At{t} \r), d\l(\Stt{t}, \At{t}\r)$ are incurred\footnote{In the planning context, the immediate-costs are known by all agents. In the learning context, we assume that the immediate-costs are observed by all agents. This assumption is innocuous for CTDE based learning algorithms where all quantities are known by a central entity, which we will refer to as the \emph{supervisor} in this work. See Section \ref{sec:marl}.}; \textit{iv}) The system moves to the next state and observations according to the transition-law $\mcl{P}_{tr}$.
\begin{figure}
    \centering
    \includegraphics[width=\linewidth]{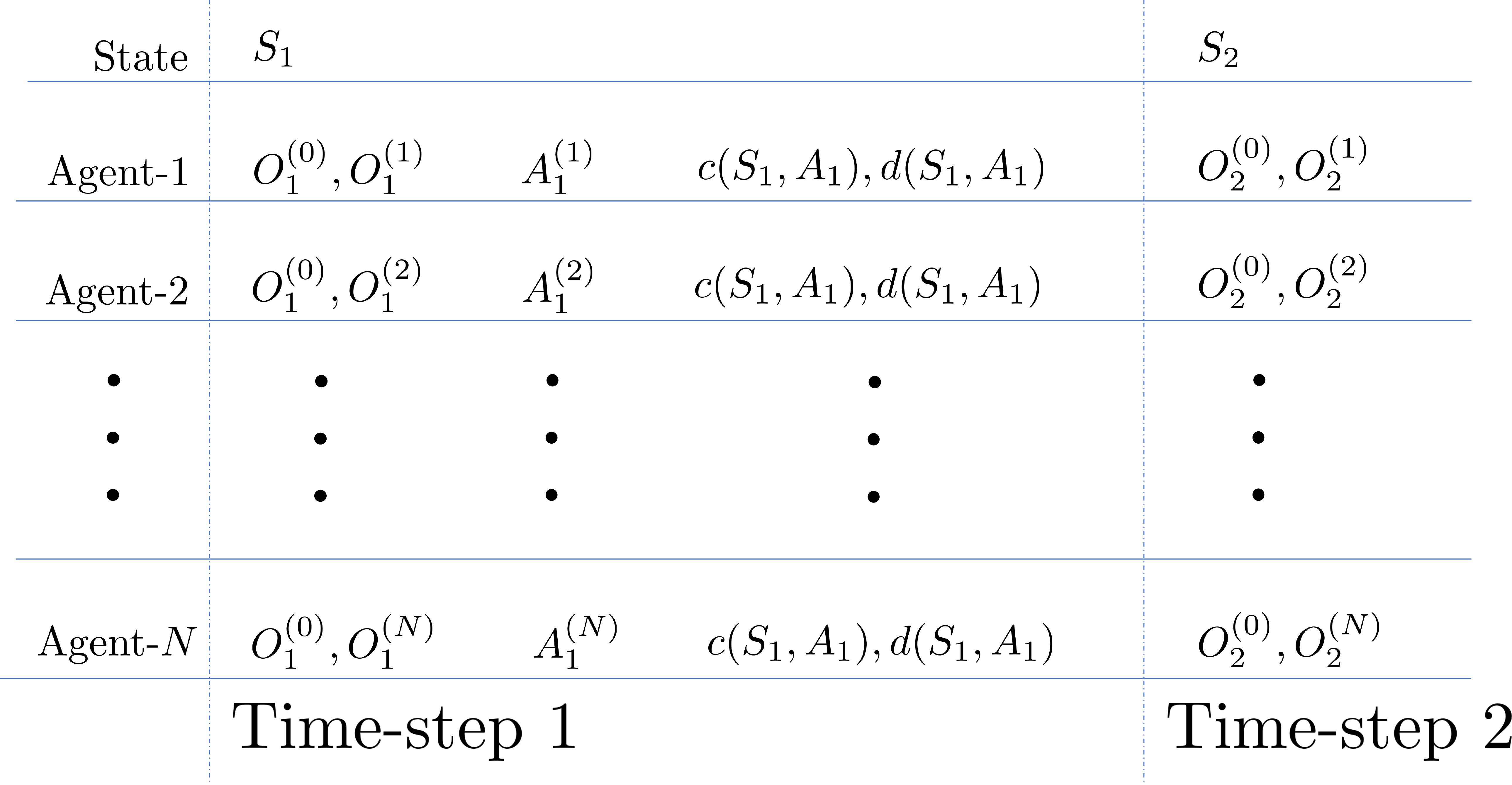}
    \caption{Relationships between Models of Team-based Sequential Decision-Making under Constraints.}
    \label{fig:timeline}
\end{figure}
Under these rules, the goal of the agents is to work cooperatively to solve the following constrained optimization problem.
\begin{equation}\tag{\textit{MA-C-POMDP}}\label{eq:macpomdp}
\begin{split}
&\text{minimize } \fullccosts{u}\\
&\hspace{10pt} \text{subject to } u \in \uspace \text{ and } \fulldcosts{u} \le \constraintv.
\end{split}
\end{equation}
Here, $\constraintv$ is a fixed $K$-dimensional real-valued vector. We refer to the solution of \eqref{eq:macpomdp} as its \emph{optimal value} and denote it by $\optcosts = \optcost$. In particular, if the set of feasible policy-profiles is empty, we set $\optcosts$ to $\infty$ and with slight abuse of terminology will consider any policy-profile in $\uspace$ to be optimal.

The following assumption about feasibility of \eqref{eq:macpomdp} will be used in the paper.
\begin{assumption}[Slater's Condition]\label{assmp:slatercondition}
There exists a policy-profile $\ov{u} \in \uspace$ and $\zeta > 0$ for which
\begin{align*}
\fulldcosts{\ov{u}} \le \constraintv - \zeta1.\numberthis\label{eq:slatercondition} 
\end{align*}
\end{assumption}

\end{enumerate}

%% file: Sections/strong_duality.tex
\section{Characterization of Strong Duality}\label{sec:strongduality}

To solve \eqref{eq:macpomdp}, let us define the Lagrangian function $L : \uspace\times \mcl{Y} \mapsto \mbb{R} \cup \{\infty\}$ as follows.
\begin{align*}
\lags{u}{\lambda} &= \lag{u}{\lambda} 
\defeq \fullccosts{u} + \dotp{\lambda}{\fulldcosts{u} -\constraintv}\\
&= \fullccosts{u} + \sum_{k=1}^{K} \lambda_k \l(D_{k}\l(u\r) - \constraintv_k \r)
,\numberthis\label{eq:lagrangian}
\end{align*}
Here, $\mcl{Y} \defeq \{ \lambda \in \mbb{R}^K : \lambda \ge 0\}$ is the set of tuples of $K$ non-negative real-numbers, each commonly known as a Lagrange-multiplier. Our first result shows that the the solution $\udl{C}$ satisfies
\begin{align*}
\optcosts &= \infsup{u\in \uspace }{\lambda\in \mcl{Y}} \lags{u}{\lambda}
,\numberthis\label{eq:optccost:infsup}
\end{align*}
and that the inf and sup can be interchanged, i.e.,
\begin{align*}
\optcosts &= \supinf{\lambda\in \mcl{Y}}{u\in \uspace } \lags{u}{\lambda}
.\numberthis\label{eq:optccost:supinf}
\end{align*}

\begin{thm}[Strong Duality and Existence of Saddle Point]\label{thm:strongduality}
Under Assumptions \ref{assmp:boundedcosts}(a) and \ref{assmp:boundedcosts}(c), the following statements hold.
\begin{enumerate}
\item[(a)] The optimal value satisfies 
\begin{align*}
\optcosts = \infsup{u\in\uspace}{\lambda\in \mcl{Y}} \lags{u}{\lambda}
.\numberthis   
\end{align*}
\item[(b)] A policy-profile $u^\star \in \uspace$ is optimal if and only if $\optcosts = \sup_{\lambda\in \mcl{Y}} \lags{u^\star}{\lambda}$.
\item[(c)] Strong duality holds for \eqref{eq:macpomdp}, i.e.,
\begin{align*}
\optcosts &= \infsup{u\in\uspace}{\lambda\in \mcl{Y}} \lags{u}{\lambda} 
= \supinf{\lambda\in \mcl{Y}}{u\in\uspace} \lags{u}{\lambda}.\numberthis
\end{align*}
Moreover, there exists a $u^\star \in \uspace$ such that $\optcosts = \sup_{\lambda \in \mcl{Y}} \lags{u^\star}{\lambda} $ and $u^\star$ is optimal for \eqref{eq:macpomdp}. 
\item[(d)] If Assumption \ref{assmp:slatercondition} holds, then there also exists $\lambda^\star \in \mcl{Y}$ such that the following saddle-point condition holds for all $(u,\lambda)\in \uspace \times \mcl{Y}$,
\begin{align*}
\lags{u^\star}{\lambda} \le \lags{u^\star}{\lambda^\star} = \optcosts \le \lags{u}{\lambda^\star}.
\numberthis\label{eq:saddlepointconditions} 
\end{align*}
i.e., $u^\star$ minimizes $\lags{\cdot}{\lambda^\star}$ and $\lambda^\star$ maximizes $\lags{u^\star}{\cdot}$. In addition to this, the primal dual pair $\l( u^\star, \lambda^\star \r)$ satisfies the complementary-slackness condition:
\begin{align*}\label{eq:compslack}
\dotp{\lambda^\star}{\fulldcosts{u^\star}-\constraintv} = 0.\numberthis
\end{align*}
\end{enumerate}
\end{thm}

\begin{proof}
\begin{enumerate}
\item[(a)] If $u \in \uspace$ is feasible (i.e., it satisfies $\fulldcosts{u} \le \constraintv$), then the $\sup$ is obtained by choosing $\lambda = 0$, so
\begin{align*}
\sup_{\lambda\in\mcl{Y}} \lags{u}{\lambda} &= \fullccosts{u}.
\numberthis\label{eq:ufeasible}
\end{align*}
If $u \in \uspace$ is not feasible, then 
\begin{align*}
\sup_{\lambda\in \mcl{Y}} \lags{u}{\lambda} = \infty.
\numberthis\label{eq:unotfeasible}
\end{align*}
Indeed, suppose WLOG that the $k^{th}$ constraint is violated, i.e., $\fulldkcosts{k}{u} > \constraintv_k$, then $\infty$ can be obtained by choosing $\lambda_k$ arbitrarily large and setting other $\lambda_k$'s to 0). From \eqref{eq:ufeasible}, \eqref{eq:unotfeasible}, and our convention that $\optcosts = \infty$ whenever the feasible-set is empty, it follows that
\begin{align*}
\optcosts = \infsup{u\in\uspace}{\lambda\in\mcl{Y} } \lags{u}{\lambda}.
\numberthis\label{eq:fullccostisinfsup}
\end{align*}

\item[(b)] By our convention on the value of $\optcosts$ (when there is no feasible policy-profile), $u^\star$ is optimal if and only if $\fullccosts{u^\star} = \optcosts $, i.e., $\sup_{\lambda\in \mcl{Y} } \lags{u^\star}{\lambda} = \optcosts$.

\item[(c)] To establish strong duality, we use a Minimax Theorem (see Proposition \ref{prop:sionminimax}) which requires $\uspace$ and $\mcl{Y}$ to be convex\footnote{Convexity is a set property rather than a topological property. In the rest of the paper, by a ``convex topological space'', we mean convexity of the set on which the topology is defined.} topological spaces (with $\uspace$ being compact also). It is clear that $\mcl{Y}$ is convex and we can endow it with the usual subspace topology of $\mbb{R}^K$. 
For $\uspace$ however, we need to endow it with a suitable topology in which it is both convex and compact. To achieve compactness, we can use the finiteness of joint-action space $\aspace$ and the countability of joint-observation space $\ospace$ to associate $\uspace$ with a product of compact sets that are parameterized by (countable number of) all possible histories. Tychonoff's theorem (see Proposition \ref{prop:tychonoff}) then helps establish compactness under the product topology. 
(Convexity holds trivially). 
Now, we make this idea precise. For $t \in \mbb{N}$ and $n\in[0,N]_{\mbb{Z}}$, let $\hstnspace{t}{n}$ denote the set of all possible realizations of $\Hstn{t}{n}$. Then, by countability of observation and action spaces, the sets
\begin{align}
\begin{split}\label{eq:hthnandh}
\hstspace{t} &\defeq \prod_{n=0}^{N} \hstnspace{t}{n},\\
\hsnspace{n} &\defeq \bigcup_{t= 1}^{\infty} \hstnspace{t}{0} \times \hstnspace{t}{n}, \text{ and }\\
\hsspace &\defeq \bigcup_{t=1}^{\infty} \hstspace{t},
\end{split}
\end{align}
are countable. Here, $\hstspace{t}$ is the set of all possible joint-histories at time $t$, $\hsnspace{n}$ is the set of all possible histories of agent $n$, and $\hsspace$ is the set of all possible joint-histories. With this in mind, one observes that $\uspace$ is in one-to-one correspondence with the set $\xuspace \defeq \prod_{n=1}^{N} \xuspacen{n}$, where
\begin{align*}
\xuspacen{n} \defeq \prod_{h \in \hsnspace{n}} \m{\anspace{n}; h},\numberthis\label{eq:xuspace}
\end{align*}
and $\m{\anspace{n}; h}$ is a copy of $\m{\anspace{n}}$ dedicated for agent-$n$'s history $h$. For example, a given policy $u$ would correspond to a point $x\in \xuspace$ such that $x_{n, \l( \hstn{h}{t}{0}, \hstn{h}{t}{n}\r) } = \utn{u}{t}{n} \l( \cdot | \hstn{h}{t}{0}, \hstn{h}{t}{n} \r) $, and similarly, vice versa. 

\hspace{5pt} Since $\anspace{n}$ is a complete separable (compact) metric space, by Prokhorov's Theorem (see Proposition \ref{prop:prokhorov}), each $\m{\anspace{n}; h}$ is a compact (and convex\footnote{Convexity of $\m{\anspace{n}}$ is trivial.}) metric space (with the topology of weak-convergence). Therefore, endowing $\xuspacen{n}$ and $\xuspace$ with the product topology makes each a compact (and convex) metric space via Tychonoff's theorem (see Proposition \ref{prop:tychonoff}), which is also metrizable (via Proposition \ref{prop:metrizability}). Given the one-to-one correspondence, \textbf{from now onward, we assume that $\uspacen{n}$ and $\uspace$ have the same topology as that of $\xuspacen{n}$ and $\xuspace$ respectively}. Henceforth, we will consider $C$, $D_k$, and $L$ as functions on topological spaces. Furthermore, since  $\uspacen{n}$'s and $\uspace$ have been shown to be compact metric spaces (hence, also complete and separable), we can also define $\borel{\uspacen{n}}$, $\borel{\uspace} = \otimes_{n=1}^{N} \borel{\uspacen{n}}$\footnote{For separable metric spaces $\mcl{W}_1, \mcl{W}_2, \ldots$, $\borel{\mcl{W}_1 \times \mcl{W}_2 \times \ldots } = \borel{\mcl{W}_1} \otimes \borel{\mcl{W}_2} \otimes \ldots$. See \cite{kallenberg2002foundations}[Lemma 1.2].}, and $\m{\uspace}$, where $\m{\uspace}$ is compact (and convex) metrizable space by Prokhorov's theorem (see Proposition \ref{prop:prokhorov}).

\hspace{5pt} To establish part (c), it will be helpful to work with (decentralized) mixtures of behavioral policy-profiles -- wherein each agent first uses a measure $\mun{n} \in \m{\uspacen{n}} $ to choose its policy-profile $\un{u}{n}$ and then proceeds with it from time 1 onward. We denote this set of mixtures by $\uspacemix \defeq \prod_{n=1}^{N} \m{\uspacen{n}}$, whose typical element, denoted by $\mu \defeq \mymathop{\times}_{n=1}^{N} \mun{n} $, is a factorized measure on $\uspace$, i.e., $\mun{n}\in \m{\uspacen{n}}$. Since $\uspacemix \subseteq \m{\uspace}$, we endow it with the same metric as that of $\m{\uspace}$. Now, we can extend the definitions of $C$, $D$, and $L$ to $\wh{C} : \uspacemix \ra \mbb{R} \cup \{\infty\}$, $\wh{D}: \uspacemix \ra \mbb{R}^K$, and $\wh{L}: \uspacemix \times \mcl{Y} \ra \mbb{R} \cup \{\infty\}$ as follows:		
\begin{align}		
\begin{split}\label{eq:lagrangianmix}		
\wh{C} (\mu) &= \wh{C}_{P_1}(\mu) \defeq \E{\mu}{P_1} \l[ \sum_{t=1}^{\infty} \alpha^{t-1} c(S_t, A_t) \r], \\		
\wh{D} (\mu) &= \wh{D}_{P_1}(\mu) \defeq \E{\mu}{P_1} \l[ \sum_{t=1}^{\infty} \alpha^{t-1} d(S_t, A_t) \r], \text{ and }\\		
\lagsmix{\mu}{\lambda} &= \lagmix{\mu}{\lambda} = \wh{C}(\mu) + \dotp{\lambda}{\wh{D}(u)}.		
\end{split}		
\end{align}
In Lemma \ref{lem:dominance}, it is shown that any $\mu \in \uspacemix$ can be replicated by a behavioral policy-policy $u \in \uspace$. Corollary \ref{cor:lbar_and_l} then shows that
\begin{align}
\begin{split}\label{eq:lag_and_lagmix}
\infsup{u\in\uspace}{\lambda\in\mcl{Y}} \lags{u}{\lambda} &= \infsup{\mu \in\uspacemix}{\lambda\in\mcl{Y}} \lagsmix{\mu}{\lambda}, \text{ and } \\
\supinf{\lambda\in\mcl{Y}}{u\in\uspace} \lags{u}{\lambda} &= \supinf{\lambda\in\mcl{Y}}{u\in\uspacemix} \lagsmix{\mu}{\lambda}.
\end{split}
\end{align}
In light of \eqref{eq:lag_and_lagmix}, it suffices to prove part (c) for $\wh{L}$. By definition, $\wh{L}$ is affine and thus trivially concave in $\lambda$. Proposition \ref{prop:integral_linearity} implies that $\wh{L}$ is convex in $\mu$ and Lemma \ref{lem:lsc2} shows that $\wh{L}$ is lower semi-continuous\footnote{For definition of lower semi-continuity, see Definition \ref{dfn:lsc}.} in $\mu$. From Proposition \ref{prop:sionminimax}, it then follows that
\begin{align*}
\infsup{u\in\uspacemix}{\lambda\in \mcl{Y}} \lagsmix{u}{\lambda} = \supinf{\lambda\in \mcl{Y}}{\uspacemix} \lagsmix{u}{\lambda},
\end{align*}
and that there exists $\mu^\star \in \uspacemix$ such that
\begin{align*}	
\sup_{\lambda\in \mcl{Y}} \lagsmix{\mu^\star}{\lambda} = \infsup{u\in\uspacemix}{\lambda\in \mcl{Y}} \lagsmix{u}{\lambda}.	
\end{align*}
The optimality of $\mu^\star$ is implied by parts (b) and (a).
\item[(d)] This follows from Lagrange-multiplier theory.
\end{enumerate}
This concludes the proof.
\end{proof}

\begin{lem}[Lower Semi-Continuity of $\wh{L}$ on $\uspacemix$]\label{lem:lsc2}
Under Assumptions \ref{assmp:boundedcosts}(a) and \ref{assmp:boundedcosts}(c), $\wh{L}$ is lower semi-continuous on $\uspacemix$.
\end{lem}
\begin{proof}
Fix $\lambda\in\mcl{Y}$ and $\mu\in \uspacemix$. Let $\l\{ \mu_i \r\}_{i=1}^{\infty}$ be a sequence of (factorized) measures in $\uspacemix$ that converges to $\mu \in \uspacemix$. Since $\uspacemix \subseteq \m{\uspace}$ and has the same metric as $\m{\uspace}$, it means that $\l\{ \mu_i \r\}_{i=1}^{\infty}$ also converges to $\mu$ in $ \m{\uspace}$. We want to show
\begin{align*}
\liminf_{i\ra\infty} \E{U \sim \mu_i}{P_1} \l[  \lags{U}{\lambda} \r] \ge \E{U \sim \mu}{P_1} \l[  \lags{U}{\lambda} \r].
\end{align*}
By Lemma \ref{lem:lsc}, $L$ is point-wise lower semi-continuous on $\uspace$. Therefore, Proposition \ref{prop:lsc} applies on $\m{\uspace}$ and the above inequality follows.
\end{proof}

\begin{lem}[Lower Semi-Continuity of $L$ on $\uspace$]\label{lem:lsc}
Under Assumptions \ref{assmp:boundedcosts}(a) and \ref{assmp:boundedcosts}(c), the functions $C$ and $D_k$'s are lower semi-continuous on $\uspace$. Hence, $L$ is lower semi-continuous on $\uspace$. 
\end{lem}
\begin{proof}
We will prove the statement for $C$. The proof of lower semi-continuity of $D_k$'s is similar. For brevity, let 
\begin{align*}
\pruphsts{u}{t}{\hst{h}{t}, \at{t}} &= \pruphst{u}{t}{\hst{h}{t}, \at{t}} \defeq \prup{u}{P_1}\l(\Hst{t} = \hst{h}{t}, \At{t} = \at{t}\r),
\\
\zuphsts{u}{t}{\hst{h}{t}, \at{t}} &= \zuphst{u}{t}{\hst{h}{t}, \at{t}}\\
&\hspace{-25pt} \defeq \pruphsts{u}{t}{\hst{h}{t}, \at{t}} \mbb{E}_{P_1}\l[ \cCost | \Hst{t} = \hst{h}{t}, \At{t} = \at{t} \r],
\end{align*}
where we use the convention $0 \cdot \infty = 0$. Then,
\begin{align*}
\fullccosts{u} &= \E{u}{P_1}\l[ \sum_{t=1}^{\infty} \alpha^{t-1} c(S_t, A_t) \r] \\
&= \E{u}{P_1}\l[ \sum_{t=1}^{\infty} \alpha^{t-1} \l( c(S_t, A_t) - \udl{c} \r) \r] + \sum_{t=1}^{\infty} \alpha^{t-1} \udl{c}\\
&\labelrel{=}{eqr:cp1u:a} \sum_{t=1}^{\infty} \alpha^{t-1} \E{u}{P_1} \l[ c(S_t, A_t) - \udl{c} \r] + \sum_{t=1}^{\infty} \alpha^{t-1} \udl{c}\\
&= \sum_{t=1}^{\infty} \alpha^{t-1} \E{u}{P_1} \l[ c(S_t, A_t) \r] \\
&\labelrel{=}{eqr:cp1u:b} \sum_{t=1}^{\infty} \alpha^{t-1} \E{u}{P_1} \l[ \mbb{E}_{P_1} \l[ c(S_t, A_t) | \Hst{t}, \At{t} \r] \r]\\
&= \sum_{t=1}^{\infty} \sum_{\hst{h}{t} \in \hstspace{t}} \sum_{\at{t}\in \aspace} \alpha^{t-1} \zuphsts{u}{t}{\hst{h}{t}, \at{t}}.
\end{align*}
Here, \eqref{eqr:cp1u:a} follows from applying the Monotone-Convergence Theorem to the (non-decreasing and non-negative) sequence $\{ \sum_{t=1}^{i} \alpha^{t-1} \l( c\l( \Stt{t}, \At{t} \r) - \udl{c} \r) \}_{i=1}^{\infty}$ (see Proposition \ref{prop:mct}); and \eqref{eqr:cp1u:b} uses the tower property of conditional expectation.\footnote{The conditional expectations $\mbb{E}_{P_1} \l[ c(S_t, A_t) | \Hst{t}, \At{t} \r]$ exist and are unique because $c(\cdot, \cdot)$ is bounded from below.}

Let $\l\{ \useq{i}{u} \r\}_{i=1}^{\infty}$ be a sequence in $\uspace$ that converges to $u$. By Fatou's Lemma (see Proposition \ref{prop:fatou}),
\begin{align*}
\liminf_{i\ra \infty} \fullccosts{\useq{i}{u}} \ge \sum_{t=1}^{\infty} \sum_{\hst{h}{t} \in \hstspace{t}} \sum_{\at{t}\in \aspace} \alpha^{t-1} \liminf_{i\ra\infty} \zuphsts{\useq{i}{u}}{t}{\hst{h}{t}, \at{t}}.\numberthis\label{eq:step1}
\end{align*}
Following Lemma \ref{lem:puth}, $\pruphsts{\useq{i}{u}}{t}{\hst{h}{t}, \at{t}} $ converges to $\pruphsts{u}{t}{\hst{h}{t}, \at{t}}$. Therefore,
\begin{align*}
\lim_{i\ra\infty} \zuphsts{\useq{i}{u}}{t}{\hst{h}{t}, \at{t}} = \zuphsts{u}{t}{\hst{h}{t}, \at{t}}.\numberthis\label{eq:step2}
\end{align*}
From \eqref{eq:step1} and \eqref{eq:step2}, it follows that
\begin{align*}
\liminf_{i\ra \infty} \fullccosts{\useq{i}{u}} \ge \fullccosts{u},
\end{align*}
which establishes the lower semi-continuity of $\fullccosts{u}$. 
\end{proof}

%% file: Sections/history_embedding.tex
\section{Planning Using Common-Information Approach and Approximate-Information States}\label{sec:history_embedding}
Theorem \ref{thm:strongduality} provides firm theoretical support for primal-dual type planning and learning algorithms for a given MA-C-POMDP. Indeed, given the optimal Lagrange-multipliers vector $\lambda^\star$, MA-C-POMDP simply reduces to a MA-POMDP\footnote{Except the fact that randomization amongst equally valuable actions cannot be ignored, in general.}, so essentially all MA-POMPD algorithms apply for gradient descent in the primal space $\uspace$. However, one must find 
$\lambda^\star$. In light of the first inequality in \eqref{eq:saddlepointconditions}, we can do this by a projected gradient ascent in the dual space $\mcl{Y}$ -- on a slower time-scale so that it sees the minimization over the primal space $\uspace$ as having essentially equilibrated. In this section, we will assume that Assumption \ref{assmp:boundedcosts} ((a)-(c)) holds.

In this section, we shall first review the common-information approach \cite{nayyar13,nayyar14} that transforms a given MA-POMDP into an equivalent SA-POMDP. We will then use insights from the resulting SA-POMDP in order to derive a compression-framework for approximately solving the original MA-C-POMDP. This framework will be an extension of \cite{hsu22} the details of which will be left as an exercise for the reader. Nevertheless, an important goal will be achieved via this exercise: the approximation criteria of the compression-framework will be independent of the Lagrange-multipliers vector $\lambda$. This property will be essential in the learning context where we would like the learning of the compression-mapping to be independent of $\lambda$ (since $\lambda$ needs to be learned as well). Note that if we, instead, directly followed the approach of \cite{hsu22}, then for each value of the Lagrange-multipliers vector $\lambda$, we would have to find a new compression-mapping, and then adapt it as $\lambda$ is changed.

To achieve optimality gaps for the said compression-framework, we will first consider \eqref{eq:macpomdp} over a finite-horizon $T<\infty$, and then (with the aid of Assumption \ref{assmp:boundedcosts}) let $T$ go to infinity. Before, we proceed further, we present a simple lemma that uses Assumption \ref{assmp:slatercondition} to get an upper-bound on $\lambda^\star$ -- the existence and search of which can therefore be restricted to a compact cube in $\mbb{R}^K$. (This shall enable us to get $\lambda$-independent optimality-gaps for the compression-framework: see Remark \ref{rem:universalbound_optimalitygap}).  

\begin{lem}\label{lem:upperbound_lambda}
Under Assumptions \ref{assmp:boundedcosts}(a) and \ref{assmp:slatercondition}, the optimal Lagrange-multipliers vector $\lambda^\star$ is upper-bounded as follows:
\begin{align*}
\| \lambda^{\star} \|_{\infty} \le \| \lambda^{\star} \|_1 \le \frac{1}{\zeta} \l(  \fullccosts{\ov{u}} - \frac{\udl{c}}{1-\alpha} \r).
\end{align*}
\end{lem}

\begin{proof}
We have the following sequence of inequalities:
\begin{align*}
\sum_{t=1}^{\infty} \alpha^{t-1} \udl{c} &\labelrel{\le}{eqr:lamda:assmp:boundedcosts}  \optcosts = \lags{u^\star}{\lambda^\star} \\
&\hspace{-0pt} \labelrel{\le}{eqr:lamda:saddlepoint} \lags{\ov{u}}{\lambda^\star}\\
&\hspace{-0pt} = \fullccosts{\ov{u}} + \dotp{\lambda^\star}{\fulldcosts{\ov{u}}-\constraintv}\\
&\hspace{-0pt} \labelrel{\le}{eqr:lamda:assmp:slater} \fullccosts{\ov{u}} + \dotp{\lambda^\star}{-\zeta1}\\
&\hspace{-0pt} \labelrel{=}{eqr:lamda:positive} \fullccosts{\ov{u}} - \zeta \|\lambda^\star\|_1 \\
&\hspace{-0pt}\le \fullccosts{\ov{u}} - \zeta \|\lambda^\star\|_{\infty}.
\end{align*}
Here, \eqref{eqr:lamda:assmp:boundedcosts} uses Assumption \ref{assmp:boundedcosts}; \eqref{eqr:lamda:saddlepoint} uses the second inequality in \eqref{eq:saddlepointconditions}; \eqref{eqr:lamda:assmp:slater} uses Assumption \ref{assmp:slatercondition}; and \eqref{eqr:lamda:positive} uses the fact that $\lambda^{\star}$ is non-negative. 
\end{proof}

\subsection{Common-Information Approach}
The common-information approach \cite{nayyar13,nayyar14} shows that a (cooperative) MA-POMDP can be converted into an equivalent SA-POMDP---called the \textit{coordinated-system}. In this system, the single agent, called the \emph{coordinator}, is a virtual entity that has access to the common observations $\Otn{t}{0}$, and does not get to see the agents' private observations and actions $( \Otn{t}{1:N}, \Atn{t}{1:N} ) \setminus \Otn{t}{0}$. Therefore, from the perspective of the coordinator, the unknown state is the environment's state combined with the private histories of all agents, i.e., $( \Stt{t}, \Hstn{t}{1:N} )$. 

At time $t \in \mbb{N}$, via a \emph{coordination policy} (to be defined later), the coordinator uses the common-history $\Hstn{t}{0}$ to \udl{deterministically} chooses an action $ \Gt{t} = \Gtn{t}{1:N} $\footnote{As $\Gt{t}$ depends deterministically on the common-history, all agents can replicate it with consensus.}, where $\Gtn{t}{n}$ maps $ \hstnspace{t}{n}$ to $ \m{\anspace{n}} $ and is designed to be an enforcing \textit{prescription} for agent-$n$---agent-$n$ applies $\Gtn{t}{n}$ to its private history $\Hstn{t}{n}$ and then draws its action $\Atn{t}{n} \sim \Gtn{t}{n} ( \Hstn{t}{n})$. We use $\gtspace{t}$ to denote the set of all possible prescriptions at time $t$, i.e., $\gtspace{t} = \prod_{n=1}^{N} \gtnspace{t}{n}$ and $\gtnspace{t}{n} \defeq \{ \gtn{t}{n} : \hstnspace{t}{n} \ra \m{\anspace{n}} \}$. We note that $\gtspace{t}$ is in one-to-one correspondence with the set
\begin{align*}
    \xtspace{\gtspace{t}} \defeq \prod_{n=1}^{N} \prod_{ \hstn{h}{t}{n} \in \hstnspace{t}{n} } \m{\anspace{n}; \hstn{h}{t}{n} }.\numberthis\label{eq:xgtspace}
\end{align*}
which is a compact (and convex) space by Tychonoff's theorem (see Proposition \ref{prop:tychonoff}), and is also metrizable (see Proposition \ref{prop:metrizability}). Henceforth, we will assume $\gtspace{t}$ to have the same topology as $\xtspace{\gtspace{t}}$.
\begin{rem}
To help achieve equivalence between the coordinated system and MA-C-POMDP (which uses decentralized policy-profiles only), we have restricted the coordinator to choose prescriptions in a deterministic manner---no randomization over the elements of 
$\xtspace{\gtspace{t}}$'s. 
\end{rem}
Let $\wtHstn{t}{0}$ denote the prescription-observation history of the coordinator, i.e.,
\begin{align}
\begin{split}\label{eq:aoh:coordinator}
\wtHstn{1}{0} &\defeq \Otn{1}{0} \text{ and } \\
\wtHstn{t}{0} &\defeq \l( \wtHstn{t-1}{0}, \Gt{t-1}, \Otn{t}{0} \r) \text{ for all } t\in[2,\infty]_{\mbb{Z}}.
\end{split}
\end{align}
Then we can define a coordination policy $v$ as a tuple $\ut{v}{1:\infty} \in \vvspace$ where $\ut{v}{t}$ maps $\wthstnspace{t}{0}$ to $\gtspace{t}$ and where agent $n$ draws its action $\Atn{t}{n}$ according to the distribution $\Gtn{t}{n}( \Hstn{t}{n} ) = ( \ut{v}{t}( \wtHstn{t}{0} ) )^{(n)}( \Hstn{t}{n} )$. 
Note that in this setup, each $\wtHstn{t}{0}$ is some deterministic function of $\Hstn{t}{0}$. Therefore, we can replace $\ut{v}{t} ( \wtHstn{t}{0} )$ by $\ut{\wt{v}}{t} ( \Hstn{t}{0} )$. One then establishes equivalence between the coordinated system and MA-C-POMDP by setting 
\begin{align*}
\utn{u}{t}{n} \l( \hstn{h}{t}{0}, \hstn{h}{t}{n} \r) 
&= \l( \ut{v}{t} \l(  \hstn{\wt{h}}{t}{0} \r) \r)^{(n)} \l( \hstn{h}{t}{n} \r) \\
&= \l( \wt{v}_t  \l(  \hstn{h}{t}{0} \r) \r)^{(n)} \l( \hstn{h}{t}{n} \r).
\end{align*}
In light of the above equivalence and the strong duality result of Theorem \ref{thm:strongduality}, from now onward, we will restrict our focus to deriving optimal and approximately optimal coordination policies for coordinated-systems whose immediate-costs are parametrized by 
$\lambda\in\mcl{Y}$, namely $\un{l}{\lambda}: \sspace \times \aspace \ra \mbb{R}$, where
\begin{align*}
\un{l}{\lambda}\l(s,a\r) &\defeq c\l(s,a\r) + \dotp{\lambda}{d\l(s,a\r) -\constraintv}.\numberthis\label{eq:l_lamda}
\end{align*}
Also, in line with our aforementioned plan, we parametrize the aggregate discounted costs by horizons $T \in \mbb{N} \cup \{ \infty \}$, namely $L_T: \vvspace \times \mcl{Y} \ra \mbb{R}$, where
\begin{align*}
    L_T \l( v, \lambda \r) = L_T^{(P_1, \alpha)} \l( v, \lambda \r) &\defeq \E{v}{P_1} \l[ \sum_{t=1}^{T} \alpha^{t-1} \lCost \r].\numberthis\label{eq:L_T}
\end{align*}
With the above setup in place, for a (finite) horizon $T \in \mbb{N}$ and a fixed Lagrange-multiplier $\lambda\in\mcl{Y}$, one has the following dynamic program (Algorithm \ref{alg:cidecomposition}) to find the optimal coordination policy for the objective $L_T$ (see \cite{nayyar13,nayyar14}). We note that the usage of $\min$ and $\argmin$ in \eqref{eq:V_tT_lamda} and \eqref{eq:v_tT_lamda_star} is justified due to compactness of $\gtspace{t}$, and the choice of prescription in \eqref{eq:v_tT_lamda_star} is arbitrary. 

\begin{algorithm}[H]
\DontPrintSemicolon
\KwInput{Time-horizon $T \in \mbb{N}$, Discount-factor $\alpha\in(0,1]$, MA-C-POMDP Model (see Section \ref{sec:problem}).}
\Parameter{Lagrange-multipliers vector $\lambda \in \mcl{Y}$.}
\KwOutput{$\utn{v}{1:T}{\lambda,\star}$ determined by $\l\{ \Qtl{t,T}{\lambda} : \wthstnspace{t}{0} \times \gtspace{t} \ra \mbb{R} \r\}_{t=1}^{T}$.
}

\nonl $V_{T+1,T}^{(\lambda)} \equiv 0$.

\nonl \For{$t = T, T-1, \dots, 1$}{
\nonl \begin{align*}
&\Qtl{t,T}{\lambda} \l( \hstn{\wt{h}}{t}{0}, \gt{t} \r) 
= \mbb{E} \l[ \lCost \r. \\
&\hspace{10pt} \l. 
+ \alpha \Vtl{t+1,T}{\lambda} \l( \wtHstn{t+1}{0} \r)  \mid \wtHstn{t}{0} = \hstn{\wt{h}}{t}{0}, \Gt{t} = \gt{t} \r].\numberthis\label{eq:Q_tT_lamda}\\
&\Vtl{t,T}{\lambda} \l( \hstn{\wt{h}}{t}{0} \r) = \min_{\gt{t} \in \gtspace{t} } \Qtl{t,T}{\lambda} \l( \hstn{\wt{h}}{t}{0}, \gt{t} \r).\numberthis\label{eq:V_tT_lamda}\\
&\utn{v}{t}{\lambda, \star} \l( \hstn{\wt{h}}{t}{0} \r) \in \argmin_{\gt{t} \in \gtspace{t}} \Qtl{t,T}{\lambda} \l( \hstn{\wt{h}}{t}{0}, \gt{t} \r).\numberthis\label{eq:v_tT_lamda_star}
\end{align*}
}
\caption{Dynamic Programming with Full Common and Private Histories in Finite-Horizon setting.}\label{alg:cidecomposition}
\end{algorithm}
\begin{rem}\label{rem:necessaryonly}
    Akin to SA-C-MDP and SA-C-POMDPs, even when $\lambda^\star$ is known, finding an optimal policy for the unconstrained objective $L_T\l( \cdot, \lambda^*\r)$ does not necessarily imply solving the (finite-horizon version of the) original constrained optimization problem---because the coordination policy $\utn{v}{1:T}{\lambda, \star} $ obtained from \eqref{eq:v_tT_lamda_star} may not satisfy the constraints. However, from Theorem \ref{thm:strongduality}, we are guaranteed that an optimal coordination policy exists and it is also true that any such policy must choose a prescription from the set in the right-hand-side of \eqref{eq:v_tT_lamda_star}. It is hard to characterize how the optimal policy randomizes between these prescriptions, thus our choice of arbitrary selection in \eqref{eq:v_tT_lamda_star}. Having said that, this issue shall remain somewhat innocuous in the learning context where $\lambda$ will be continuously updated based on constraint violations. A similar remark would apply to the (approximate) dynamic program in Algorithm \ref{alg:dp_approximatestate}.
\end{rem}

From dynamic programming theory, it is known that $\Vtl{t,T}{\lambda} : \wthstnspace{t}{0} \ra \mbb{R}$ (see \eqref{eq:V_tT_lamda}) satisfies 
\begin{align*}
\Vtl{t,T}{\lambda} \l( \hstn{\wt{h}}{t}{0} \r) = \inf_{v\in\vvspace} \E{v}{P_1} \l[ \sum_{\tau=t}^{T} \alpha^{\tau - t} \un{l}{\lambda}\l(\Stt{\tau}, \At{\tau}\r) \Big| \wtHstn{t}{0} = \hstn{\wt{h}}{t}{0} \r].
\end{align*}
Therefore, the (finite-horizon) coordination policy $\utn{v}{1:T}{\lambda, \star}$ (see \eqref{eq:v_tT_lamda_star}) minimizes the (finite-horizon) objective $L_T$, i.e.,
\begin{align*}
    L_T\l( \utn{v}{1:T}{\lambda, \star}, \lambda \r) &= \inf_{v\in\vvspace} L_T\l( v, \lambda \r).
\end{align*}
\begin{rem}
The first argument of $L_T$, by definition (see \eqref{eq:L_T}), should be an element of $\vvspace$. However, since the specification of the policy for times $T+1$ and onward does not matter, the above equation is consistent (with slight abuse of notation). 
\end{rem}

Using Assumption \ref{assmp:boundedcosts}, we can now compare $\Vtl{t,T}{\lambda}$ (the optimal performance on a finite-horizon $T\ge t$) with the optimal performance on the infinite-horizon. Let us define value-functions $\{ \Vtl{t}{\lambda} : \wthstnspace{t}{0} \ra \mbb{R} \}_{t=1}^{\infty}$,
\begin{align*}
&\Vtl{t}{\lambda} \l( \hstn{\wt{h}}{t}{0} \r) \\
&\defeq \inf_{v\in \vvspace} 
\E{v}{P_1} \l[ \sum_{\tau=t}^{\infty} \alpha^{\tau-t} \un{l}{\lambda}\l(\Stt{\tau}, \At{\tau}\r) \Big| \wtHstn{t}{0} = \hstn{\wt{h}}{t}{0} \r],\numberthis\label{eq:V_t_lamda}
\end{align*}
and the corresponding prescription-value-functions $\{ \Qtl{t}{\lambda} : \wthstnspace{t}{0} \times \gtspace{t} \ra \mbb{R} \}_{t=1}^{\infty}$,
\begin{align*}
&\Qtl{t}{\lambda} \l( \hstn{\wt{h}}{t}{0}, \gt{t} \r) \defeq \mbb{E}_{P_1} \l[ \un{l}{\lambda}\l(\Stt{t}, \At{t}\r) \r. \\
&\hspace{10pt} \l. + \alpha \Vtl{t+1}{\lambda} \l( \wtHstn{t+1}{0} \r) \Big| \wtHstn{t}{0} = \hstn{\wt{h}}{t}{0}, \Gt{t} = \gt{t} \r].\numberthis\label{eq:Q_t_lamda}
\end{align*}
Then, the following bound on $\Vtl{t}{\lambda}$ with respect to $\Vtl{t,T}{\lambda}$ ($T\ge t$) holds.
\begin{prop}\label{prop:VtT_vs_Vt}
    Fix $\lambda\in\mcl{Y}$, $\alpha\in(0,1)$, and $t \in \mbb{N}$. Suppose Assumption \ref{assmp:boundedcosts} holds and consider a (finite) horizon $T\in [t, \cdot]$. Then, for any $\hstn{\wt{h}}{t}{0} \in \wthstnspace{t}{0}$, the following relation holds between $\Vtl{t,T}{\lambda} \l( \hstn{\wt{h}}{t}{0} \r) $ and $\Vtl{t}{\lambda} $.
    \begin{align*}
        &\Vtl{t,T}{\lambda} \l( \hstn{\wt{h}}{t}{0} \r) + \frac{\alpha^{T-t+1}}{1-\alpha} \un{\udl{l}}{\lambda} \le \Vtl{t}{\lambda} \l( \hstn{\wt{h}}{t}{0} \r) \\
        &\hspace{20pt} \le \Vtl{t,T}{\lambda} \l( \hstn{\wt{h}}{t}{0} \r) + \frac{\alpha^{T-t+1}}{1-\alpha} \un{\ov{l}}{\lambda},
    \end{align*}
    where
    \begin{align*}\numberthis\label{eq:lower_and_upper_lcost}
        \un{\udl{l}}{\lambda} \defeq \udl{c} + \dotp{\lambda}{\udl{d} - \constraintv } \text{ and }
        \un{\ov{l}}{\lambda} \defeq \ov{c} + \dotp{\lambda}{\ov{d} - \constraintv }
    \end{align*}
\end{prop}
\begin{proof}
    The proof can be established by backward induction. For brevity, it is omitted.
\end{proof}
With Proposition \ref{prop:VtT_vs_Vt}, we have a bound on the gap between $\Vtl{t}{\lambda}$ and $\Vtl{t,T}{\lambda}$ ($T\ge t$), that decays exponentially with $T$ (due to discounting). Therefore, for the time being, let us focus on the finite-horizon setup for which an optimal coordination policy can be found by Algorithm \ref{alg:cidecomposition}. There are, however, two key issues with Algorithm \ref{alg:cidecomposition}. Firstly, the domain of $\wtHstn{t}{0}$ grows exponentially in time and while one may compress $\wtHstn{t}{0}$ to a belief-state $\Pi_t \defeq \pr_{P_1} ( \Stt{t}, \Hstn{t}{1:N} \mid \wtHstn{t}{0} )  $ (without loss of optimality), 
the update of $\Pi_t$ requires knowledge of $\mcl{P}_{tr}$ (the transition-law) and $P_1$ (joint distribution of initial state and initial joint-observation) which are not available in the learning context. Secondly, and more importantly, the domain of the unobserved state $\Stt{t}, \Hstn{t}{1:N}$ grows exponentially with time and the number of agents---which leads to doubly-exponential growth of the coordinator's prescriptions. Due to these issues, Algorithm \ref{alg:cidecomposition} is generally not implementable, and thus remains conceptual. 

One means to address the above challenges is to compress the increasing common and private histories in such a way that when policies are restricted to take actions based on the compressed images of these histories, the performance is still approximately good. This is the goal in \cite{hsu22} which proposes one such compression-framework, and then characterizes its optimality-gap (as a function of the compression attributes). In the remainder of the section, we will extend the notions of \cite{hsu22} to the constrained setting, and as mentioned earlier, we will do so in a manner that the learning of the compression mapping shall remain independent of the Lagrange-multipliers vector $\lambda$. Also, it will be helpful to keep in mind that our ultimate goal will be to get a bound on the gap between $\Vtl{t,T}{\lambda}$ and $\whVtl{t,T}{\lambda}$ where $\whVtl{t,T}{\lambda}$ will be the optimal cost-to-go when coordination policies are restricted to use a prespecified compression framework.

\subsection{Approximate State Representations (Finite-Horizon)}
The framework in \cite{hsu22} uses a three-step process. It first compresses private histories from $\Hstn{t}{1:N} $ to an \textit{approximate sufficient private state (ASPS)}, which we will denote by $\Zhtn{t}{1:N} 
$. Compressing private histories to ASPS then changes the set of coordinator's prescriptions to a reduced set, the elements of which we will denote by $\lamdaht{t}$. Finally, the coordinator's prescription-observation history (now with the set of reduced prescriptions) is compressed from $\wtHstn{t}{0}$ to an \textit{approximate sufficient common state (ASCS)}, which we will denote by $\Zhtn{t}{0}$. With these three steps in mind, we now give formal definitions of ASPS and ASCS. (Note that ASCS relies on having an ASPS beforehand).

\begin{dfn}[Finite-Horizon ASPS Generator]\label{dfn:asps_generator_finite_horizon}
Let $T\in\mbb{N}$ and let $\zhtnspace{1:T}{1:N}$ be a pre-specified collection of topological spaces.\footnote{
We can assume each $\zhtnspace{t}{n}$ is a Euclidean space of some fixed dimension that can vary with $t$ and $n$, but a time-invariant domain is best in practice.} A collection $ \varthetahtn{1:T}{1:N} $ of compression-functions where $\varthetahtn{t}{n}$ maps $\hstnspace{t}{0} \times \hstnspace{t}{n}$ to $\zhtnspace{t}{n}$ is called a $T$-horizon $\l(\eps_{p,1}, \eps_{p,2}, \delta_p \r)$-ASPS generator if the process $\{ \Zhtn{t}{1:N} \}_{t=1}^{T}$, with $\Zhtn{t}{n} = \varthetahtn{t}{n}( \Hstn{t}{0}, \Hstn{t}{n} ) $ almost-surely, satisfies the following properties.
\begin{itemize}
\item[] \textbf{(ASPS1)} It evolves in a recursive manner: for each $n \in [N] $,
\begin{align*}
\Zhtn{t}{n}=\phihtn{t}{n} \l( \Zhtn{t-1}{n}, 
\Otn{t}{0}, \l( \Otn{t}{n}, \Atn{t-1}{n}\r)\setminus \Otn{t}{0} \r).
\end{align*}

\item[] \textbf{(ASPS2.1)} It suffices for approximate prediction of objective-cost: for all $\l(\hstn{\wt{h}}{t}{0}, \hstn{h}{t}{1:N}, \at{t}\r) \in \wthstnspace{t}{0} \times \prod_{n=1}^{N} \hstnspace{t}{n} \times \aspace$,
\begin{align*}
&\l| \mbb{E}_{P_1} \l[  c\l( \Stt{t}, \At{t} \r) \Big| \hstn{\wt{h}}{t}{0}, \hstn{h}{t}{1:N}, \at{t} \r] - \r. \\
&\hspace{40pt} \l. \mbb{E}_{P_1} \l[ c\l( \Stt{t}, \At{t} \r) \Big| \hstn{\wt{h}}{t}{0}, \zhtn{t}{1:N}, \at{t} \r]  \r| \le \frac{\eps_{p,1}}{4}.
\end{align*}

\item[] \textbf{(ASPS2.2)} It suffices for approximate prediction of constraint cost: for all $\l(\hstn{\wt{h}}{t}{0}, \hstn{h}{t}{1:N}, \at{t}\r) \in \wthstnspace{t}{0} \times \prod_{n=1}^{N} \hstnspace{t}{n} \times \aspace$,
\begin{align*}
&\l\| \mbb{E}_{P_1} \l[  d\l( \Stt{t}, \At{t} \r) \Big| \hstn{\wt{h}}{t}{0}, \hstn{h}{t}{1:N}, \at{t} \r] - \r. \\
&\hspace{40pt} \l. \mbb{E}_{P_1} \l[ d\l( \Stt{t}, \At{t} \r) \Big| \hstn{\wt{h}}{t}{0}, \zhtn{t}{1:N}, \at{t} \r]  \r\|_{\infty} \le \frac{\eps_{p,2}}{4}.
\end{align*}

\item[] \textbf{(ASPS3)} It suffices for approximate prediction of observations: for all $\l(\hstn{\wt{h}}{t}{0}, \hstn{h}{t}{1:N}, \at{t}\r) \in \wthstnspace{t}{0} \times \prod_{n=1}^{N} \hstnspace{t}{n} \times \aspace$,
\begin{align*}
&\kappa \l( 
\mbb{P}_{P_1} \l( \Otn{t+1}{0:N} \mid \hstn{\wt{h}}{t}{0}, \hstn{h}{t}{1:N}, \at{t} \r), \r. \\ 
&\hspace{30pt} \l. \mbb{P}_{P_1} \l( \Otn{t}{0:N} \mid \hstn{\wt{h}}{t}{0}, \zhtn{t}{1:N}, \at{t} \r) 
\r)  \leq \frac{\delta_p}{8},
\end{align*}
where $\kappa\l(\cdot, \star \r)$ is the total variation distance between the two probability measures.
\end{itemize}
\end{dfn}

Let us denote the range of $\varthetahtn{t}{n}$ by $\zhtnspaces{t}{n}$, i.e., $\zhtnspaces{t}{n} \defeq \varthetahtn{t}{n} \l( \hstnspace{t}{0} \times \hstnspace{t}{n} \r)$. Then the above definition induces a compression in the coordinator's prescriptions from $\Gt{t}$ to $\Lamdaht{t} = \Lamdahtn{t}{1:N}$ where $\Lamdahtn{t}{n}$ maps $\zhtnspaces{t}{n}$ (instead of $\hstnspace{t}{n}$) to $\m{\anspace{n}}$. We use $\lamdahtspace{t}$ to denote the set of all possible reduced prescriptions at time $t$, i.e., $\lamdahtspace{t} = \lamdahtnspace{t}{1:N}$ where $ \lamdahtnspace{t}{n} \defeq \{ \lamdahtn{t}{n} : \zhtnspaces{t}{n} \mapsto \m{\anspace{n}} \}$. We note that $\lamdahtspace{t}$ is in one-to-one correspondence with the set
\begin{align*}
    \xtspace{\lamdahtspace{t}} \defeq \prod_{n=1}^{N} \prod_{ \zhtn{t}{n} \in \zhtnspaces{t}{n} } \m{\anspace{n}; \zhtn{t}{n}}.\numberthis\label{eq:xlamdahtspace}
\end{align*}
which is a compact (and convex) space by Tychonoff's theorem (see Proposition \ref{prop:tychonoff}), and is also metrizable. From hereon, we will assume $\lamdahtspace{t}$ to have the same topology as $\xtspace{\lamdahtspace{t}}$.

Having detailed the compression of \textit{i}) private histories to ASPS, and \textit{ii}) private-history based prescriptions to ASPS-based prescriptions, we now proceed to formally characterizing the compression of the common history to ASCS.

\begin{dfn}[Finite-Horizon ASCS Generator]\label{dfn:ascs_generator_finite_horizon}
Let $T\in\mbb{N}$ and let $\zhtnspace{1:T}{0}$ be a collection of topological spaces.\footnote{
For our purposes, we can assume each $\zhtnspace{t}{0}$ is a Euclidean space of some fixed dimension that can vary with $t$.} For a given $T$-horizon $\l(\eps_{p,1}, \eps_{p,2}, \delta_p \r)$-ASPS generator, a collection $ \varthetahtn{1:T}{0} $ of compression-functions where $\varthetahtn{t}{0}$ maps $\wthstnspace{t}{0}$ to $\zhtnspace{t}{0}$ is called the corresponding $T$-horizon $\l(\eps_{c,1}, \eps_{c,2}, \delta_c \r)$-ASCS generator if the process $\{ \Zhtn{t}{0} \}_{t=1}^{T}$, with $\Zhtn{t}{0} = \varthetahtn{t}{0}( \wtHstn{t}{0} ) $ almost-surely, satisfies the following properties.
\begin{enumerate}
\item \textbf{(ASCS1)} It evolves in a recursive manner:
\begin{align*}
\Zhtn{t}{0}=\phihtn{t}{0} \l(\Zhtn{t-1}{0}, \Lamdaht{t-1}, \Otn{t}{0} \r).
\end{align*}

\item \textbf{(ASCS2.1)} It suffices for approximate prediction of objective-cost: for all $\l(\hstn{\wt{h}}{t}{0}, \lamdaht{t} \r) \in \wthstnspace{t}{0} \times \lamdahtspace{t}$,
\begin{align*}
&\l| \mbb{E}_{P_1} \l[  c\l( \Stt{t}, \At{t} \r) \Big| \hstn{\wt{h}}{t}{0}, \lamdaht{t} \r] - \r. \\
&\hspace{40pt} \l. \mbb{E}_{P_1} \l[ c\l( \Stt{t}, \At{t} \r) \Big| \zhtn{t}{0}, \lamdaht{t} \r]  \r| \le \frac{\eps_{c,1}}{4}.
\end{align*}

\item \textbf{(ASCS2.2)} It suffices for approximate prediction of constraint-cost: for all $\l(\hstn{\wt{h}}{t}{0}, \lamdaht{t} \r) \in \wthstnspace{t}{0} \times \lamdahtspace{t}$,
\begin{align*}
&\l\| \mbb{E}_{P_1} \l[  d\l( \Stt{t}, \At{t} \r) \Big| \hstn{\wt{h}}{t}{0}, \lamdaht{t} \r] - \r. \\
&\hspace{40pt} \l. \mbb{E}_{P_1} \l[ d\l( \Stt{t}, \At{t} \r) \Big| \zhtn{t}{0}, \lamdaht{t} \r]  \r\|_{\infty} \le \frac{\eps_{c,2}}{4}.
\end{align*}

\item \textbf{(ASCS3)} It suffices for approximate prediction of common-observations: for all $\l(\hstn{\wt{h}}{t}{0}, \lamdaht{t} \r) \in \wthstnspace{t}{0} \times \lamdahtspace{t}$,
\begin{align*}
&\kappa \l( 
\mbb{P}_{P_1} \l( \Otn{t+1}{0} \mid \hstn{\wt{h}}{t}{0}, \lamdaht{t} \r), \mbb{P}_{P_1} \l( \Otn{t}{0} \mid \zhtn{t}{0}, \lamdaht{t} \r) 
\r)  \leq \frac{\delta_c}{8}.
\end{align*}
\end{enumerate}
\end{dfn}
\begin{rem}
ASPS-1 and ASCS-1 are important for designing implementable algorithms; the recursive nature of $\Zhtn{t}{n}$'s obviates the requirement of storing the full histories.
\end{rem}

Let us denote the range of $\varthetahtn{t}{0}$ by $\zhtnspaces{t}{0}$, i.e., $\zhtnspaces{t}{0} \defeq \varthetahtn{t}{0} \l( \wthstnspace{t}{0} \r)$. Then combining ASCS with ASPS (and ASPS-based prescriptions) leads to a reduction in the policy search space of the coordinator from $\vvspace$ to $\wh{\vvspace}( \varthetahtn{1:T}{0:N} )$. A typical element $\wh{v}$ of $\wh{\vvspace}$ is a tuple $\wh{v}_{1:\infty}$, where for all $t \in [T]$, $\wh{v}_t$ maps $\zhtnspaces{t}{0}$ to $\lamdahtspace{t}$, such that to take action $\Atn{t}{n}$, agent $n$ uses the distribution $$\Lamdahtn{t}{n}\l(\Zhtn{t}{n} \r) = \l[ \wh{v}_t \l(\Zhtn{t}{0}\r) \r]^{(n)} \l( \Zhtn{t}{n} \r), $$ 
which is almost-surely the same as 
$$ \l[ \wh{v}_t \l( \varthetahtn{t}{0}\l(\wtHstn{t}{0}\r) \r) \r]^{(n)} \l( \varthetahtn{t}{n} \l( \Hstn{t}{0}, \Hstn{t}{n} \r)  \r).$$

Given $T$-horizon ASPS and ASCS generators $\varthetahtn{1:T}{0:N}$, we can find an optimal-policy in $\wh{\vvspace}$ (thus approximately optimal in $\vvspace$) using the (approximate) dynamic program in Algorithm \ref{alg:dp_approximatestate}.

\begin{algorithm}[H]
\DontPrintSemicolon
\KwInput{Time-horizon $T \in \mbb{N}$, Discount-factor $\alpha\in(0,1]$, MA-C-POMDP Model (see Section \ref{sec:problem}).}
\Parameter{Lagrange-multipliers vector $\lambda \in \mcl{Y}$.}
\KwOutput{$\utn{\wh{v}}{1:T}{\lambda,\star}$ determined by $\l\{ \whQtl{t,T}{\lambda} : \zhtnspaces{t}{0} \times \lamdahtspace{t} \ra \mbb{R} \r\}_{t=1}^{T}$.
}

\nonl $\wh{V}_{T+1,T}^{(\lambda)} \equiv 0$.

\nonl \For{$t = T, T-1, \dots, 1$}{
\nonl \begin{align*}
&\whQtl{t,T}{\lambda}\l(\zhtn{t}{0}, \lamdaht{t}\r) = 
\mbb{E} \l[
\lCost \r. \\
&\hspace{10pt} \l. + \alpha \whVtl{t+1,T}{\lambda}  \l( \Zhtn{t+1}{0} \r)  \mid \Zhtn{t}{0} = \zhtn{t}{0}, \Lamdaht{t} = \lamdaht{t} \r].\numberthis\label{eq:whQ_tT_lamda}\\
&\whVtl{t,T}{\lambda}\l( \zhtn{t}{0} \r) = \min_{\lamdaht{t} \in \lamdahtspace{t} } \whQtl{t,T}{\lambda}\l( \zhtn{t}{0}, \lamdaht{t} \r)\numberthis\label{eq:whV_tT_lamda}\\
&\utn{\wh{v}}{t}{\lambda,\star} \l(\zhtn{t}{0} \r) = \argmin_{\lamdaht{t} \in \lamdahtspace{t}} \whQtl{t,T}{\lambda}\l( \zhtn{t}{0}, \lamdaht{t} \r).\numberthis\label{eq:whv_tT_lamda_star}
\end{align*}
}
\caption{
Dynamic Programming with Compressed Common and Private Histories in Finite-Horizon setting.}\label{alg:dp_approximatestate}
\end{algorithm}
Like before, the usage of $\min$ and $\argmin$ in Algorithm \ref{alg:dp_approximatestate} is justified due to compactness of $\lamdahtspace{t}$, and the choice of the reduced prescription in \eqref{eq:whv_tT_lamda_star} is arbitrary. Since ``$\Zhtn{t}{0} = \varthetahtn{t}{0} \l( \wtHstn{t}{0}  \r)$'' and ``$\Zhtn{t}{n} = \varthetahtn{t}{n} \l( \Hstn{t}{0}, \Hstn{t}{n} \r) \forall\  n\in[N]$'' hold almost-surely, from standard results in dynamic programming theory, it follows that $\whVtl{t,T}{\lambda} : \wthstnspace{t}{0} \ra \mbb{R} $ (see \eqref{eq:whV_tT_lamda}) satisfies 
\begin{align*}
&\whVtl{t,T}{\lambda} \l( \varthetahtn{t}{0} \l( \hstn{\wt{h}}{t}{0} \r) \r) \\
&\hspace{10pt} = \inf_{v\in\wh{\vvspace}} \E{\wh{v}}{P_1} \l[ \sum_{\tau=t}^{T} \alpha^{\tau - t} \lCost \Big| \wtHstn{t}{0} = \hstn{\wt{h}}{t}{0} \r].
\end{align*}
Therefore, the (finite-horizon) coordination policy $\utn{\wh{v}}{1:T}{\lambda, \star}$ (see \eqref{eq:whv_tT_lamda_star}) minimizes the (finite-horizon) objective \eqref{eq:L_T} amongst all coordination policies that are restricted to draw actions based on ASCS, ASPS, and ASPS-based prescriptions over the horizon $[T]$, i.e.,
\begin{align*}
    L_T\l( \utn{\wh{v}}{1:T}{\lambda, \star}, \lambda \r) &= \inf_{\wh{v}\in\wh{\vvspace}( \varthetahtn{1:T}{0:N} )} L_T\l( v, \lambda \r).
\end{align*}
A natural question is to estimate the finite-horizon optimality gap of the \emph{ASPS-ASCS coordination policy} obtained from Algorithm \ref{alg:dp_approximatestate} from the one obtained via Algorithm \ref{alg:cidecomposition}.

\begin{prop}[Optimality Gap for Finite-Horizon]\label{prop:optimalitygap_T}
Fix $ \lambda \in \mcl{Y}$, $\alpha\in(0,1]$, $T \in \mbb{N}$,  
and $T$-horizon ASPS and ASCS generators $\varthetahtn{1:T}{0:N}$ (see Definitions \ref{dfn:asps_generator_finite_horizon} and \ref{dfn:ascs_generator_finite_horizon}). 
Suppose that Assumption \ref{assmp:boundedcosts} holds. Then, for any $t\in[T]$ and any $\hstn{\wt{h}}{t}{0} \in \wthstnspace{t}{0}$ with $\gt{t}^\star \in \argmin_{\gt{t} \in \gtspace{t} } \Qtl{t,T}{\lambda} \l( \hstn{\wt{h}}{t}{0}, \gt{t} \r)$, there exists $\lamdaht{t} \in \lamdahtspace{t}$ such that
\begin{align*}
&\whQtl{t,T}{\lambda} \l( \varthetahtn{t}{0}\l( \hstn{\wt{h}}{t}{0} \r), \lamdaht{t} \r) - \Qtl{t,T}{\lambda}\l( \hstn{\wt{h}}{t}{0}, \gt{t}^\star \r) \\
&\hspace{80pt} \le M_c\l(t; \alpha, T \r) + M_p \l(t; \alpha, T \r),\\
&\whVtl{t,T}{\lambda} \l( \varthetahtn{t}{0}\l( \hstn{\wt{h}}{t}{0} \r) \r) - \Vtl{t,T}{\lambda}\l( \hstn{\wt{h}}{t}{0} \r) \\
&\hspace{80pt} \le M_c\l(t; \alpha, T \r) + M_p\l(t; \alpha, T \r),
\end{align*}
where,
\begin{align*}
&M_c \l(t; \alpha, T \r) = M_c^{(\eps_{c,1}, \eps_{c,2}, \delta_c, \lambda, \ulbar{c}, \ulbar{d}, 
\constraintv)} \l( t; \alpha, T \r) \\
&\hspace{10pt} \defeq \l( \eps_{c,1} + \|\lambda\|_1 \eps_{c,2} \r) \\ 
&\hspace{20pt} + \alpha \l( \sum_{\tau = t+1}^{T} \alpha^{T-\tau} \r) \l[ \l( \eps_{c,1} + \eps_{c,2} \| \lambda \|_1 + N\l( \alpha, T\r) \delta_c \r) \r],\numberthis\label{eq:M_c}\\
&M_p \l(t; \alpha, T \r) = M_p^{(\eps_{p,1}, \eps_{p,2}, \delta_p, \lambda, \ulbar{c}, \ulbar{d}, 
\constraintv)} \l(t; \alpha, T \r) \\
&\hspace{10pt} \defeq \l( \eps_{p,1} + \|\lambda\|_1 \eps_{p,2} \r)\l( \sum_{\tau=t}^{T} \alpha^{T-\tau} \r) \\
&\hspace{20pt} + \alpha \l( \sum_{i=0}^{T-t-1} \sum_{j=0}^{T-t-1-i} \alpha^{i+j} \r) \l[ \l( \eps_{p,1} + \|\lambda\|_1 \eps_{p,2} \r. \r. \\
&\hspace{150pt} \l. \l. + N \l( \alpha, T \r)\delta_p \r)  \r],\numberthis\label{eq:M_p}\\
&N\l( \alpha, T \r) = N^{(\lambda,\ulbar{c},\ulbar{d},\constraintv)} \l( \alpha, T \r) \\
&\hspace{0pt} \defeq  \sum_{\tau=1}^{T} \alpha^{\tau-1} \l[ \ulbar{c} + \|\lambda\|_1 \l( \ulbar{d} + \frac{1}{2} \l( \max_k \constraintv_k - \min_k \constraintv_k \r) \r)  \r].\numberthis\label{eq:N_alpha_T}
\end{align*}
\end{prop}
\begin{proof}
The proof follows the same methodology as in \cite{hsu22}[Theorem 7]. For brevity, it is omitted. It is helpful to note that the main result of \cite{hsu22} (namely Theorem 7) follows as a special case of this proposition (take $\alpha=1$ and $\lambda = 0$).
\end{proof}
Propositions \ref{prop:VtT_vs_Vt} and \ref{prop:optimalitygap_T} yield the following corollary.
\begin{cor}\label{cor:optimalitygap_T}
    Fix $\lambda\in\mcl{Y}$, $\alpha\in(0,1)$, $T \in \mbb{N}$, and $T$-horizon ASPS and ASCS generators $\varthetahtn{1:T}{0:N}$ (see Definitions \ref{dfn:asps_generator_finite_horizon} and \ref{dfn:ascs_generator_finite_horizon}). Suppose that Assumption \ref{assmp:boundedcosts} holds. Then, for any $t\in[T]$ and any $\hstn{\wt{h}}{t}{0} \in \wthstnspace{t}{0}$, the following relation holds between $\whVtl{t,T}{\lambda} \l( \varthetahtn{t}{0}\l(\hstn{\wt{h}}{t}{0}\r) \r) $ and $\Vtl{t}{\lambda} \l( \hstn{\wt{h}}{t}{0} \r) $.
    \begin{align*}
        &\whVtl{t,T}{\lambda} \l( \varthetahtn{t}{0}\l(\hstn{\wt{h}}{t}{0}\r) \r) + \frac{\alpha^{T-t+1}}{1-\alpha} \un{\udl{l}}{\lambda} \\
        &\hspace{30pt} - M_c\l(t;\alpha, T \r) - M_p\l(t;\alpha, T \r) \le \Vtl{t}{\lambda} \l( \hstn{\wt{h}}{t}{0} \r) \\
        &\hspace{90pt} \le \whVtl{t,T}{\lambda} \l( \varthetahtn{t}{0}\l(\hstn{\wt{h}}{t}{0}\r) \r) + \frac{\alpha^{T-t+1}}{1-\alpha} \un{\ov{l}}{\lambda}.
    \end{align*}
\end{cor}
\begin{rem}\label{rem:universalbound_optimalitygap}
Assumption \ref{assmp:slatercondition} can be combined with Corollary \ref{cor:optimalitygap_T} to get $\lambda$-independent bounds on the optimality gap. 
\end{rem}

\subsection{Approximate State Representations (Infinite-Horizon)}
We will now extend the notions from the finite-horizon setting to the infinite-horizon case. In doing so, synonymous to the standard MDP literature, our goal is to identify an ASPS-ASCS based fixed-point iteration scheme that can approximate the optimal value function $\Vtl{t}{\lambda}$ and return a policy that depends, in a time-homogeneous way, on the approximate-state processes $\{ \Zhtn{t}{0:N}\}_{t=1}^{\infty}$.

\begin{dfn}[Infinite-Horizon ASPS-Generator]\label{dfn:asps_generator_infinite_horizon}
Let $\zhnspace{1:N}$ be a collection of topological spaces. A collection $ \varthetahn{1:N} $ of compression-functions where $\varthetahn{n}$ maps $\cup_{t=1}^{\infty} \hstnspace{t}{0} \times \hstnspace{t}{n}$ to $\zhnspace{n}$ is called an infinite-horizon $\l(\eps_{p,1}, \eps_{p,2}, \delta_p \r)$-ASPS generator if the process $\{ \Zhtn{t}{1:N} \}_{t=1}^{\infty}$, with $\Zhtn{t}{n} = \varthetahn{n}( \Hstn{t}{0}, \Hstn{t}{n} ) $ almost-surely, satisfies ASPS-1, ASPS-2.1, ASPS-2.2, and ASPS-3 along with the addition that the evolution functions do not depend on time, i.e., $\phihtn{t}{n} = \phihtn{t'}{n}$.
\end{dfn}

\begin{dfn}[Infinite-Horizon ASCS-Generator]\label{dfn:ascs_generator_infinite_horizon}
Let $\zhnspace{0}$ be a topological space. A compression-function $ \varthetahn{0} $ where $\varthetahn{0}$ maps $\cup_{t=1}^{\infty}\wthstnspace{t}{0}$ to $\zhnspace{0}$ is called an infinite-horizon $\l(\eps_{c,1}, \eps_{c,2}, \delta_c \r)$-ASCS generator if the process $\{ \Zhtn{t}{0} \}_{t=1}^{\infty}$, with $\Zhtn{t}{0} = \varthetahn{0}( \wtHstn{t}{0} ) $ almost-surely, satisfies ASCS-1, ASCS-2.1, ASCS-2.2, and ASCS-3, along with the addition that 
\begin{enumerate}
    \item the evolution functions do not depend on time, i.e., $\phihtn{t}{0} = \phihtn{t'}{0}$; and
    \item the conditional expectations $ \mbb{E}_{P_1} \l[ c\l( \Stt{t}, \At{t} \r) \Big| \zhtn{t}{0}, \lamdaht{t} \r] $ and $ \mbb{E}_{P_1} \l[ d\l( \Stt{t}, \At{t} \r) \Big| \zhtn{t}{0}, \lamdaht{t} \r] $ do not depend on time.
\end{enumerate}

\end{dfn}
Let us denote the range of $\varthetahn{n}$ by $\zhnspaces{n}$, i.e., $\zhnspaces{0} \defeq \varthetahn{0} \l( \cup_{t=1}^{\infty} \wthstnspace{t}{0} \r)$ and $\zhnspaces{n} \defeq \varthetahn{n} \l( \cup_{t=1}^{\infty} \hstnspace{t}{0} \times \hstnspace{t}{n} \r)$ for all $n \in [N]$. Then the above definition leads to a time-invariant ASPS-based prescription space 
$\lamdahspace = \lamdahnspace{1:N}$ where $ \lamdahnspace{n} \defeq \{ \lamdahn{n} : \zhnspaces{n} \mapsto \m{\anspace{n}} \}$. We note that $\lamdahspace$ is in one-to-one correspondence with the set
\begin{align*}
    \xtspace{\lamdahspace} \defeq \prod_{n=1}^{N} \prod_{ \zhn{n} \in \zhnspaces{n} } \m{\anspace{n}; \zhn{n}}.\numberthis\label{eq:xlamdahspace}
\end{align*}
which is a compact (and convex) space by Tychonoff's theorem (see Proposition \ref{prop:tychonoff}), and is also metrizable (see Proposition \ref{prop:metrizability}). From hereon, we will assume $\lamdahspace$ to have the same topology as $\xtspace{\lamdahspace}$.

Definitions \ref{dfn:asps_generator_infinite_horizon} and \ref{dfn:ascs_generator_infinite_horizon} 
naturally lead to an approximate Bellman-type operator $\wh{B} : \{ \zhnspaces{0} \ra \mbb{R} \} \ra \{\zhnspaces{0} \ra \mbb{R}\} $ defined as follows: for a uniformly bounded function $\whVl{\lambda} : \zhnspaces{0} \ra \mbb{R}$,
\begin{align*}
    \l[ \wh{B} \whVl{\lambda} \r] \l( \zhn{0} \r) &\defeq \min_{\lamdah \in \lamdahspace} \mbb{E}_{P_1} \l[ \lCost \r. \\
    &\hspace{-30pt} \l. + \alpha \whVl{\lambda}\l( \Zhtn{t+1}{0} \r) \Big| \Zhtn{t}{0} = \zhtn{t}{0}, \Lamdaht{t} = \lamdah \r].\numberthis\label{eq:whB_whV}
\end{align*}
Note that with the definitions of infinite-horizon ASPS and ASCS (and the time-homogeneous prescriptions), the expectation on the right-hand-side in \eqref{eq:whB_whV} does not depend on time, and due to discounting ($\alpha\in(0,1)$), the operator $\wh{B}$ is a contraction under the supremum norm. Therefore, under Assumption \ref{assmp:boundedcosts}, the fixed point equation $\whVl{\lambda} = \wh{B} \whVl{\lambda}$ has a unique bounded solution according to the Banach fixed-point theorem, (the vector space of uniformly bounded functions with the supremum norm is a complete metric space, see Proposition \ref{prop:banach}).

\begin{thm}[Optimality Gap for Infinite-Horizon]\label{thm:optimalitygap}
Fix $\alpha\in(0,1)$ and $\lambda\in\mcl{Y}$, and infinite-horizon ASPS and ASCS generators $\varthetahn{0:N}$ (see Definitions \ref{dfn:asps_generator_infinite_horizon} and \ref{dfn:ascs_generator_infinite_horizon}). 
Suppose that Assumption \ref{assmp:boundedcosts} holds. Consider the fixed point equation \eqref{eq:whB_whV} which we rewrite as follows:
\begin{align*}
     \whVl{\lambda} \l( \zhn{0} \r) &\defeq \min_{\lamdah \in \lamdahspace} \whQl{\lambda} \l( \zhn{0}, \lamdah \r),\numberthis \label{eq:whVl} \\ 
    \whQl{\lambda} \l( \zhn{0}, \lamdah \r) &\defeq \mbb{E}_{P_1} \l[ \lCost \r. \\
    &\hspace{-30pt} \l. + \alpha \whVl{\lambda}\l( \Zhtn{t+1}{0} \r) \Big| \Zhtn{t}{0} = \zhtn{t}{0}, \Lamdaht{t} = \lamdah \r].\numberthis\label{eq:whQl}
\end{align*}
Let $ \whVl{\lambda,\star}$ denote the fixed-point of \eqref{eq:whB_whV} and $\whQl{\lambda, \star}$ denote the corresponding prescription-value function. Then, for any $t\in\mbb{N}$ and any $\hstn{\wt{h}}{t}{0} \in \wthstnspace{t}{0}$ with $\gt{t}^\star \in \argmin_{\gt{t} \in \gtspace{t} } \Qtl{t}{\lambda} \l( \hstn{\wt{h}}{t}{0}, \gt{t} \r)$, there exists $\lamdah \in \lamdahspace$ such that
\begin{align*}
&\whQl{\lambda,\star}\l( \varthetahn{0} \l( \hstn{\wt{h}}{t}{0} \r), \lamdah \r) - M_c \l( \alpha, \infty \r) - M_p\l( \alpha, \infty \r) \\
&\hspace{40pt} \le \Qtl{t}{\lambda}\l( \hstn{\wt{h}}{t}{0}, \gt{t}^\star \r) \le \whQl{\lambda,\star}\l( \varthetahn{0} \l( \hstn{\wt{h}}{t}{0} \r), \lamdah \r),\\
&\whVl{\lambda}\l( \varthetahn{0} \l( \hstn{\wt{h}}{t}{0} \r) \r) - M_c\l( \alpha, \infty \r) - M_p\l( \alpha, \infty \r) \\
&\hspace{40pt} \le \Vtl{t}{\lambda}\l( \hstn{\wh{h}}{t}{0} \r) \le \whVl{\lambda}\l( \varthetahn{0} \l( \hstn{\wt{h}}{t}{0} \r) \r),
\end{align*}
where $\Vtl{t}{\lambda}$ and $\Qtl{t}{\lambda}$ were defined in \eqref{eq:V_t_lamda} and \eqref{eq:Q_t_lamda}, respectively.
\end{thm}

\begin{proof}
Consider the following sequence of value-functions: $^{0}\whVl{\lambda} \equiv 0$ and for $i\in \mbb{N}$, $^{i+1}\whVl{\lambda} = \wh{B} \l(^{i}\whVl{\lambda}\r)$. Let $T \in [t,\cdot]_{\mbb{Z}}$; by construction of $\l\{^{i}\whVl{\lambda}\r\}_{i=0}^{\infty}$, it follows that $^{T-t+1}\whVl{\lambda} = \whVtl{t,T}{\lambda} $. Therefore, from Corollary 5.1, 
\begin{align*}
    &^{T-t+1}\whVl{\lambda} \l( \varthetahtn{t}{0}\l(\hstn{\wt{h}}{t}{0}\r) \r) + \frac{\alpha^{T-t+1}}{1-\alpha} \un{\udl{l}}{\lambda} \\
    &\hspace{20pt} - M_c\l( \alpha, T \r) - M_p\l( \alpha, T \r) \le \Vtl{t}{\lambda} \l( \hstn{\wt{h}}{t}{0} \r)\\
    &\hspace{40pt} \le ^{T-t+1}\whVl{\lambda} \l( \varthetahtn{t}{0}\l(\hstn{\wt{h}}{t}{0}\r) \r) + \frac{\alpha^{T-t+1}}{1-\alpha} \un{\ov{l}}{\lambda}.
    \end{align*}
Taking the limit $T\ra\infty$, we get
\begin{align*}
    &\whVl{\lambda,\star} \l( \varthetahtn{t}{0}\l(\hstn{\wt{h}}{t}{0}\r) \r) - M_c\l(\alpha, \infty\r) - M_p\l(\alpha, \infty\r)  \\
    &\hspace{20pt} \le \Vtl{t}{\lambda} \l( \hstn{\wt{h}}{t}{0} \r) \le \whVl{\lambda, \star} \l( \varthetahtn{t}{0}\l(\hstn{\wt{h}}{t}{0}\r) \r).
\end{align*}
By choosing $\lamdah \in \lamdahspace$ to be the minimizing prescription in the corresponding prescription-value function $\whQl{\lambda,\star}$, we get the inequality in \eqref{eq:whQl}.
\end{proof}

Theorem \ref{thm:optimalitygap} gives us a time-homogeneous ASPS-ASCS based coordination policy that is approximately optimal for the $\lambda$-parametrized infinite-horizon objective $L_{\infty}\l( \cdot, \lambda \r)$. Specifically, let $\ut{\wh{v}} = \ut{\wh{v}}{1:\infty} $ be such that for all $t\in\mbb{N}$, $\ut{\wh{v}}{t}\l( \varthetahn{0} \l( \hstn{\wt{h}}{t}{0} \r) \r) \in \argmin_{\lamdah \in \lamdahspace} \whQl{\lambda,\star} \l( \varthetahn{0} \l( \hstn{\wt{h}}{t}{0} \r) \r)$. Then,
\begin{align*}
L_\infty \l( \wh{v}, \lambda \r) - \inf_{v\in \vvspace} L_\infty \l( v, \lambda \r) \le M_c\l( \alpha, \infty \r) + M_p\l(\alpha, \infty \r).
\end{align*}

%% file: Sections/marl.tex
\section{Primal Dual Type Reinforcement Learning Using Approximate Information States}\label{sec:marl}
\begin{figure*}[!t]
\centering
\subfloat
{\includegraphics[width=.95\linewidth]{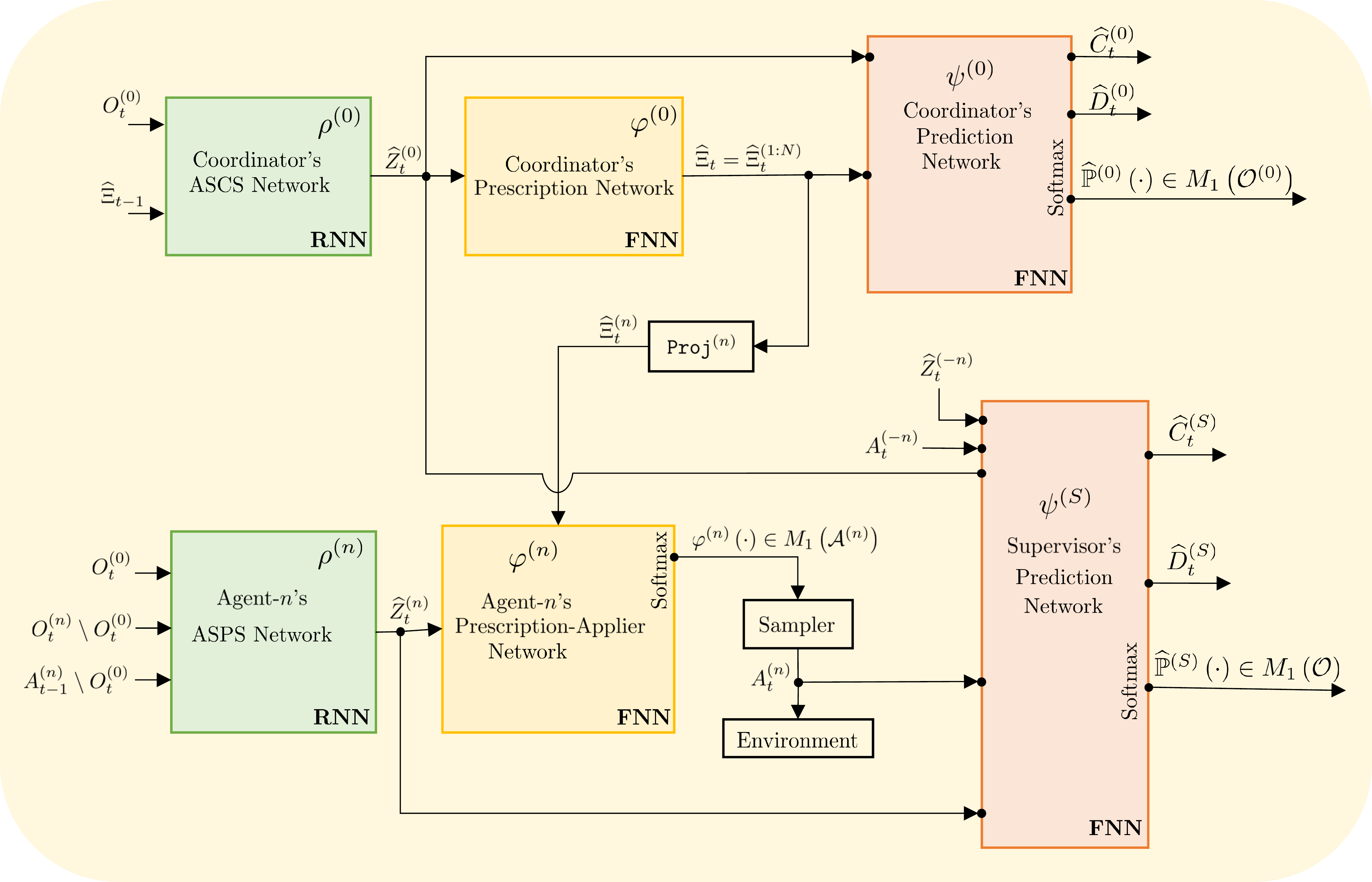}}
\caption{Centralized Training for cooperative MA-C-POMDPs based on ASPS and ASCS. (The superscript $(-n)$ denotes $[N]\setminus \{n\}$). 
Once the training is complete, the execution phase is distributed -- the prediction networks $\psi^{(0)}$ and $\psi^{(S)}$ are no longer needed.
}
\label{fig:marl1}
\end{figure*}
In this section, we use the notions of ASPS and ASCS to develop an algorithmic framework for reinforcement learning in (cooperative) MA-C-POMDPs. This framework will involve recurrent and feed-forward neural-networks as function-approximators, and will be based on \emph{centralized training, distributed execution (CTDE)} wherein the training will be performed in a three time-scale stochastic approximation setup\footnote{See \cite{borkar08_sabook} for details.} as follows: first, the (time-invariant) approximate-state generators will be learnt on fast time-scale (using loss-functions motivated by ASPS-2.1/2.2/3 and ASCS-2.1/2.2/3), then an ASPS-ASCS based coordination policy will be learnt on medium time-scale (using off-the-shelf policy-gradient algorithms), and finally the optimal Lagrange-multipliers vector will be learnt on the slowest time-scale (using projected gradient-ascent).

\subsection{Function-Approximators}
We use the following types of 
function-approximators (see Figure \ref{fig:marl1}):
\begin{enumerate}
    \item \emph{Coordinator's ASCS Network}, modelled by a recurrent neural network (RNN), has inputs $\Otn{t}{0}, \Lamdaht{t-1}$, internal state $\Zhtn{t-1}{0} $, and output $\Zhtn{t}{0}$. It is denoted by $\rho^{(0)}$.
    \item \emph{Agent-$n$'s ASPS Network}, modelled by an RNN with inputs $\Otn{t}{0}, ( \Otn{t}{n}, \Atn{t-1}{n} ) \setminus \Otn{t}{0} $, internal state $\Zhtn{t-1}{n} $, and output $\Zhtn{t}{n}$. It is denoted by $\rho^{(n)}$. 
    \item \emph{Coordinator's Prescription Network}, modelled by a feed-forward neural network (FNN) with input $\Zhtn{t}{0} $ and output $\Lamdahtn{t}{1:N}$. It is denoted by $\varphin{0}$.
    \item \emph{Agent-$n$'s Prescription-Applier Network}, modelled by an FNN (with softmax as its last layer), has inputs $\Zhtn{t}{n}, \Lamdahtn{t}{n}$, and outputs a distribution on $\anspace{n}$.
    \item \emph{Coordinator's Prediction Network}, modelled by an FNN, has inputs $ \Zhtn{t}{0}, \Lamdahtn{t}{1:N}$, and outputs $\utn{\wh{C}}{t}{0}, \utn{\wh{D}}{t}{0}$ and $\un{\wh{\pr}}{0}\l( \cdot \r) \in \m{\onspace{0}}$. It is denoted by $\psi^{(0)}$. Here, $\utn{\wh{C}}{t}{0}$ and $\utn{\wh{D}}{t}{0}$ respectively serve as estimates of the conditional expectation of $\cCost$ and $\dCost$ given $\Zhtn{t}{0}, \Lamdahtn{t}{1:N}$ (see ASCS-2.1/2.2) and $\un{\wh{\pr}}{0}\l( \cdot \r)$ serves as an estimate of the conditional distribution of $\Otn{t+1}{0}$ given $\Zhtn{t}{0}, \Lamdahtn{t}{1:N}$ (see ASCS-3).
    \item \emph{Supervisor's Prediction Network}, modelled by an FNN, has inputs $ \Zhtn{t}{0:N}, \Atn{t}{1:N}$, and outputs $\utn{\wh{C}}{t}{S}, \utn{\wh{D}}{t}{S}$ and $\un{\wh{\pr}}{S}\l( \cdot \r) \in \m{\onspace{0:N}}$. It is denoted by $\psi^{(S)}$. Here, $\utn{\wh{C}}{t}{S}$ and $\utn{\wh{D}}{t}{S}$ respectively serve as estimates of the conditional expectation of $\cCost$ and $\dCost$ given $\wtHstn{t}{0}, \Zhtn{t}{1:N}, \Atn{t}{1:N}$ (see ASPS-2.1/2.2), and $\un{\wh{\pr}}{S}\l( \cdot \r)$ serves as an estimate of the conditional distribution of $\Ot{t+1}$ given $\wtHstn{t}{0}, \Zhtn{t}{1:N}, \Atn{t}{1:N}$ (see ASCS-3).
    
    Here, \emph{supervisor} is a omniscient entity that can access the union of the information of all agents and is needed because the approximation criteria in ASPS-2.1/2.2/3 require the knowledge of the tuples $\Zhtn{t}{0:N}$ and $\Atn{t}{1:N}$.
\end{enumerate}

\subsection{Centralized Training Distributed Execution}
The architectural setup of the aforementioned function-approximators is shown pictorially in \figurename \ref{fig:marl1}. Here, the networks $\rho \defeq \rhon{0:N}$ and $ \varphi \defeq \varphin{0:N}$ collectively define a parametrized ASPS-ASCS coordination policy which we denote by $\ut{\wh{v}}{\rho, \varphi}$. The state networks $\rho$ collectively generate ASPS and ASCS $\Zhtn{t}{0:N}$. The coordinator's prescription network $\varphin{0}$ takes the generated ASCS $\Zhtn{t}{0}$ as input and outputs a \emph{pseudo-prescription} tuple $\Lamdaht{t}$.\footnote{We use the term \emph{pseudo-prescription} because the original interpretation of a prescription is lost.} Then, each agent-$n$'s prescription-applier network $\varphin{n}$ takes $\Zhtn{t}{n}$ and 
$\Lamdahtn{t}{n}$ as input and outputs a distribution on $\anspace{n}$. This distribution is used by the agent to draw its action $\Atn{t}{n}$. 
Finally, the environment generates the next observations in response to the joint-action $\At{t}$. 

\input{Sections/pseudocode.tex}

In the above framework, the prediction-networks $\psi^{(0)}$ and $\psi^{(S)}$ are used (during the training phase) to generate estimates of the conditional expectations of immediate-costs and conditional distributions of observations. These estimates are compared with ground-truth realizations to drive the learning of the state-generators $\rho$. As concerns the learning of the coordination policy, the basic idea (synonymous to single-agent learning settings) is to get sample-paths based estimates of the policy-gradient, 
\begin{align*}
\nabla_{\varphi} L_{\infty} \l(\ut{\wh{v}}{\rho, \varphi},\lambda  \r) = \nabla_{\varphi} \E{\ut{\wh{v}}{\rho, \varphi}}{P_1} \l[ \sum_{t=1}^{\infty} \alpha^{t-1} \lCost \r].\numberthis\label{eq:policy_gradient}
\end{align*}
For learning the Lagrange-multipliers vector, sample-bath estimates of the below gradient are used.
\begin{align*}
\nabla_{\lambda} L_{\infty}\l(\ut{\wh{v}}{\rho, \varphi},\lambda \r) = \nabla_{\lambda}  \E{\ut{\wh{v}}{\rho, \varphi}}{P_1} \l[ \sum_{t=1}^{\infty} \alpha^{t-1} \lCost \r].\numberthis\label{eq:lagrangian_gradient}
\end{align*}
Once the training is complete, the execution phase is distributed -- the prediction networks $\psi^{(0)}$ and $\psi^{(S)}$ are no longer needed and (as mentioned earlier) the coordinator's networks can be instantiated by all agents. In the remainder of this section, we give a concrete instance of this framework that is based on multi-agent extension of the (single-agent) REINFORCE algorithm. Extending other policy-gradient algorithms such as actor-critic methods is 
left as an exercise.

\subsection{Three Time-scale Stochastic Approximation - Example Instantiation Based on REINFORCE \cite{sutton98} Algorithm}
For simplicity, we assume that the observation-sets $\onspace{n}$'s are finite. Also, with slight abuse of notation, we will denote the weights of a neural-network by the same letter that is used to denote it. Let $\l\{ \delta_{1,i} \r\}_{i=1}^{\infty}, \l\{ \delta_{2,i} \r\}_{i=1}^{\infty}, \l\{ \delta_{3,i} \r\}_{i=1}^{\infty} $ be three sequences of step-sizes that satisfy the standard three time-scale stochastic approximation conditions \cite{borkar08_sabook}, namely,
\begin{align}
\begin{split}\label{eq:stepsizes}
    \sum_{i=1}^{\infty} \delta_{1, i} = \sum_{i=1}^{\infty} \delta_{2, i} = \sum_{i=1}^{\infty} \delta_{3, i} = \infty,\\
    \sum_{i=1}^{\infty} \delta_{1, i}^2 + \sum_{i=1}^{\infty} \delta_{2, i}^2 + \sum_{i=1}^{\infty} \delta_{3, i}^2 < \infty,\\
    \frac{\delta_{3, i} }{ \delta_{2, i}} , \frac{\delta_{2, i} }{ \delta_{1, i}} \xrightarrow{i\ra\infty} 0.
    \end{split}
\end{align}
In our setup, $\delta_{1,i}$ will be used to learn the approximate-state generator networks, $\delta_{2,i}$ will be used to update the parameters of the ASPS-ASCS based coordination policy, and $\delta_{3,i}$ will be used to update the Lagrange-multipliers vector.

\begin{rem}
In practical implementations, the above conditions are rarely satisfied.
\end{rem}

\subsubsection{Learning the Approximate-State Generator Networks}
To drive the learning of the approximate-state generators, a few definitions are in order.
\begin{enumerate}
    \item To minimize the epsilons in the definitions of ASCS and ASPS, we define the loss-function,
    \begin{align*}
        l_2 &: \mbb{R}\times \mbb{R} \ra \mbb{R}_{\ge 0}\\
        l_2\l(\cdot, \star \r) &= \operatorname{smoothL1}\l( \cdot - \star \r)\\
        &\defeq \begin{cases}
        \frac{1}{2} \l( \cdot - \star \r)^2, &|\cdot-\star|<1,\\
        |\cdot - \star| - \frac{1}{2}, &\text{otherwise}.
        \end{cases}\numberthis\label{eq:l2}
    \end{align*}
    \item To minimize $\delta_c$ in ASCS-3, we define the negative log-likelihood loss-function,
    \begin{align*}
        l_{c,3} : \onspace{0} \times \m{\onspace{0}} &\ra \mbb{R}_{\ge 0} 
        ,\\
        l_{c,3} \l( \on{0}, \wh{\pr}^{(0)}\l( \cdot \r) \r) &\defeq - \log \l(  \wh{\pr}^{(0)}\l( \on{0} \r) + \eta \r).\numberthis\label{eq:lc3}
    \end{align*}
    \item And similarly, to minimize $\delta_p$ in ASPS-3, we define
    \begin{align*}
        l_{p,3} : \ospace \times \m{\ospace} &\ra \mbb{R}_{\ge 0} 
        ,\\
        l_{p,3}\l( o, \wh{\pr}^{(S)}\l(\cdot\r) \r) &\defeq - \log \l( \wh{\pr}^{(S)}\l(o\r) + \eta \r).\numberthis\label{eq:lp3}
    \end{align*}
\end{enumerate}
In \eqref{eq:lc3} and \eqref{eq:lp3}, $\eta$ is a sufficiently-small hyper-parameter to avoid indefinite gradients.

Consider a set of $B$ finite-horizon trajectories of length $T$ generated (independently) using coordination policy $\ut{\wh{v}}{\rho, \varphi} $. We denote this set by $\l\{ \tau_j \r\}_{j=1}^{B}$, where $\tau_j$ is given by
\begin{align*}
\tau_j &\defeq \l\{ o_{j,t}, a_{j,t}, c_{j,t}, d_{j,t}, \zhtn{j,t}{0:N}, \lamdaht{j,t}, \r. \\
&\hspace{20pt} \l. \wh{c}_{j,t}^{(0)}, \wh{d}_{j,t}^{(0)}, \wh{c}_{j,t}^{(S)}, \wh{d}_{j,t}^{(S)}, \wh{\pr}^{(0)}_{j,t}\l( \cdot \r), \wh{\pr}^{(S)}_{j,t}\l( \cdot \r) \r\}_{t=1}^{T}.\numberthis\label{eq:tauj}
\end{align*}
Here, $\star_{j,t}$ denotes the realization of the corresponding variable at time $t$ of the $j^{th}$ trajectory and $T$ denotes the common length of all the $B$ trajectories.\footnote{For the case of infinite-horizon total expected discounted costs, $T$ should preferably be on the order of $\frac{1}{1-\alpha}$.} For learning of ASCS-generator $\rho^{(0)}$ (coupled with the learning of the prediction network $\psin{0}$), we can use the loss-function: 
\begin{align*}
    &l_{\rhon{0}, \psin{0}} \l( c_{j, t}, d_{j, t}, o_{j, t+1}^{(0)}, \wh{c}_{j,t}^{(0)}, \wh{d}_{j,t}^{(0)}, \wh{\pr}_{j,t}^{(0)}\l(\cdot \r) \r) \\
    &\defeq \beta_0 l_2\l( \ut{c}{j,t}, \ut{\wh{c}}{j,t} \r) + \sum_{k=1}^{K} \beta_k l_2 \l( \ut{\l(d_k\r)}{j,t}, \ut{\l(\wh{d}_k\r)}{j,t} \r) \\
    &\hspace{10pt} + \beta_{K+1}l_{c,3} \l( \otn{j,t}{0}, \wh{\pr}_{j,t}^{(0)}\l( \cdot \r) \r).\numberthis\label{eq:l_rho0_psi0}
\end{align*}
Similarly, for learning of the ASPS-generator (coupled with the learning of $\rhon{0}$ and $\psin{S}$), we can use the loss-function:
\begin{align*}
    l_{\rho, \psin{S}} &\defeq \beta'_0 l_2\l( \ut{c}{j,t}, \ut{\wh{c}}{j,t} \r) + \sum_{k=1}^{K} \beta'_k l_2 \l( \ut{\l(d_k\r)}{j,t}, \ut{\l(\wh{d}_k\r)}{j,t} \r) \\
    &\hspace{30pt} + \beta'_{K+1}l_{p,3} \l( \ot{j,t}, \wh{\pr}_{j,t}^{(S)}\l( \cdot \r) \r).\numberthis\label{eq:l_rho_psiS}
\end{align*}
In \eqref{eq:l_rho0_psi0} and \eqref{eq:l_rho0_psi0}, the weights $\beta_k, \beta'_k$ are hyper-parameters that satisfy $\beta_k, \beta'_k > 0, \text{ and } \sum_{k=0}^{K+1} \beta_k = \sum_{k=0}^{K+1}\beta'_k = 1$. The above loss-functions lead to the following empirical risk functions whose gradients can be used to update the networks $\rho, \psin{0}$, and $\psin{S}$.
\begin{align*}
    &R_{ \un{\rho}{0}, \un{\psi}{0} } \l( \l\{ \tau_j \r\}_{j=1}^{B} \r)
    \defeq  \frac{1}{BT} \sum_{j=1}^{B} \sum_{t=1}^{T} l_{\un{\rho}{0},\un{\psi}{0} } 
    \l( \cdot \r),\numberthis\label{eq:R_rho0_psi0}\\
    &R_{ \rho, \un{\psi}{S} } 
    \l( \l\{ \tau_j \r\}_{j=1}^{B} \r)
    \defeq  \frac{1}{BT} \sum_{j=1}^{B} \sum_{t=1}^{T} l_{\un{\rho}{0},\un{\psi}{0} } 
    \l( \cdot \r).\numberthis\label{eq:R_rho_psiS}
\end{align*}
These risk functions can be used to update the state-generator and prediction networks $\rho, \psi^{(0)}, \psi^{(S)}$ (See Lines \ref{line:alg3:R_rho0_psi0} and \ref{line:alg3:R_rho_psiS} in Algorithm \ref{alg:reinforce_macpomdp}).

\subsubsection{Learning Prescription and Prescription-Applier Networks}
Based on the collected trajectories, we define cost-to-go terms,
\begin{align*}
    g_{j,t} \defeq \sum_{t'=t}^{T} \alpha^{t'-t} \l( c_{j,t} + \dotp{\lambda}{d_{j,t}-\constraintv} \r).\numberthis\label{eq:gjt}
\end{align*}
Then, the sample-paths based estimate of $\nabla_{\varphi} L_{\infty} \l(\ut{\wh{v}}{\rho,\varphi},\lambda \r)$, in flavor of the REINFORCE algorithm \cite{sutton98}, is given by
\begin{align*}
&\wh{\nabla_{\varphi} L_{\infty}}\l(\ut{\wh{v}}{\rho,\varphi},\lambda \r) \\
&\hspace{5pt} \defeq \nabla_{\varphi} \frac{1}{B} \sum_{j=1}^{B} \sum_{t=1}^{T}  g_{j, t} \l[ \sum_{n=1}^{N} \log \l( \varphi^{(n)}\l( \atn{j,t}{n}, \zhtn{j,t}{n}, \r. \r. \r.\\
&\hspace{100pt} \l. \l. \l. \proj{n}{ \un{\varphi}{0}\l( \zhtn{j,t}{0} \r) } 
\r)   \r) \r],\numberthis\label{eq:pg_estimate_reinforce}
\end{align*}
which can be used to update the prescription and prescription-applier networks $\varphi$. (See Lines \ref{line:alg3:gjt} and \ref{line:alg3:pg_estimate} in Algorithm \ref{alg:reinforce_macpomdp}).

\subsubsection{Learning the Lagrange-Multiplier}
The sample-paths based estimate of $\nabla_{\lambda} L_{\infty} \l(\ut{\wh{v}}{\rho,\varphi},\lambda \r) $ is simply given by
\begin{align*}
    \wh{\nabla_{\lambda} L_{\infty}} \l( \ut{\wh{v}}{\rho,\varphi} ,\lambda\r) \defeq \l[ 
    \l( \frac{1}{B} \sum_{j=1}^{B} \sum_{t=1}^{T} \alpha^{t-1} d_{j,t}\r) - \constraintv
    \r]^{+},\numberthis\label{eq:lg_estimate}
\end{align*}
which is used to project the updated $\lambda$ onto $\mcl{Y}$ (see Line \ref{line:alg3:lg_estimate}) in Algorithm \ref{alg:reinforce_macpomdp}).

%% file: Sections/pseudocode.tex
\begin{algorithm}[t]
\DontPrintSemicolon

\KwInput
{
Initial weights of all neural-networks 
$\l( \utn{\rho}{0}{0:N}, \utn{\varphi}{0}{0:N}, \utn{\psi}{0}{0}, \utn{\psi}{0}{S} \r) $.}
\Parameter{
Batch-size ($B$), 
simulated-episodes' horizon ($T$), 
Learning schedules $\{ \delta_{1,i}, \delta_{2,i}, \delta_{3,i} \}_{i=1}^{\infty}$ satisfying \eqref{eq:stepsizes}, $maxiters$.
}
\KwOutput{
Learned weights of all neural-networks.
}
$i \la 0$.\\
\Repeat{convergence or $i > maxiters$.}{
Collect a set of $B$ trajectories namely $\l\{ \tau_j \r\}_{j=1}^{B}$, each of horizon $T$, by running the ASPS-ASCS based coordination policy $\ut{\wh{v}}{\rho_i,\varphi_i}$. (See \eqref{eq:tauj}).\label{line:alg3:trajectories}

\tcc{Coordinator's Backprop}
Compute $R_{ \un{\rho}{0}, \un{\psi}{0} } \l( \l\{ \tau_j \r\}_{j=1}^{B} \r)$ according to \eqref{eq:R_rho0_psi0}.\label{line:alg3:R_rho0_psi0}

\nonl \begin{align*}
\begin{bmatrix}
\un{\rho_{i+1}}{0}\\
\un{\psi_{i+1}}{0}
\end{bmatrix} \la 
\begin{bmatrix}
\un{\rho_{i}}{0}\\
\un{\psi_{i}}{0}
\end{bmatrix}
- \delta_{1,i} \nabla_{\un{\rho}{0}, \un{\psi}{0} } R_{\un{\rho}{0}, \un{\psi}{0}} \big|_{\un{\rho_i}{0}, \un{\psi_i}{0}} . 
\end{align*}

\tcc{Supervisor's Backprop}
\nl Compute $R_{ \rho, \un{\psi}{S} } \l( \l\{ \tau_j \r\}_{j=1}^{B} \r)$ according to \eqref{eq:R_rho_psiS}.\label{line:alg3:R_rho_psiS}
\nonl \begin{align*}
&\begin{bmatrix}
\un{\rho_{i+1}}{0}\\
\un{\rho_{i+1}}{1:N}\\
\un{\psi_{i+1}}{S}
\end{bmatrix} \la 
\begin{bmatrix}
\un{\rho_{i+1}}{0}\\
\un{\rho_{i}}{1:N}\\
\un{\psi_{i}}{S}
\end{bmatrix}
\\
&\hspace{60pt} - \delta_{1, i} \nabla_{\rho, \un{\psi}{S} } R_{\rho, \un{\psi}{S}} \big|_{\un{\rho_{i+1}}{0}, \un{\rho_i}{1:N}, \un{\psi_i}{S} } . 
\end{align*}

\tcc{Gradient-descent in the primal}
\nl Compute cost-to-go terms $\{ \{ g_{j,t} \}_{t=1}^{T} \}_{j=1}^{B}$ given by \eqref{eq:gjt}.\label{line:alg3:gjt}

Compute $\wh{\nabla_{\varphi}L_{\infty}}\l(\ut{\wh{v}}{\rho,\varphi} ,\lambda  \r) \big|_{\varphi_i}$ according to \eqref{eq:pg_estimate_reinforce}.\label{line:alg3:pg_estimate}

\nonl \begin{align*}
    \begin{bmatrix}
    \un{\varphi_{i+1}}{0}\\
    \un{\varphi_{i+1}}{1:N}
    \end{bmatrix}
    &\la
    \begin{bmatrix}
    \un{\varphi_{i}}{0}\\
    \un{\varphi_{i}}{1:N}
    \end{bmatrix}
    - \delta_{2, i}
    \wh{\nabla_{\varphi}L_{\infty}}\l(\ut{\wh{v}}{\rho,\varphi} ,\lambda  \r) \big|_{\varphi_i}.
\end{align*}

\tcc{Gradient-ascent in the dual}
\nl Compute $\wh{\nabla_{\lambda}L_{\infty}}\l(\ut{\wh{v}}{\rho,\varphi} ,\lambda  \r) \big|_{\lambda_i}$ according to \eqref{eq:lg_estimate}.\label{line:alg3:lg_estimate}

\nonl \begin{align*}
    \lambda_{i+1} \la \lambda_{i} + \delta_{3, i} \wh{\nabla_{\lambda}L_{\infty}}\l(\ut{\wh{v}}{\rho,\varphi} ,\lambda  \r) \big|_{\lambda_i}.
\end{align*}
$i \la i+1$.
}
\caption{Pseudo-code for History-Embedding Based Reinforcement Learning in constrained Multi-Agent POMDPs.}\label{alg:reinforce_macpomdp}
\end{algorithm}

%% file: Sections/conclusion.tex
\section{Conclusion}\label{sec:conclusion}
In this work, we studied a cooperative multi-agent constrained POMDP in the setting of infinite-horizon expected total discounted costs. We established strong-duality and existence of a saddle point using Sion's Minimax Theorem -- by giving a suitable topology to the space of all possible policy-profiles and then establishing lower semi-continuity of the Lagrangian function. Then, we laid out the characterization of a universal compression-framework that is independent of the Lagrange-multipliers vector -- therefore easing the learning of the compression-mappings in the learning context. Based on the said compression-framework, we then designed a primal-dual MARL framework based on CTDE and three timescale stochastic approximation with neural-networks for function-approximation.

The theoretical results in this work are the first of its kind and will hopefully encourage investigation of MA-C-POMDPs in more general settings. Furthermore, an empirical evaluation of the proposed primal-dual MARL framework (an extension of single-agent REINFORCE algorithm~\cite{sutton98}) or its actor-critic variants is also an interesting direction.

%% file: Sections/appendices.tex
\appendices


\renewcommand{\theequation}{\thesection.\arabic{equation}}

\input{Sections/app_intermediate_results}

\input{Sections/app_facts}

\input{Sections/minimax_theorem}

\input{Sections/appendix_notation.tex}

%% file: Sections/app_intermediate_results.tex
\section{Intermediary Results for Theorem \ref{thm:strongduality}}\label{sec:appendix:intermediary_results}
\begin{lem}[Equivalence between Behavioral Policy-Profiles and their (decentralized) Mixtures]\label{lem:dominance}
Fix a (factorized) measure $\mu \in \uspacemix$. Then there exists a behavioral policy-profile $\udl{u}=\udl{u}(\mu) \in \uspace$, such that for any $t \in \mbb{N}$, $\hst{h}{t} \in \hstspace{t}$, and $\at{t} \in \aspace$,
\begin{align*}
    p\l( \mu, t, \hst{h}{t}, \at{t} \r) = p\l( u, t, \hst{h}{t}, \at{t} \r),
\end{align*}
where, for brevity and with slight abuse of notation,
\begin{align*}
p\l( \cdot, t, \hst{h}{t}, \at{t} \r) &= \prup{\cdot}{P_1}\l(\Hst{t} = \hst{h}{t}, \At{t} = \at{t} \r), \text{ and}\\
p\l( \cdot, t, \hst{h}{t} \r) &= \prup{\cdot}{P_1}\l(\Hst{t} = \hst{h}{t} \r).
\end{align*}
\end{lem}
\begin{proof}
Define $\udl{u}=\udl{u}(\mu) \in \uspace$ such that
\begin{align*}
    &\ut{\udl{u}}{t}\l( \at{t} | \hst{h}{t}\ \r) = \prod_{n=1}^{N} 
    \utn{\udl{u}}{t}{n} 
    \l( \atn{t}{n} | \hstn{h}{t}{0}, \hstn{h}{t}{n} \r)  \\
    &= \begin{cases}
    \frac{ p\l( \mu, t, \hst{h}{t}, \at{t} \r)
    }{p\l( \mu, t, \hst{h}{t}\r)}, &\text{if } p\l(\mu, t, \hst{h}{t} \r) \ne 0,\\
    \prod_{n=1}^{N} \frac{1}{|\anspace{n}|}, &\text{otherwise}.
    \end{cases}\numberthis\label{eq:dominance:u}
\end{align*}
The above assignment is correct because the right-hand-side of \eqref{eq:dominance:u} is a fully-factorized function of $\atn{t}{n}$'s.
\begin{align*}
    &p\l( \mu, t, \hst{h}{t}, \at{t} \r) = \int_{U} \mu\l( du \r) \prup{u}{P_1} \l( \Hst{t} = \hst{h}{t}, \At{t} = \at{t} \r)\\
    &= \int_{\uspace} P_1\l( \sspace, \hst{h}{1} \r) \prod_{t'=2}^{t} \pr_{P_1} \l( \ot{t'} | \hst{h}{t'-1}, \at{t'-1} \r)\\
    &\hspace{10pt} \times \prod_{n=1}^{N} \prod_{t'=1}^{t} \utn{u}{t'}{n}\l( \atn{t'}{n} |  \hstn{h}{t'}{0}, \hstn{h}{t'}{n} \r) \mu\l( du \r)
    \\
    &=P_1\l( \sspace, \hst{h}{1} \r) \prod_{t'=1}^{t} \pr_{P_1} \l( \ot{t'} | \hst{h}{t-1}, \at{t'-1} \r)\\
    &\hspace{10pt} \times \prod_{n=1}^{N} \int_{\uspacen{n}}  \prod_{t'=1}^{t} \utn{u}{t'}{n}\l( \atn{t'}{n} |  \hstn{h}{t'}{0}, \hstn{h}{t'}{n} \r) \mun{n}\l( d\un{u}{n}\r),
\end{align*}
where the last equality follows from Tonneli's Theorem (see Proposition \ref{prop:tonneli}). We will now prove, by forward induction, that for all $t\in\mbb{N}$, $\udl{u}$ and $\mu$ induce the same distribution on the pair $\l( \Hst{t}, \At{t} \r)$.

\begin{enumerate}
    \item \textbf{Base Case}: For time $t=1$, let $\ot{1} \in \hstspace{1} = \ospace$ and $\at{1} \in \aspace$. We have
    \begin{align*}
        p\l( \mu, 1, \ot{1}, \at{1}\r) &= P_1\l( \sspace, \ot{1} \r) \int_{\uspace} \mu\l( du \r) \ut{u}{1}\l( \at{1} | \ot{1} \r),
    \end{align*}
    and
    \begin{align*}
        p\l( \udl{u}, 1, \ot{1}, \at{1}\r) &= P_1\l( \sspace, \ot{1} \r) \ut{\udl{u}}{1}\l( \at{1} | \ot{1} \r) \\
        &\hspace{-30pt}= P_1\l( \sspace, \ot{1} \r) \frac{p\l( \mu, 1, \ot{1}, \at{1}\r)}{p\l( \mu, 1, \ot{1}\r)}\\ 
        &
        \hspace{-30pt}=p\l( \mu, 1, \ot{1}, \at{1}\r),
    \end{align*}
    where the last equality follows from $p\l( \mu, 1, \ot{1}\r) = P_1\l( \sspace, \ot{1} \r) $.

    \item \textbf{Induction Step}. Assuming that the statement is true for time $t$, we show that it is true for time $t+1$. Let $\hst{h}{t+1} = \l( \ot{1:t+1}, \at{1:t} \r) = \l( \hst{h}{t}, \at{t}, \ot{t+1} \r) \in \hstspace{t+1}$ and $\at{t+1} \in \aspace$. We have
    \begin{align*}
        p\l(\mu, t+1, \hst{h}{t+1} \r) &= \int_{\uspace} \mu\l( du \r) \prup{u}{P_1} \l( \Hst{t+1} = \hst{h}{t+1} \r)\\
        &\hspace{-60pt} = \int_{\uspace} \mu\l( du \r) \prup{u}{P_1} \l( \Hst{t} = \hst{h}{t}, \At{t} = \at{t}, \Ot{t+1} = \ot{t+1} \r)\\
        &\hspace{-60pt}= \int_{\uspace} \mu\l( du \r) \prup{u}{P_1} \l( \Hst{t} = \hst{h}{t}, \At{t} = \at{t}, \r) \\
        &\hspace{-40pt} \times \prup{u}{P_1} \l( \Ot{t+1} = \ot{t+1} | \Hst{t} = \hst{h}{t}, \At{t} = \at{t} \r)\\
        &\hspace{-60pt}= \int_{\uspace} \mu\l( du \r) \prup{u}{P_1} \l( \Hst{t} = \hst{h}{t}, \At{t} = \at{t}, \r) \\
        &\hspace{-40pt} \times \pr_{P_1} \l( \Ot{t+1} = \ot{t+1} | \Hst{t} = \hst{h}{t}, \At{t} = \at{t} \r)\\
        &\hspace{-60pt}= p\l(\mu, t, \hst{h}{t}, \at{t}\r) \pr_{P_1} \l( \Ot{t+1} = \ot{t+1} | \Hst{t} = \hst{h}{t}, \At{t} = \at{t} \r)\\
        &\hspace{-60pt}\labelrel{=}{eqr:dominance:ind} p\l(\udl{u}, t, \hst{h}{t}, \at{t}\r) \pr_{P_1} \l( \Ot{t+1} = \ot{t+1} | \Hst{t} = \hst{h}{t}, \At{t} = \at{t} \r)\\
        &\hspace{-60pt}=p\l(\udl{u}, t+1, \hst{h}{t+1} \r),
    \end{align*}
where \eqref{eqr:dominance:ind} uses the inductive hypothesis. The above work implies 
\begin{align*}
&p\l(\udl{u}, t+1, \hst{h}{t+1}, \at{t+1} \r) \\
&\hspace{0pt} = p\l(\udl{u}, t+1, \hst{h}{t+1} \r) \cdot\ \ut{\udl{u}}{t+1}\l(\at{t+1} | \hst{h}{t+1} \r) \\
&\hspace{0pt} = p\l(\mu, t+1, \hst{h}{t+1} \r) \frac{p\l(\mu, t+1, \hst{h}{t+1}, \at{t+1} \r)}{p\l(\mu, t+1, \hst{h}{t+1} \r)}\\
&\hspace{0pt} = p\l(\mu, t+1, \hst{h}{t+1}, \at{t+1} \r).
\end{align*}
\end{enumerate}
This completes the proof.
\end{proof}

\begin{cor}\label{cor:lbar_and_l}	
Fix $\lambda \in \mcl{Y}$. For any $\mu \in \uspacemix$, there exists $u = u(\mu) \in \uspace$ such that $\lags{u}{\lambda} = \lagsmix{\mu}{\lambda}$.	
\end{cor}	
\begin{proof}	
One notes that $\wh{C}(\mu)$ and $\wh{D}(\mu)$ can be written as:	
\begin{align*}	
\wh{C}(\mu) &= \sum_{t=1}^{\infty} \alpha^{t-1} \E{\mu}{P_1} \l[ \mbb{E}_{P_1} \l[ c\l( \Stt{t}, \At{t} \r)  \r] | \Hst{t}, \At{t} \r],\\	
\wh{D}(\mu) &= \sum_{t=1}^{\infty} \alpha^{t-1} \E{\mu}{P_1} \l[ \mbb{E}_{P_1} \l[ d\l( \Stt{t}, \At{t} \r)  \r] | \Hst{t}, \At{t} \r],	
\end{align*}	
and the result follows.	
\end{proof}

\begin{lem}\label{lem:puth}[Limit Probabilities for a converging sequence of policy-profiles]
Let $\l\{ \useq{i}{u} \r\}_{i=1}^{\infty}$ be a sequence in $\uspace$ that converges to $u$. Then, for any $t \in \mbb{N}$, $ \hst{h}{t} \in \hstspace{t} $, and $\at{t} \in \mcl{A}$,
\begin{align*}
\lim_{i\ra \infty}  \pruphsts{\useq{i}{u}}{t}{\hst{h}{t}, \at{t}} = \pruphsts{u}{t}{\hst{h}{t}, \at{t}},
\end{align*}
where $\pruphsts{\cdot}{t}{\hst{h}{t}, \at{t}} = \prup{\cdot}{P_1} \l( \Hst{t} = \hst{h}{t}, \At{t} = \at{t} \r)$. In other words, for every $t \in \mbb{N}$, the sequence of measures $\l\{ \pruphsts{ \useq{i}{u}}{t}{\cdot, \cdot} \r\}_{i=1}^{\infty}$ converges weakly to $\pruphsts{u}{t}{\cdot, \cdot}$.
\end{lem}
\begin{proof}
Given that $\useq{i}{u}$ converges to $u$, by the definition of convergence in product topology, for every $n \in [N]$, $\useq{i}{\utn{u}{t}{n}} (\hstn{h}{t}{0}, \hstn{h}{t}{n} )$ converges weakly to $ \utn{u}{t}{n} ( \hstn{h}{t}{0}, \hstn{h}{t}{n} ) $. Since $\mcl{A}^n$ is finite, this means that for each $\an{n}\in \anspace{n}$, $\useq{i}{\utn{u}{t}{n}} ( \an{n} | \hstn{h}{t}{0}, \hstn{h}{t}{n} )$ converges to $\utn{u}{t}{n} ( \an{n} | \hstn{h}{t}{0}, \hstn{h}{t}{n} )$, which further implies that for all $a \in \aspace$, $ \useq{i}{\ut{u}{t}} ( a | \hst{h}{t}) $ converges to $ \ut{u}{t} ( a | \hst{h}{t}) $. Now, we use forward induction to prove the statement. 
\begin{enumerate}
\item \textbf{Base Case}:  For time $t=1$, let $\ot{1} \in \hstspace{1} = \ospace$ and $\at{1} \in \mcl{A}$. We have
\begin{align*}
\pruphsts{\useq{i}{u}}{1}{\ot{1}, \at{1}}
=P_1\l( \sspace, o \r) \useq{i}{\ut{u}{1}} \l( \at{1} | \ot{1} \r) 
\ra \pruphsts{u}{1}{\ot{1}, \at{1}}.
\end{align*}

\item \textbf{Induction Step}: Assuming that the statement is true for time $t$, we show that it is true for time $t+1$. Let $\hst{h}{t+1} = \l( \ot{1:t+1}, \at{1:t} \r) = \l( \hst{h}{t}, \at{t}, \ot{t+1} \r) \in \hstspace{t+1}$ and $\at{t+1} \in \aspace$. We have
\begin{align*}
&\pruphsts{\useq{i}{u}}{t+1}{\hst{h}{t+1}, \at{t+1}} \\
&\hspace{0pt} =  
\pruphsts{\useq{i}{u}}{t}{\hst{h}{t}, \at{t}} \useq{i}{\ut{u}{t+1}} \l( \at{t+1} | \hst{h}{t+1} \r) \\
&\hspace{5pt} \times \pr_{P_1} \l( \Ot{t+1} = \ot{t+1} | \Hst{t} = \hst{h}{t}, \At{t} = \at{t} \r).
\end{align*}
By inductive hypothesis, $\pruphsts{\useq{i}{u}}{t}{\hst{h}{t}, \at{t}} $ converges to $\pruphsts{u}{t}{\hst{h}{t}, \at{t}}$, and $ \useq{i}{\ut{u}{t}}\l( \at{t+1} | \hst{h}{t+1}\r) $ converges to $ \ut{u}{t} \l( \at{t+1} | \hst{h}{t+1}\r) $ by assumption. We conclude that $\pruphsts{\useq{i}{u}}{t+1}{\hst{h}{t+1}, \at{t+1}}$ converges to $\pruphsts{u}{t+1}{\hst{h}{t+1}, \at{t+1}}$.
\end{enumerate}
This completes the proof.
\end{proof}

%% file: Sections/app_facts.tex
\section{Helpful Facts and Results}\label{sec:appendix:helpful_facts}
\begin{dfn}[Semi-continuity]\label{dfn:lsc}
A function $f : \mcl{X} \mapsto [-\infty, \infty]$ on a topological space $\mcl{X}$ is called \emph{lower semi-continuous} if for every point $x_0 \in \mcl{X}$,
\begin{align*}
\liminf\limits_{x\ra x_0} f(x) \ge f(x_0).
\end{align*}
We call $f$ as an upper semi-continuous function $-f$ is lower semi-continuous.
\end{dfn}

\begin{prop}[Monotone Convergence Theorem]\label{prop:mct}
    Let $\l(X, \mcl{M}, \mu \r)$ be a measure-space. Let $\l\{ f_i \r\}_{i=1}^{\infty}$ be an increasing sequence of measurable functions which are uniformly bounded-from-below. Then, 
    \begin{align*}
        &\int_{X} \lim_{i\ra\infty} f_i(x) \mu(dx) = \lim_{i\ra\infty} \int_{X} f_i(x) \mu(dx). 
    \end{align*}
\end{prop}

\begin{prop}[Tonneli's Theorem]\label{prop:tonneli}
    Let $f$ be a measurable function on the cartesian product of two $\sigma$-finite measure spaces $(X, \mcl{M}, \mu)$ and $(Y, \mcl{N}, \nu)$ which is bounded from below. Then, 
    \begin{align*}
        &\int_{X\times Y} f(x,y) (\mu \times \nu) (d(x,y))\\ 
        &\hspace{10pt} =\int_X \l( \int_Y f(x,y) \nu(dy)\r) \mu(dx)\\
        &\hspace{10pt} =\int_Y \l(\int_X f(x,y) \mu(dx)\r) \nu(dy).
    \end{align*}
\end{prop}

\begin{prop}[Fatou's Lemma]\label{prop:fatou}
    Let  $(X, \mcl{M}, \mu)$ be a measure-space and let $\{ f_i \}_{i=1}^{\infty}$ be a sequence of measurable functions which are uniformly bounded from below. Then,
    \begin{align*}
        & \liminf_{i\ra\infty} \int f_i(x) \mu (dx)\ge \int \liminf_{i\ra\infty} f_i(x) \mu(dx).
    \end{align*}
\end{prop}

\begin{prop}[Tychonoff's Theorem]\label{prop:tychonoff}
Product of countable number of compact spaces is compact under the product topology.
\end{prop}

\begin{prop}[Metrizability of Product Topology on Countable Product of Metric Spaces]\label{prop:metrizability}
Product of countable number of metric spaces, when endowed with the product topology, is metrizable.
\end{prop}


\begin{prop}[Prokhorov's Theorem]\label{prop:prokhorov}
Let $\l( \mcl{X}, \metric{\mcl{X}} \r)$ be a complete separable metric space with distance metric $\metric{\mcl{X}}$ and let $\borel{\mcl{X}}$ denote the Borel $\sigma$-algebra generated by $\metric{\mcl{X}}$. Let $\m{\mcl{X}}$ be the set of all probability measures on $\borel{\mcl{X}}$ and let $\tau$ denote the topology of weak-convergence on $\m{\mcl{X}}$. Then,  
\begin{enumerate}
    \item The topological space $ \l(\m{\mcl{X}} , \tau\r)$ is completely-metrizable. That is, there exists a complete metric $\metric{\m{\mcl{X}}}$ on $ \m{\mcl{X}}$ that induces the same topology as $ \tau $.
    \item An arbitrary collection $W \subseteq \m{\mcl{X}}$ of probability measures in $ \m{\mcl{X}}$ is tight iff its closure in $\tau $ is compact (i.e., $W$ is precompact in $\tau$).
\end{enumerate}
\end{prop}


\begin{prop}[Hyperplane Separation Theorem]\label{prop:separation_theorem}
Let $M$ be a non-empty convex subset of 
$\mbb{R}^n$. If $x_0 \in \mbb{R}^n$ does not belong to $M$, there exists $\rho \in \mbb{R}^n$ such that
\begin{align*}
\rho \neq 0 \text { and } \inf_{x \in M} \dotp{p}{x} \geq \dotp{p}{x_0}.
\end{align*}
\end{prop}

\begin{prop}[Integral of Bounded-from-Below function with respect to Convex Combination of Non-negative Measures]\label{prop:integral_linearity}
Let $\l(X, \mcl{M}\r)$ be a measure-space. Let $f : X \ra \mbb{R} \cup \{ \infty \}$ be a measurable function that is bounded from below, and let $\mu, \nu$ be two non-negative measures on $\mcl{M}$. Then, for any $\theta \in [0,1]$,
\begin{align*}
    &\int f(x) \l(\theta \mu + (1-\theta) \nu \r)(dx) \\
    &\hspace{10pt} = \theta \int f(x) \mu(dx) + (1-\theta) \int f(x)\nu(dx). 
\end{align*}
    
\end{prop}

\begin{prop}[Behavior of Integrals of a Bounded-from-Below and Lower Semi-Continuous Function]\label{prop:lsc}
Let $(\mcl{X}, \metric{\mcl{X}})$ be a complete separable metric space with distance metric $\metric{\mcl{X}}$ and let $\borel{\mcl{X}}$ denote the Borel $\sigma$-algebra generated by $\metric{\mcl{X}}$. Let $\l( \m{\mcl{X}} , \metric{\m{\mcl{X}}} \r)$ be the complete metric space of all probability measures on $\borel{\mcl{X}}$ with the topology of weak-convergence.\footnote{Prokhorov's theorem (see Proposition \ref{prop:prokhorov}) ensures completeness and metrizability of $\m{\mcl{X}}$.} Let $\mu \in \m{\mcl{X}}$ and let $f : \mcl{X} \ra \mbb{R} \cup \l\{ \infty\r\} $ be a function that is lower semi-continuous $\mu$-amost-everywhere\footnote{Lower semi-continuity of $f$ ensures that it is measurable.} and is bounded from below. Then, the function
\begin{align*}
H : \m{\mcl{X}} \mapsto \mbb{R} \cup \l\{ \infty \r\},\  
H(\mu') \defeq \int f(x) \mu'(dx)
\end{align*}
is lower semi-continuous at $\mu$. In particular, if $f$ is point-wise lower semi-continuous, then $H$ is also point-wise lower semi-continuous (on $\m{\mcl{X}}$).
\end{prop}
\begin{proof}
Define $f' : \mcl{X} \ra \mbb{R} \cup \{\infty \}$ as $f'(x) \defeq f(x) \wedge \liminf_{y\ra x} f(y)$. Then, $f'$ minorizes $f$\footnote{That is, $f'(x) \le f(x)$.}, is lower semi-continuous, and coincides with $f$ at $x$ if and only if $f$ is lower semi-continuous at $x$. Also, $f'$ is bounded from below (since $f$ is). By Proposition \ref{prop:lsc3}, $f'$ can be written as the point-wise limit of increasing sequence of uniformly bounded-from-below continuous functions from $\mcl{X}$ into $\mbb{R} \cup \{ \infty \}$, say $\l\{ g_i 
\r\}_{i=1}^{\infty} $, i.e., $f'(x) = \lim_{i\ra \infty} g_i(x)$. Then, for every $\mu' \in \m{\mcl{X}}$,
\begin{align*}
\int f'(x)\mu'(dx) = \int \lim_{i\ra \infty} g_i(x) \mu'(dx) = \lim_{i\ra \infty} \int g_i(x)\mu'(dx),
\end{align*}
where the last equality follows from the Montone Convergence Theorem (see Proposition \ref{prop:mct}). The above equality shows that the function $H' : \m{\mcl{X}} \ra \mbb{R} \cup \{ \infty\}$ such that $H'(\mu') = \int f'(x) \mu'(dx)$, is the point-wise limit of an increasing sequence of uniformly bounded-from-below continuous functions. Therefore, by Proposition \ref{prop:lsc3}, $H'$ is lower semi-continuous. Now, if $f$ is lower semi-continuous $\mu$-almost-everywhere, then $f = f'$ $\mu-$almost-everywhere. This gives,
\begin{align*}
H(\mu) &= \int f(x) \mu(dx) \\
&= \int f'(x) \mu(dx) \\
&\labelrel{=}{eqr:lsc:H2islsc} \liminf_{\mu'\ra\mu} H'(\mu') \\
&\labelrel{\le}{eqr:lsc:H2minorizesH} \liminf_{\mu'\ra\mu} H(\mu'),
\end{align*}
Here, \eqref{eqr:lsc:H2islsc} uses lower semi-continuity of $H'$ and \eqref{eqr:lsc:H2minorizesH} follows from the fact that $H'$ minorizes $H$ (since $f'$ minorizes $f$). The inequality $H(\mu) \le \liminf_{\mu'\ra\mu} H(\mu')$ is the definition of lower semi-continuity at $\mu$. 
\end{proof}

\begin{prop}[Equivalent Characterization of a Bounded-from-Below Lower Semi-Continuous Function]\label{prop:lsc3}
Let $\l( \mcl{X}, \metric{\mcl{X}} \r)$ be a metric space. Then, a function $f : \mcl{X} \ra \mbb{R} \cup \{ \infty \}$ is a bounded-from-below lower semi-continuous function if and only if it can be written as the point-wise limit of an increasing sequence of uniformly bounded-from-below continuous functions from $\mcl{X}$ into $\mbb{R} \cup \{ \infty \}$. 
\end{prop}
\begin{proof}
\textbf{Necessity}: Define $f_n : \mcl{X} \ra \mbb{R} \cup \{ \infty \}$ as follows:
\begin{align*}
f_n\l( x \r) &\defeq \inf_{y\in\mcl{X}} \l\{ f(y) + n \metric{\mcl{X}} \l(x, y\r) \r\}.
\end{align*}
\begin{enumerate}
\item \textit{Increasing}: 
\begin{align*}
f_{n+1}\l( x \r) = \inf_{y\in\mcl{X}} \l\{ f(y) + (n+1)\metric{\mcl{X}}\l( x,y\r) \r\} \ge f_n(x).
\end{align*}
\item \textit{Uniformly Bounded-from-Below}: Since $f_n\l(x \r) \ge \inf_{y\in\mcl{X}} \l\{ f(y) \r\}$ and $f$ is bounded-from-below, the functions $\l\{ f_n\r\}_{n=1}^{\infty}$ are uniformly bounded-from-below.
\item \textit{Continuity}: By triangle-inequality,
\begin{align*}
f(y) + n\metric{\mcl{X}}\l(y, z\r) \le
f(y) + n\metric{\mcl{X}}\l(y, w\r) +  n\metric{\mcl{X}}\l(w, z\r),
\end{align*}
and therefore, taking the infimum over $y$ on both sides gives $ f_n\l( z \r) - f_n\l( w \r) \le n\metric{\mcl{X}}\l( w, z\r) $. Similarly, we can get $ f_n\l( w \r) - f_n\l( z \r) \le n\metric{\mcl{X}}\l( w, z\r) $, and so
\begin{align*}
|f_n\l( z \r) - f_n\l( w \r)| \le n \metric{\mcl{X}} \l(w, z\r).
\end{align*}
The above relation shows that $f_n$ is Lipschitz and thus continuous.
\item \textit{Point-wise Convergence to $f$}: Fix $x_0 \in \mcl{X}$ and $\eps>0$. We would like to show that there exists a positive integer $n' = n'(x_0, \eps)$ such that, for all $ n \ge n'$, $| f_n\l(x_0\r) - f\l(x_0\r) | < \eps$. Since $f$ is lower semi-continuous at $x_0$, there exists $\delta = \delta(x_0, \eps) > 0$ such that
\begin{align*}
\metric{\mcl{X}}\l( x_0, y\r) < \delta \implies f(y) >  f(x_0) -\eps.\numberthis\label{eq:lsc2:implication}
\end{align*}
Since $f$ is bounded-from-below (and $\delta>0$), there exists a positive integer $n'=n'(\delta(x_0,\eps))$ such that
\begin{align*}
&\metric{\mcl{X}}\l( x_0, y\r) \ge \delta\\
&\hspace{0pt} \implies \forall\  n\ge n', f(y) + n\metric{\mcl{X}}\l( x_0, y\r) > f(x_0)\\
&\hspace{0pt} \implies \forall\  n\ge n', \inf_{\metric{\mcl{X}}\l( x_0, y\r) \ge \delta } \l\{ f(y) + \metric{\mcl{X}}(x_0, y) \r\}\ge f\l(x_0\r).
\end{align*}
So, for all $n\ge n'$, we have
\begin{align*}
f(x_0) \ge f_n\l( x_0 \r) &= \inf_{\metric{\mcl{X}}\l( x_0, y\r) \le \delta } \l\{ f(y) + n\metric{\mcl{X}}(x_0, y) \r\}\\
&\ge\inf_{\metric{\mcl{X}}\l( x_0, y\r) \le \delta } \l\{ f(y) \r\}\\
&\labelrel{>}{eqr:lsc2:1}\inf_{\metric{\mcl{X}}\l( x_0, y\r) \le \delta } \l\{ f(x_0) - \eps \r\}\\
&=f(x_0) - \eps.
\end{align*}
where \eqref{eqr:lsc2:1} uses \eqref{eq:lsc2:implication}.
\end{enumerate}
\hspace{5pt} \textbf{Sufficiency}: Let $\l\{ f_n \r\}_{n=1}^{\infty} $ be an increasing sequence of uniformly bounded-from-below continuous functions from $\mcl{X}$ into $\mbb{R} \cup \l\{ \infty \r\}$. Since the sequence is monotonic, it has a point-wise-limit $f : \mcl{X} \ra \mbb{R} \cup \l\{ \infty \r\}$ which is bounded-from-below because all the functions in the sequence are uniformly bounded-from-below. We need to show that $f$ is lower semi-continuous. 

Fix $x_0 \in \mcl{X}$ and $\eps>0$. We would like to show that there exists $\delta = \delta(x_0,\eps)>0$ such that $\metric{\mcl{X}}\l( x_0, y\r) < \delta \implies f(y) >  f(x_0) -\eps $. Since  $\l\{ f_n \r\}_{n=1}^{\infty} $ is increasing (and converges point-wise to $f$), there exists a positive integer $n'=n'(x_0, \eps)$ such that, for all $n\ge n'$, $f(x_0) \ge f_n(x_0) \ge f(x_0) - \frac{\eps}{2}$. Since $f_{n'}$ is lower semi-continuous, there exists $\delta=\delta(n'(x_0, \eps)) > 0$ such that $\metric{\mcl{X}}\l( x_0, y\r)<\delta \implies f(y) \ge f_{n'}(y) > f_{n'}(x_0) - \frac{\eps}{2} \ge f(x_0) - \eps$. 
\end{proof}

\begin{prop}[Banach Fixed-Point Theorem]\label{prop:banach}
    Let $\l(\mcl{X}, \metric{\mcl{X}}\r)$ be a (non-empty) complete metric space with a contraction mapping $T: \mcl{X} \ra \mcl{X}$. Then $T$ admits a unique fixed-point $x^\star$ in $\mcl{X}$ (i.e. $T\l(x^\star\r)=x^\star$ ). Furthermore, $x^\star$ can be found as follows: start with an arbitrary element $x_0 \in \mcl{X}$ and define a sequence $\l(x_i\r)_{i \in \mbb{N}}$ by $x_i = T\l(x_{i-1}\r)$ for $i \in \mbb{N}$. Then, $\lim _{i \ra \infty} x_i=x^\star$.
\end{prop}

%% file: Sections/minimax_theorem.tex
\section{A Minimax Theorem for Functions with Positive Infinity
}\label{sec:appendix:minimax}

\begin{prop}[A Minimax Theorem For Functions with Positive Infinity]\label{prop:sionminimax}
Let $\mcl{X}$ and $\mcl{Y}$ be convex topological spaces where $\mcl{X}$ is also compact. Consider a function $f : \mcl{X} \times \mcl{Y} \ra \mbb{R} \cup \{ \infty \} $ such that
\begin{enumerate}
\item for each $y \in \mcl{Y}$, $f\l(\cdot, y \r)$ is convex and lower semi-continuous.
\item for each $x \in \mcl{X}$, $f\l(x, \cdot \r)$ is concave.
\item If $f (x, y) = \infty$, then $f(x, y') = \infty$ for all $y'\in\mcl{Y}$.
\end{enumerate}
Then, there exists $x^\star \in \mcl{X}$ such that
\begin{align*}
\sup_{y\in \mcl{Y}} f\l( x^\star, y \r) &=
\inf_{x \in \mcl{X}} \sup_{y \in \mcl{Y}} f\l( x, y \r)\\
&=\sup_{y \in \mcl{Y}} \inf_{x \in \mcl{X}} f(x, y).
\end{align*}
\end{prop}

Proposition \ref{prop:sionminimax} is a mild adaptation of the Minimax theorem presented in \cite{aubin_book_2002}[Theorem 8.1] where a real-valued function is considered. In the MA-C-POMDP model described in Section \ref{sec:problem}, it is possible that $\fullccosts{u}$ and hence $\lags{u}{\lambda}$ is $\infty$ for all $\lambda \in \mcl{Y}$. We will use the same methodology as in \cite{aubin_book_2002}[Propositions 8.2 and 8.3] to prove Proposition \ref{prop:sionminimax}. In particular, the entire proof remains the same except that in Lemma \ref{lem:lem8.2:aubin:modified}, the compactness of $\mcl{X}$ is used together with Assumption 3). 

Define
\begin{align*}
f^{\sharp}(x) & :=\sup_{y \in \mcl{Y}} f(x, y), & & v^{\sharp}:=\inf_{x \in \mcl{X}} \sup_{y \in \mcl{Y}} f(x, y) \numberthis\\
f^b(y) & :=\inf_{x \in \mcl{X}} f(x, y), & & v^{\flat}:=\sup_{y \in \mcl{Y}} \inf_{x \in \mcl{X}} f(x, y).\numberthis
\end{align*}
To show the equality of $v^{\sharp}$ and $v^{\flat}$, we will introduce an intermediate value $v^{\natural}$ ($v$ natural) and prove successively that $v^{\natural}=v^{\sharp}$ and that $v^{\natural}=v^{\flat}$. 

We denote the family of finite subsets $J$ of $\mcl{Y}$ by $\mcl{J}$. We set
$$
v_J^{\sharp}:=\inf_{x \in \mcl{X}} \sup_{y \in J} f(x, y)
$$
and
$$
v^{\natural}:=\sup_{J \in \mcl{J}} v_J^{\sharp}=\sup_{J \in \mcl{J}} \inf_{x \in \mcl{X}} \sup_{y \in J} f(x, y).
$$
Since every point $y$ of $\mcl{Y}$ may be identified with the finite subset $\{y\} \in \mcl{J}$, we note that $v_{\{y\}}^{\sharp}=f^b(y)$ and consequently, $v^{\flat}=\sup_{y \in \mcl{Y}} v_{\{y\}}^{\sharp} \leq \sup_{J \in \mcl{J}} v_J^{\sharp} = v^{\natural}$. Also, since $\sup_{y \in J} f(x, y) \leq \sup_{y \in \mcl{Y}} f(x, y)$, we deduce that $v_J^{\sharp} \leq v^{\sharp}$, and hence $v^{\natural} \leq v^{\sharp}$. In summary, we have shown that
\begin{align*}
v^{\flat} \leq v^{\natural} \leq v^{\sharp} .
\end{align*}
Lemma \ref{lem:prop8.2:aubin} shows that $v^{\natural} = v^{\sharp} $ and Lemma \ref{lem:prop8.3:aubin} shows that $v^{\flat} = v^{\natural}$. This concludes the proof.

\begin{lem}\label{lem:prop8.2:aubin}
Consider a function $f : \mcl{X} \times \mcl{Y} \mapsto \mbb{R} \cup \{ \infty \} $ such that $\mcl{X}$ is compact and for each $
y \in \mcl{Y}$, $f(\cdot, y)$ is lower semi-continuous. Then, there exists $x^\star \in \mcl{X}$ such that
$$
\sup_{y \in \mcl{Y}} f(x^\star, y)=v^{\sharp}
$$
and
$$
v^{\natural}=v^{\sharp} .
$$
\end{lem}

\begin{rem}
Since the functions $f(\cdot, y)$ are lower semi-continuous, the same is true of the function $f^{\sharp}$.\footnote{Supremum of arbitrary collection of lower semi-continuous functions is lower semi-continuous.} Since $\mcl{X}$ is compact, Weierstrass's theorem implies the existence of $x^\star \in \mcl{X}$ which minimises $f^{\sharp}$. Following (3), this may be written as
\begin{align*}
&\sup_{y \in \mcl{Y}} f(x^\star, y) = f^{\sharp}(x^\star) = \inf_{x \in \mcl{X}} f^{\sharp}(x) \\
&\hspace{50pt} = \inf_{x \in \mcl{X}} \sup_{y \in \mcl{Y}} f(x, y)=v^{\sharp}.
\end{align*}
In comparison to this, Lemma \ref{lem:prop8.2:aubin} proves that $v^{\natural} = v^{\sharp} $.
\end{rem}

\begin{proof}
It suffices to show that there exists $x^\star \in \mcl{X}$ such that
\begin{align*}
\sup_{y \in \mcl{Y}} f(x^\star, y) \leq v^{\natural}.\numberthis\label{eq:vsharp<=vnatural}
\end{align*}
Since $v^{\sharp} \leq \sup_{y \in \mcl{Y}} f(x^\star, y)$ and $v^{\natural} \leq v^{\sharp}$, we shall deduce that $v^{\natural}=v^{\sharp}$.
We set
$$
S_{y}:=\l\{x \in \mcl{X} \mid f(x, y) \leq v^{\natural}\r\}.
$$
The inequality \eqref{eq:vsharp<=vnatural} is equivalent to the inclusion
\begin{align*}
x^\star \in \bigcap_{y \in \mcl{Y}} S_{y}.\numberthis\label{eq:nonemptyintersection}
\end{align*}
Thus, we must show that this intersection is non-empty.
For this, we shall prove that the $S_{y}$ are closed sets (inside the compact set $\mcl{X}$) with the finite-intersection property.\footnote{The intersection of an arbitrary collection of closed sets that lie inside a compact set and satisfy the finite-intersection property, is non-empty.}

If $v^{\natural} = \infty$, then every $S_y$ equals $\mcl{X}$ and the intersection is trivially non-empty. Therefore, WLOG, assume that $v^{\natural}$ is finite. Then the set $S_{y}$ is a lower section of the lower semi-continuous function $f(\cdot, y)$ and is thus closed.\footnote{The lower section of a lower semi-continuous function is closed. For every $\eta \in \mbb{R}$, the corresponding lower section is defined as $\{x \in \mcl{X} : f(x) \le \eta \}$.}

We show that for any finite sequence $J :=\l\{y_{1}, y_{2}, \ldots, y_{n}\r\} \in \mcl{J}$ of $\mcl{Y}$, the finite intersection
$$
\bigcap_{i \in [n]} S_{y_i} \neq \emptyset
$$
is non-empty. In fact, since $\mcl{X}$ is compact, and since $\max_{y \in J} f(\cdot, y) $ is lower semi-continuous, it follows that there exists $\hat{x} \in \mcl{X}$ which minimises this function. Such an $\hat{x} \in \mcl{X}$ satisfies
\begin{align*}
\max_{y \in J} f(\hat{x}, y) &= \inf_{x \in \mcl{X}} \max_{y \in J} f(x, y) \\
&\leq \sup_{J \in \mcl{J}} \inf_{x \in \mcl{X}} \max_{y \in J} f(x, y)=v^{\natural} .
\end{align*}
Since $\mcl{X}$ is compact, the intersection of the closed sets $S_{y}$ is non-empty and there exists $x^\star \in \mcl{X}$ satisfying \eqref{eq:nonemptyintersection} and thus \eqref{eq:vsharp<=vnatural}.
\end{proof}

\begin{lem}\label{lem:prop8.3:aubin}
Consider a function $f : \mcl{X} \times \mcl{Y} \mapsto \mbb{R} \cup \{ \infty \} $ such that $\mcl{X}$ and $\mcl{Y}$ are convex sets, (i) for each $y \in \mcl{Y}$, $f(\cdot, y)$ is convex, and (ii) for each $x \in \mcl{X}$, $f(x, \cdot)$ is concave. Then, $v^{\flat}=v^{\natural}$.
\end{lem}
\begin{proof}
We set $M_J:=\l\{\lambda \in \mbb{R}_{\ge 0}^{|J|} \mid \sum_{i=1}^n \lambda_i=1\r\}$. With any finite (ordered) subset $J \defeq = \l\{y_1, y_2, \ldots, y_n\r\}$, we associate the mapping $\phi_J$ from $\mcl{X}$ to $\l( \mbb{R} \cup \{ \infty \} \r)^{|J|}$ defined by
$$
\phi_J(x):=\l(f\l(x, y_1\r), \ldots, f\l(x, y_n\r)\r)
$$
We also set
$$
w_J:=\sup_{\lambda \in M_J} \inf_{x \in \mcl{X}} \dotp{\lambda}{\phi_J(x)}
$$
We prove successively that
\begin{enumerate}
    \item $\sup_{J\in\mcl{J}} w_J \leq v^{\flat}$ (Lemma \ref{lem:lem8.3:aubin}).
    \item $\sup_{J\in\mcl{J}} v^{\sharp}_{J}  \le \sup_{J\in\mcl{J}} w_J$ (Lemma \ref{lem:lem8.2:aubin:modified}).
\end{enumerate}
Hence, the inequalities
\begin{align*}
v^{\natural} = \sup_{J \in \mcl{J}} v_J^{\sharp} \leq \sup_{J \in \mcl{J}} w_J \leq v^{\flat} \leq v^{\natural}
\end{align*}
imply the desired equality 
 $v^{\flat}=v^{\natural}$.
\end{proof}

\begin{lem}\label{lem:lem8.3:aubin}
Consider a function $f : \mcl{X} \times \mcl{Y} \mapsto \mbb{R} \cup \{ \infty \} $ such that $\mcl{Y}$ is convex and for each $x \in \mcl{X}$, $f(x, \cdot)$ is concave. Then, for any finite subset $J$ of $\mcl{Y}$, we have $w_J \leq v^{\flat}$. Hence, $$\sup_{J\in\mcl{J}} w_J \le v^{\flat}.$$    
\end{lem}
\begin{proof}
With each $\lambda \in M_J$, we associate the point $y_\lambda:=\sum_{i=1}^n \lambda_i y_i$ which belongs to $\mcl{Y}$ since $\mcl{Y}$ is convex. The concavity of the functions $\l\{ f(x, \cdot)\r\}_{x\in\mcl{X}}$ implies that
\begin{align*}
\forall x \in \mcl{X}, \quad \sum_{i=1}^n \lambda_i f\l(x, y_i\r) \leq f\l(x, y_\lambda\r).
\end{align*}
Consequently,
\begin{align*}
\inf_{x \in \mcl{X}} \sum_{i=1}^n \lambda_i f\l(x, y_i\r) &\leq \inf_{x \in \mcl{X}} f\l(x, y_\lambda\r) \\
&\leq \sup_{y \in \mcl{Y}} \inf_{x \in \mcl{X}} f(x, y) \defeq v^{\flat}.
\end{align*}
The proof is completed by taking the supremum over $M_J$.
\end{proof}

\begin{lem}\label{lem:lem8.2:aubin:modified}
Consider a function $f : \mcl{X} \times \mcl{Y} \mapsto \mbb{R} \cup \{ \infty \} $ such that $\mcl{X}$ is a convex compact topological space, for each $y \in \mcl{Y}$, $f(\cdot, y)$ is convex and lower semi-continuous, and $f (x, y) = \infty$ implies $f(x, y') = \infty$ for all $y'\in\mcl{Y}$.
Then, 
\begin{align*}
v^{\natural} \defeq \sup_{J \in \mcl{J}} v_J^{\sharp} \leq \sup_{J \in \mcl{J}} w_J .
\end{align*}
\end{lem}
\begin{proof}
WLOG we assume that $\sup_{J \in \mcl{J}} w_J < \infty$. In this case, we can rewrite $w_J$ as $\supinf{\lambda\in M_J}{x\in \mcl{X}_J} \dotp{\lambda}{\phi_J(x)}$ where 
$$\mcl{X}_J \defeq \bigcap_{y\in J} dom f(\cdot, y).$$
To see this, note that $\dotp{\lambda}{\phi_J(x)} $ is a lower semi-continuous function on the compact space $\mcl{X}$. By Weierstrass theorem, $\dotp{\lambda}{\phi_J(x)} $ achieves its minimum in $\mcl{X}$ and we can write $w_J = \sup_{\lambda \in M_J} \dotp{\lambda}{\phi_J(\hat{x}(\lambda))}$. Suppose that $\hat{x}(\lambda) \in \mcl{X} \setminus \mcl{X}_J$, i.e., there exists $y \in J$ such that $\hat{x}(\lambda) \notin dom f(\cdot, y)$. This implies that $\hat{x}(\lambda) \notin dom f(\cdot, y')$ for all $y' \in J$. This renders $w_J$ to be infinity which contradicts our assumption $\sup_{J\in\mcl{J}} w_J < \infty $.

Therefore, now onward, we assume each $w_J = \sup_{\lambda\in M_J} \inf_{x \in \mcl{X}_J} \dotp{\lambda }{ \phi_J(x) }$. To prove the lemma, it suffices to show that $v_J^{\sharp} \le w_J $. Let $\eps>0$ and denote $\mbf{1} \defeq (1, \ldots, 1)$. We shall show that
\begin{align*}
\l( w_J + \eps \r) \mbf{1} \in \phi_J(\mcl{X}_J) + \mbb{R}_{\ge 0}^n .\numberthis\label{eq:wj_in_convex_set}
\end{align*}
Suppose that this is not the case. Since $\phi_J(\mcl{X}_J)+\mbb{R}_{\ge 0}^n$ is a convex set in $\mbb{R}^n $, following Lemma \ref{lem:lem8.1:aubin:modified}, we may use the hyperplane separation theorem (see Proposition \ref{prop:separation_theorem}), via which there exists $\rho \in \mbb{R}^n$, $\rho \neq 0$, such that
\begin{align*}
\sum_{i=1}^n \rho_i \l( w_J + \eps \r) &=\dotp{\rho}{\l(w_J + \eps\r) \mbf{1}} \\
&\leq \inf_{v \in \phi_J(\mcl{X}_J)+\mbb{R}_{\ge 0}^n} \dotp{\rho}{v}\\
& =\inf_{x \in \mcl{X}_J} \dotp{\rho}{ \phi_J(x)} + \inf_{u \in \mbb{R}_{\ge 0}^n} \dotp{\rho}{u}.
\end{align*}
Then $\inf_{u \in \mbb{R}_{\ge 0}^n} \dotp{\rho}{u}$ is bounded below and consequently, $\rho$ belongs to $\mbb{R}_{\ge 0}^n$ and $\inf_{u \in \mbb{R}_{\ge 0}^n} \dotp{\rho}{u}$ is equal to 0. Since $\rho$ is non-zero, $\sum_{i=1}^n \rho_i$ is strictly positive. We set $\bar{\lambda} =\rho / \sum_{i=1}^n \rho_i \in M_J$ and 
deduce that
\begin{align*}
w_J + \eps &\leq \inf_{x \in \mcl{X}_J } \dotp{\bar{\lambda}} {\phi_J(x)} \\
&\leq \sup_{\substack{\lambda \in M_J}} \inf_{x \in \mcl{X}_J} \dotp{\lambda}{ \phi_J(x) }= w_J.
\end{align*}
This is impossible and thus \eqref{eq:wj_in_convex_set} is established, which implies that there exist $x_{\eps} \in \mcl{X}_J$ and $u_{\eps} \in \mbb{R}_{\ge 0}^n$ such that $\l(w_J+\eps\r) \mathbf{1}=$ $\phi_J\l(x_{\eps}\r)+u_{\eps}$.
From the definition of $\phi_J$, we deduce that
\begin{align*}
\forall i=1, \ldots, n, \quad f\l(x_{\eps}, y_i\r) \leq w_J+\eps,
\end{align*}
and hence 
\begin{align*}
v_J^{\sharp} \leq \max _{i=1, \ldots, n} f\l(x_{\eps}, y_i\r) \leq w_J+\eps.
\end{align*}
We complete the proof of the lemma by letting $\eps$ tend to 0.   
\end{proof}

\begin{lem}\label{lem:lem8.1:aubin:modified}
Consider a function $f : \mcl{X} \times \mcl{Y} \mapsto \mbb{R} \cup \{ \infty \} $ such that $\mcl{X}$ is convex and for each $y \in \mcl{Y}$, $f(\cdot, y)$ is convex. Then, $\phi_J(\mcl{X}_J) + \mbb{R}_{\ge 0}^n$ is a convex set in $\mbb{R}^n$.    
\end{lem}
\begin{proof} 
Take any convex combination $\alpha_1 \l(\phi_J (x_1) + u_1 \r) + \alpha_2 \l(\phi_J(x_2) + u_2\r)$ where $\alpha_1, \alpha_2 \geq 0$, $\alpha_1 + \alpha_2 = 1$, $x_1$ and $x_2$ are in $\mcl{X}_J$, and $u_1$ and $u_2$ are in $\mbb{R}_{\ge 0}^n$. Let $x = \alpha_1 x_1 + \alpha_2 x_2$. For each $y \in J$, the function $f(\cdot, y)$ is convex, therefore $\phi_J(x) \le \alpha_1 \phi_J(x_1) + \alpha_2 \phi_J(x_2) < \infty$ (latter by definition of $\mathcal{X}_J$). Hence, $x \in \mcl{X}_J$. We can write the convex combination in the form $\phi_J (x) + u$ where $u = \alpha_1 u_1 + \alpha_2 u_2 + \alpha_1 \phi_J(x)+\alpha_2 \phi_J(y)-\phi_J(x)$. Note that $u \in \mbb{R}_{\ge 0}^{n}$ because $\phi_J(x) \le \alpha_1 \phi_J(x_1) + \alpha_2 \phi_J(x_2)$. Consequently, $\alpha_1\l(\phi_J\l(x\r)+u_1\r)+\alpha_2\l(\phi_J\l(y\r)+u_2\r)=\phi_J(x)+u$ belongs to $\phi_J(\mcl{X}_J)+\mbb{R}_{\ge 0}^n$.
\end{proof}

%% file: Sections/appendix_notation.tex
\section{List of Symbols}\label{sec:appendix:notation}
\begin{itemize}
\setlength\itemsep{2pt}
\item MDP: Markov Decision Process.
\item POMDP: Partially Observable Markov Decision Process.
\item SA-MDP: Single-Agent MDP.
\item SA-POMDP: Single-Agent POMDP.
\item SA-C-MDP: Single-Agent Constrained MDP.
\item SA-C-POMDP: Single-Agent Constrained POMDP.
\item MA-POMDP: Multi-Agent POMDP.
\item MA-C-POMDP: Multi-Agent Constrained POMDP.
\item MARL: Multi-Agent Reinforcement Learning.
\item CTDE: Centralized Training Distributed Execution.
\item ASPS: Approximate Sufficient Private State.
\item ASCS: Approximate Sufficient Common State.
\item RNN: Recurrent Neural Network.
\item FNN: Feed-forward Neural Network.
\item $N$: Number of agents.
\item $\sspace$: State space.
\item $\onspace{0}$: Space of common observations of all agents.
\item $\onspace{n}$: Space of private observations of agent-$n$.
\item $\ospace$: Joint-observation space, given by $\ospace = \prod_{n=0}^{N} \onspace{n}$.
\item $\anspace{n}$: Space of actions of agent $n$.
\item $\aspace$: Joint-action space, given by $\aspace = \prod_{n=1}^{N} \anspace{n}$.
\item $M_1(\cdot)$: Set of all probability measures on topological space $\cdot$ endowed with the topology of weak convergence. 
\item $\mcl{P}_{tr}$: Transition-law. See \eqref{eq:transitionlaw}.
\item $P_{saBo}$: Probability that the next state is in the measurable set $B$ and the next joint-observation is $o$ given action $a$ is taken at current state $s$. See \eqref{eq:psabo}.
\item $c, d$: Immediate-costs: $c$ is the immediate objective cost and $d$ is the immediate constraint cost.
\item $\udl{c}, \ov{c}, \udl{d}, \ov{d}$: Upper and lower bounds on immediate costs. See Assumption \ref{assmp:boundedcosts}.
\item $\alpha$: Discount factor.
\item $P_1$: Initial distribution on the initial state and joint-observation. See \eqref{eq:initialdistribution}.
\item $\uspacen{n}$: Space of policies of agent $n$.
\item $\un{u}{n}$: Used to denote a policy of agent $n$. (in $\uspacen{n}$).
\item $\uspace$: Space of (decentralized) policy-profiles, $\prod_{n=1}^{N} \uspacen{n}$.
\item $u$: Used to denote a policy-profile (in $\uspace$). 
\item $\uspacemix$: Space of (decentralized) mixtures of policy-profiles in $\uspace$, i.e., $\prod_{n=1}^{N} \m{\uspacen{n}}$.
\item $\mu$: Used to denote a typical element of $\uspacemix$,  given by $\mymathop{\times}_{n=1}^{N} \mun{n}$.
\item $\prup{u}{P_1}, \E{u}{P_1}$: Probability measure and expectation operator corresponding to policy-profile $u\in \uspace$ and initial-distribution $P_1$.
\item $C$: Infinite-horizon expected total discounted objective cost. See \eqref{eq:C}.
\item $D$: Infinite-horizon expected total constraint cost. See \eqref{eq:D}.
\item $\Hstn{t}{0}$: Common history of all agents at time $t$.
\item $\Hstn{t}{n}$: Private history of agent $n$ at time $t$.
\item $\Hst{t}$: Joint history at time $t$, given by $\Hstn{t}{0:N}$.
\item $\hstnspace{t}{n}, \hstspace{t}, \hsspace$: See \eqref{eq:hthnandh}.
\item $\optcosts$: Optimal solution of \eqref{eq:macpomdp}. See \eqref{eq:optccost:infsup}.
\item $\zeta$: Slack of feasible policy-profile $\ov{u}$ in Assumption \ref{assmp:slatercondition}. See \eqref{eq:slatercondition}.
\item $\mcl{Y}$: Space of non-negative Lagrange-multipliers, $\l\{ \lambda \in \mbb{R}^K: \lambda\ge 0\r\}$.
\item $L$: Lagrangian function for \eqref{eq:macpomdp}. See \eqref{eq:lagrangian}.
\item $\wh{L}$: Extended Lagrangian function for \eqref{eq:macpomdp}. See \eqref{eq:lagrangianmix}.
\item $\xuspace$: $\prod_{n=1}^{N}\xtspace{\uspacen{n}}$, also see \eqref{eq:xuspace}.
\item $\pruphsts{u}{t}{\hst{h}{t}, \at{t}}$: 
$\prup{u}{P_1}\l( \Hst{t} = \hst{h}{t}, \At{t} = \at{t} \r)$.
\item $\zuphsts{u}{t}{\hst{h}{t}, \at{t}}$: 
$\E{u}{P_1}\l[ \cCost \mid \Hst{t} = \hst{h}{t}, \At{t} = \at{t} \r] $.
\item $^i u$: $i^{th}$ policy-profile in the sequence $\l\{^i u\r\}_{i=1}^{\infty}$.
\item $\Gt{t} = \Gtn{t}{1:N}$: Coordinator's prescriptions at time $t$.
\item $\gtspace{t}$: Set of all possible prescriptions at time $t$.
\item $ \xtspace{\gtspace{t}}$: See \eqref{eq:xgtspace}.
\item $\wtHstn{t}{0}$: Prescription-observation history of coordinator.
\item $\un{l}{\lambda}$: Immediate-cost in unconstrained version of \eqref{eq:macpomdp} parametrized by Lagrange-multiplier $\lambda\in\mcl{Y}$. See \eqref{eq:l_lamda}.
\item $L_T$: See \eqref{eq:L_T}.
\item $\Qtl{t,T}{\lambda}$: See \eqref{eq:Q_tT_lamda}. 
\item $\Vtl{t,T}{\lambda}$: Optimal cost-to-go from time $t\in[T]$ onward in a finite-horizon $T$. See \eqref{eq:V_tT_lamda}.
\item $\utn{v}{1:T}{\lambda, \star}$: See \eqref{eq:v_tT_lamda_star}.
\item $\Vtl{t}{\lambda}$: See \eqref{eq:V_t_lamda}.
\item $\Qtl{t}{\lambda}$: See \eqref{eq:Q_t_lamda}.
\item $ \un{\udl{l}}{\lambda}$: Lower bound on $\un{l}{\lambda}$. See \eqref{eq:lower_and_upper_lcost}.
\item $\un{\ov{l}}{\lambda} $: Upper bound on $\un{l}{\lambda}$. See \eqref{eq:lower_and_upper_lcost}.
\item $ \Pi_t$: Conditional distribution of $\Stt{t}, \Hstn{t}{1:N}$ given $\wtHstn{t}{0}$, $\pr_{P_1} \l( \Stt{t}, \Hstn{t}{1:N} \mid \wtHstn{t}{0} \r)$.
\item $\Zhtn{t}{1:N}$: ASPS at time $t$. See Definitions \ref{dfn:asps_generator_finite_horizon} and \ref{dfn:asps_generator_infinite_horizon}. 
\item $\Lamdaht{t}$: ASPS-based prescription at time $t$.
\item $\Zhtn{t}{0}$: ASCS at time $t$. See Definitions \ref{dfn:ascs_generator_finite_horizon} and \ref{dfn:ascs_generator_infinite_horizon}.
\item $\varthetahtn{t}{1:N}$: $t$-th component of (finite-horizon) ASPS-generator. See Definition \ref{dfn:asps_generator_finite_horizon}.
\item $\eps_{p,1}, \eps_{p,2}, \delta_{p}$: Attributes of ASPS-generator.
\item $\phihtn{t}{1:N}$: $t$-th component of evolution functions of (finite-horizon) ASPS-generator. See ASPS-1.
\item $\kappa \l(\cdot, \star \r)$: Total variation distance between probability measures $\cdot$ and $\star$. 
\item $\lamdahtnspace{t}{n}$: Set of all possible ASPS-ASCS based prescriptions for agent $n$ at time $t$
\item $\lamdahtspace{t}$. Set of all possible ASPS-ASCS based prescriptions at time $t$, $\prod_{n=1}^{N}\lamdahtnspace{t}{n} $.
\item $\xtspace{\lamdahtspace{t}}$: See \eqref{eq:xlamdahtspace}.
\item $\varthetahtn{t}{0}$: $t$-th component of (finite-horizon) ASCS-generator.
\item $\eps_{c,1}, \eps_{c,2}, \delta_{c}$: Attributes of ASCS-generator.
\item $\phihtn{t}{0}$: $t$-th component of evolution functions of (finite-horizon) ASCS-generator. See ASCS-1.
\item $\whVtl{t,T}{\lambda}$ See \eqref{eq:whV_tT_lamda}.
\item $\whQtl{t,T}{\lambda}$: See \eqref{eq:whQ_tT_lamda}. 
\item $\utn{\wh{v}}{1:T}{\lambda,\star}$: See \eqref{eq:whv_tT_lamda_star}.
\item $M_c\l( \cdot; \alpha, T \r)$: See \eqref{eq:M_c}.
\item $M_p\l( \cdot; \alpha, T\r)$: See \eqref{eq:M_p}.
\item $N\l( \alpha, T\r)$: See \eqref{eq:N_alpha_T}.
\item $\varthetahn{1:N}$: Infinite-horizon ASPS-generator. See Definition \ref{dfn:asps_generator_infinite_horizon}.
\item $\varthetahn{0}$: Infinite-horizon ASCS-generator. See Definition \ref{dfn:ascs_generator_infinite_horizon}.
\item $\lamdahspace$: Set of all possible ASPS-ASCS based prescriptions when infinite-horizon ASPS and ASCS generators are used.
\item $\xtspace{\lamdahspace}$: See \eqref{eq:xlamdahspace}.
\item $\wh{B}$: See \eqref{eq:whB_whV}.
\item $\whVl{\lambda}, \whQl{\lambda}$: See \eqref{eq:whVl} and \eqref{eq:whQl}.
\item $\rhon{0} $: RNN to serve as (infinite-horizon) ASCS-generator.
\item $\rhon{1:N}$: RNNs to collectively serve as (infinite-horizon) ASPS-generator.
\item $\varphin{0}$: Coordinator's prescription network.
\item $\varphin{1:N}$: Prescription-applier networks of all agents.
\item $\psin{0}$: Coordinator's prediction network.
\item $\psin{S}$: Supervisor's prediction network.
\item $\nabla_{\varphi}L_{\infty}\l( \ut{\wh{v}}{\rho, \varphi} \r)$: Policy-gradient. See \eqref{eq:policy_gradient}.
\item $\delta_{1,i}, \delta_{2,i}$, $\delta_{3,i}$: Sequences of time-steps that satisfy three time-scale stochastic approximation conditions. See \eqref{eq:stepsizes}.
\item $l_2, l_{c,3}, l_{p,3}$: See \eqref{eq:l2}, \eqref{eq:lc3}, and \eqref{eq:lp3}.
\item $\eta$: Used as a placeholder in \eqref{eq:lc3} and \eqref{eq:lp3}.
\item $B$: Batch-size of trajectories.
\item $\tau_j$: $j^{th}$ trajectory. See \eqref{eq:tauj}.
\item $l_{\rhon{0}, \psin{0}}$: See \eqref{eq:l_rho0_psi0}.
\item $l_{\rho, \psin{S}}$: See \eqref{eq:l_rho_psiS}.
\item $R_{\rhon{0}, \psin{0}}$: See \eqref{eq:R_rho0_psi0}.
\item $R_{\rho, \psin{S}}$: See \eqref{eq:R_rho_psiS}.
\item $g_{j,t}$: Cost-to-go at time $t$ in $j^{th}$ trajectory. See \eqref{eq:gjt}.
\item $\wh{\nabla_{\varphi}L_{\infty} }\l( \ut{\wh{v}}{\rho, \varphi} \r)$: REINFORCE-estimate of policy-gradient $\nabla_{\varphi} L_{\infty} \l( \ut{\wh{v}}{\rho,\varphi} \r)$. See \eqref{eq:pg_estimate_reinforce}.

\end{itemize}